\documentclass[11pt,reqno]{amsart}
\usepackage{amsmath}
\usepackage{amsfonts}

\numberwithin{equation}{section}
\usepackage{times}
\usepackage{amsmath,amsfonts,amstext,amssymb,amsbsy,amsopn,amsthm,eucal}
\usepackage{txfonts}
\usepackage{dsfont}
\usepackage{graphicx}   % for figures
\usepackage{hyperref}
\usepackage{accents}
\usepackage{enumerate}
\usepackage{xcolor}

\usepackage{verbatim}
\definecolor{mycustompurple}{RGB}{154, 36, 79} % 定义自己的颜色
\usepackage[utf8]{inputenc}
\usepackage{hyperref} % 引入超链接宏包
\hypersetup{
    colorlinks=true,            % 激活链接颜色，去掉链接边框
    linkcolor=red,              % 文档内部链接颜色（如图表等引用）
    citecolor=green,            % 文献引用链接颜色
    filecolor=mycustompurple,   % 文件链接颜色
    urlcolor=magenta            % 外部URL链接颜色
}

\usepackage{fancyhdr}

\fancypagestyle{mystyle}{%
    \fancyhf{} % 清空页眉和页脚的设置
    \fancyhead[LE,RO]{\thepage} % 在偶数页左侧和奇数页右侧显示页码
}
\newcommand{\Haus}{\mathcal{H}}

\newcommand{\Lip}{\text{Lip}}

\newcommand{\supp}{\operatorname{supp}}

\newcommand{\R}{{\rm R}}
\newcommand{\Rm}{{\rm Rm}}
\newcommand{\Ric}{{\rm Ric}}

\newcommand{\Vol}{{\rm Vol}}

\newcommand{\rv}{{\rm v}}

\newcommand{\vepsilon}{\varepsilon}

\newcommand{\N}{\mathcal{N}}
\newcommand{\C}{\mathcal{C}}

\newcommand{\cC}{\mathcal{C}}

\newcommand{\cE}{\mathcal{E}}

\newcommand{\cH}{\mathcal{H}}

\newcommand{\cN}{\mathcal{N}}

\newcommand{\cV}{\mathcal{V}}

\newtheorem{theorem}{Theorem}[section]

\newtheorem{proposition}[theorem]{Proposition}
\newtheorem{claim}[theorem]{Claim}
\newtheorem{lemma}[theorem]{Lemma}

\theoremstyle{definition}
\newtheorem{definition}[theorem]{Definition}
\theoremstyle{remark}
\newtheorem{remark}{Remark}[section]
\theoremstyle{remark}

\theoremstyle{remark}

\theoremstyle{remark}
\theoremstyle{remark}

\begin{document}
\pagestyle{mystyle}
  
\title{$L^2$ Curvature Bounds on Ricci flows with bounded scalar curvature up to dimension 7}
\date{}
\author{Xianbin Huang}
\address{X. Huang,~~School of mathematical Sciences,  Zhejiang University, Hangzhou, China}
\email{2592363260@foxmail.com}
\author{Wancheng Zhang}
\address{W. Zhang,~~School of mathematical Sciences,  Zhejiang University, Hangzhou, China}
\email{12335021@zju.edu.cn}
 
\begin{abstract}
In this paper, we study Ricci flows with uniformly bounded scalar curvature. On such Ricci
flows we establish uniform $L^2$ Riemannian
curvature bounds and $L^4$ Ricci
curvature bounds up to dimension 7. These estimates improves   $L^{2-\vepsilon}$ curvature bounds established by Bamler \cite{B18} and $L^2$ curvature bounds in dimension 4 established by Bamler and Zhang \cite{BZ15}.

More generally, we establish   uniform $L^2$ curvature bounds on $n$-manifolds satisfying an integral $L^p$ Ricci curvature bound for $p> n/2$, a uniform volume lower bound, and a strong $\vepsilon$-regularity assumption. This result extends the celebrated $L^2$ curvature estimates of Jiang and Naber \cite{JN}, which were established under pointwise two-sided Ricci curvature bounds. To overcome the lack of pointwise control, we develop an $L^p$ adaptation of the neck decomposition theorem and establish a  superconvexity estimate for the curvature on neck regions. Furthermore, we demonstrate that our structural assumptions are dynamically natural on Ricci flows with bounded scalar curvature. As an application, by combining our estimate with the a priori $L^{4-\vepsilon}$ curvature radius bound of Ricci flow established by Bamler, we obtain uniform $L^2$ curvature bounds on Ricci flows  with bounded scalar curvature of dimension $n \le 7$.

\end{abstract}
\maketitle
\tableofcontents
\section{Introduction}

Understanding the deep relationship between bounds on the Ricci curvature and the control of the full Riemannian curvature tensor is a central theme in geometric analysis. A cornerstone result in this area is the $L^2$ curvature estimate proved by Jiang and Naber \cite{JN}, which resolved the $L^2$ curvature conjecture. They demonstrated that for a Riemannian manifold $(M^n, g)$, a pointwise two-sided bound on the Ricci curvature $|\Ric| \le n-1$ coupled with a uniform volume non-collapsing condition $\Vol(B_1(x)) \ge \rv > 0$ enforces a  $L^2$ bound on the full Riemann curvature tensor: $\int_{B_1(x)} |\Rm|^2 \le C(n, \rv)$, where the upper bound  $C(n, \rv)$ depends only on the dimension $n$ and the non-collapsing constant $\rv$.

The analogous question for Ricci flows—whether uniform  $L^2$ curvature estimates hold under bounded scalar curvature and suitable non-collapsing conditions—has been investigated by Simon, Bamler and Zhang. Simon \cite{S15} and Bamler-Zhang \cite{BZ15} independently  established a uniform  $L^2$ curvature bound in dimension 4, and this result has been particularly useful in the study of Kähler Ricci flows on Fano manifolds. For general dimensions, Bamler developed a comprehensive compactness and partial regularity theory for Ricci flows with bounded scalar curvature and entropy lower bound \cite{B18}. In this framework, Bamler proved a uniform $L^{2-\vepsilon}$ curvature estimate in arbitrary dimensions.
This naturally raises the following question:
\textit{Does a uniform  $L^2$
  curvature estimate hold for Ricci flows with bounded scalar curvature in all dimensions?}

Bamler's $L^{2-\vepsilon}$ estimate represents the best possible result using the current methods, and new ideas appear necessary to bridge the gap between $2-\vepsilon$ and 2. One of the main results in this paper is to give an affirmative answer to above question when the dimension is at most 7:
\begin{theorem}
Let \( (M^n, (g_t)_{t \in [-2,0]}) \) be a compact Ricci flow  of dimension $n\le 7$. We further assume

\begin{enumerate} \item[$(1)$] The scalar curvature satisfies the uniform upper bound  
\[ |\R|_{g(t)}\le A \quad \text{on} \quad M \times [-2, 0]. \]

\item[$(2)$]  Perelman's entropy satisfies the uniform lower bound  
\[ \nu[g(-2), 4] = \inf_{0 < \tau \leq 4} \mu[g(-2), \tau]\ge -A. \]
\end{enumerate}
Then there exists $C(n,A)<\infty$ such that for any \( t \in [-1, 0] \) and \( x \in M \), we have 
\begin{enumerate} \item[$(a)$]  \textbf{$L^4$ Ricci curvature bound:}
\[\int_{B_{g(t)}(x,1)} |\Ric|_{g(t)}^{4} \, d\Vol_{(g(t)} \leq  C(n,A);\]
 \item[$(b)$] \textbf{$L^2$ Riemannian curvature bound:}
\[\int_{B_{g(t)}(x,1)} |\Rm|_{g(t)}^{2} \, d\Vol_{(g(t)} \leq  C(n,A).\]
\end{enumerate} 
\end{theorem}
It should be emphasized that the restriction to low dimensions in our theorem is solely of a technical nature, due to the loss  integrability $4>p>n/2$. The proof of $L^4$ Ricci curvature bound  follows almost from the same argument of $L^2$  Riemannian  curvature bound, so we will mainly focus on the $L^2$ estimate in this paper. Indeed, we prove an 
$L^2$ curvature estimate in general dimensions under an  $L^p$ Ricci curvature assumption with  
$p>n/2$. Combining this estimate with Bamler’s a priori $L^{4-\vepsilon}$ curvature radius bound for Ricci flow yields, as a corollary, a uniform  $L^2$
  curvature bound for Ricci flows with bounded scalar curvature in dimensions  $n\le 7$.

Therefore, the central concern of this paper is to establish  $L^2$ curvature estimates under $L^p$ Ricci curvature assumptions. Our main idea is directly inspired by the work of Jiang and Naber; however, their original framework assumes pointwise bounds on Ricci curvature, whereas our goal is to work under much weaker integral assumptions.  In many situations—particularly in the context of Ricci flow—uniform pointwise bounds on the Ricci curvature are too restrictive, and one typically can only obtain integral bounds via entropy or potential estimates. Consequently, the primary goal of this paper is to extend the Jiang–Naber  $L^2$ curvature estimates to the significantly more flexible framework of integral  $L^p$ Ricci curvature bounds.

  To make the overall framework more clear, we make the following assumptions on Riemannian manifold $(M^n, g)$ with a fixed base  point $q$:  
\begin{enumerate}
    \item[(A1)] \textbf{$L^p$ Ricci bound:} There exists  $p > \frac{n}{2}$ such that
\begin{equation}\label{Lp}
        \int_{B_4(q)} |\Ric|^p \le \Lambda.
\end{equation}

    \item[(A2)] \textbf{Uniform non-collapsing:} There exists $\rv>0$ such that for any  $0<r\le 1$,
\begin{equation}\label{noncollapsing}   \Vol(B_r(q)) \ge \rv r^n.
\end{equation}
 \item[(A3)] \textbf{Codimension four of singularity and strong $\vepsilon$-regularity:} For any $\vepsilon>0$, there exists $\delta>0$ such that for any $B_{2r}(x)\subset M$ with $d_{GH}(B_{2r}(x),B_{2r}(0^{n-3},z))< \delta r$, where $(0^{n-3},z)$ is the  tip of  some metric cone $\mathbb{R}^{n-3}\times C(Z^2)$, we have $d_{GH}(B_{r}(x),B_{r}(0^{n}))< \vepsilon r$. 
 
 Furthermore, in this case,  there exists $B>0$ such that 
\begin{equation}\label{Strong epsilon-regularity}
\begin{aligned}
        &r^2\sup_{B_r(x)}|\Rm|+r\sup_{B_r(x)}|\Ric|<   \vepsilon, \\
 r\sup_{B_{r/2}(x)}|\nabla \Rm|\le B\sup_{B_{r}(x)}|\Rm|, &\quad \sup_{B_{r/2}(x)}r|\nabla \Ric|+\sup_{B_{r/2}(x)}r^2|\nabla^2 \Ric|\le  B \sup_{B_{r}(x)}|\Ric|.
\end{aligned}
\end{equation}
  \end{enumerate}
We remark here that assumption (A3) holds automatically in some situations, in particular for Ricci flows with bounded scalar curvature. We can now state our main estimate as follows:
\begin{theorem} \label{thm:main}
Let $(M^n,g)$  satisfy assumptions $(A1), (A2), (A3)$ above, then there exists a constant \(C = C(n,p,\rv,\Lambda,B)<\infty\) such that  
\[
\int_{B_1(q)} |\Rm|^2 \le C .
\]
\end{theorem}

\subsubsection{Outline of this paper}
In section 2, we review some basic background and preliminary material of spaces with integral Ricci curvature bound  and Ricci flows with bounded scalar curvature. In section 3, we introduce the definition and concept of neck regions, they are open sets that look like the singular space $\mathbb{R}^{n-4}\times C(S^3/\Gamma)$ on many scales. Although the precise definition of a neck region is a bit technical, it is important for us to establish more refined estimates on the neck regions. In section 4, we prove the $L^2$
curvature bound on neck regions, under an additional assumption that we have already proved some Ahlfors regularity bound. The proof of this section follows from a powerful superconvexity estimate originated from \cite{JN}. In section 5, we study the better behavior of splitting functions on neck regions, and use these good estimates to prove the Ahlfors regularity bound. Thus finish the proof of  $L^2$
curvature bound on neck regions. In section 6, we inductively decompose our manifolds $M$
into two basic type of pieces, namely,  neck regions and $\vepsilon$-regularity regions with uniformly
bounded curvature. The proof will require multiple covering steps and lemmas, where we will first decompose $B_1(q)$ into a much larger collection of balls, and then proceed to recover those
balls that are not either neck regions or smooth regions. In the last section, we combine the $L^2$ estimates on neck region 
with the neck decomposition theorem  in order to finish the proof of $L^2$ estimates on $M$.
\subsubsection{Declaration on the Use of AI}
No artificial intelligence (AI) tools were used during the preparation of this manuscript. In fact, the main results of this work had already been completed by the end of the last year.
\section{Preliminaries}
Let us review some basic computations and regularity theory of spaces with integral Ricci curvature bound. For the most part, we work under assumptions (A1) and (A2) throughout this section. At the end, we review some recent progress on Ricci flows with bounded scalar curvature, due to the work of Bamler and Zhang \cite{BZ15},\cite{B18}. 
\subsection{Basic computations on manifolds}
\begin{proposition}\label{laplacian of Rm}
On a Riemannian manifold $(M,g)$, the Laplacian of the curvature tensor $\Rm$ satisfies 
\[\Delta\Rm=\nabla^2 \Ric+\Rm*\Rm+\Rm*\Ric.
\] More precisely,
\begin{align*} \Delta \Rm_{ijkl}&=\nabla_i\nabla_k \Ric_{jl}-\nabla_j\nabla_k\Ric_{il}+\nabla_j\nabla_l\Ric_{ik}-\nabla_i\nabla_l\Ric_{jk}\\
&\quad -2(B_{ijkl}-B_{ijlk}-B_{iljk}+B_{ikjl})+g^{pq}(\Rm_{qjkl}\Ric_{pi}+\Rm_{iqkl}\Ric_{pj}),
\end{align*}
where $B_{ijkl}= g^{pq}g^{rs}\Rm_{ipjr}\Rm_{kqls}$.
 In particular,  
\[\frac{1}{2}\Delta |\Rm|^2= |\nabla \Rm|^2+\langle\nabla^2 \Ric,\Rm\rangle+\langle\Rm*\Rm,\Rm\rangle+\langle\Rm*\Ric,\Rm\rangle.\]
\end{proposition}
The following improved Kato inequality will play an important role in the computation of superconvexity estimate Theorem \ref{superconvexity}. A detialed proof of this inequality can be found in section 9 of \cite{JN}.

\begin{proposition}[Improved Kato inequality]\label{Improved Kato inequality} Let $(M^n,g)$ be a Riemannian manifold, there exist $\varepsilon(n)>0$ and $C(n)>0$ such that
 \[ (1+\varepsilon(n))|\nabla |\Rm||^2\le |\nabla\Rm|^2+C(n)|\nabla \Ric|^2.
 \]
 \end{proposition}
\subsection{Regularity  of spaces with integral Ricci curvature bound}
Since Riemannian metric satisfies  elliptic equation $\Delta {g_{ij}}+Q_{ij}(g,\partial g)=-2\Ric_{ij}$ in harmonic coordinates, harmonic radius plays an important role in the regularity theory. Let us briefly recall the definition of harmonic radius.
\begin{definition}[Harmonic radius]\label{harmonic radius}
Fix \( \alpha \in (0, 1) \). For \( x \in M \), we define the \( C^\alpha \) harmonic radius at \( x \), denoted by \( r^\alpha_h(x)\) (or simply \(r_h(x) \) when there is no confusion),  to be the maximal radius \( r>0 \) such that there exist harmonic coordinates \( x = (x^1, \dots, x^n) : B_r(x) \to \mathbb{R}^n \) satisfying
\[
\sup_{i,j} \bigl( \|g_{ij}-\delta_{ij}\|_{C^0} + r^\alpha \|g_{ij}\|_{C^\alpha} \bigr) \leq 10^{-n}.  
\]
Similarly, we define $W^{2,p}$ harmonic radius by requiring the scaling invariant $W^{2,p}$ bound
\[\sup_{i,j} \bigl( \|g_{ij}-\delta_{ij}\|_{C^0} +\sum_{k=1}^2 r^{k-\frac{n}{p}} \|\partial_k g_{ij} \|_{L^p} \bigr) \leq 10^{-n}.
\]
 \end{definition}  In harmonic coordinates, a priori \( L^{p>n/2} \)-bound of the Ricci curvature gives the uniform \( W^{2,p} \)-bound of the metric tensor \( g_{ij} \), which in turn implies the uniform \( C^{2-n/p} \) bound of the metric by Sobolev embedding theorem. Throughout the remainder of this paper, unless stated otherwise,  $r_h$  refers to the $W^{2,p}$ harmonic radius.
 
  Similar to the case of bounded Ricci curvature, we have the following $\vepsilon$-regularity theorem.
\begin{theorem}[\cite{TZ}, Theorem 2.35]
\label{harmonic radius}
Let $(M^n,g)$  satisfy $(A1)$. There exist \( \vepsilon(n,p,\Lambda) > 0 \) and \( r_0(n,p,\Lambda) > 0 \) such that if   $
\Vol(B_r(q)) \geq (1 - \vepsilon) \omega_nr^n$ $($or equivalently, $d_{GH}( B_r(q),B_r(0^n))<\vepsilon )$, then 
$r_h^{2-n/p}(q) \geq  r_0.$
\end{theorem} 

We also have the following volume comparison theorem  thanks to the work of \cite{PW}, which plays a crucial way to study local behavior of manifolds with   integral Ricci curvature bound.
    \begin{theorem}[Relative volume comparison theorem]\label{Volume comparison}
 Let $(M^n,g)$   satisfy  $\displaystyle R^2\fint_{B_{R}(x)}|\Ric|^{p}<\delta^p$ for some $p>n/2$. Then for $0<r\le R$,\[
 \left( \frac{\Vol(B_{r}(x))}{\Vol(B_{R}(x))} \right)  \geq \left( \frac{r}{R} \right)^n \left( 1 - C(n, p) \delta^{\frac{1}{2}}  \right)^{2p}.\]
 In particular, if $R^2\fint_{B_{R}(x)}|\Ric|^{p} \le \delta(n,p)$, then
 \[
 \left( \frac{\Vol(B_{r}(x))}{\Vol(B_{R}(x))} \right)  \geq \frac{1}{2}\left( \frac{r}{R} \right)^n  .\]
\end{theorem}
An immediate consequence of volume comparison theorem is the following precompactness theorem for manifolds with $L^p$ Ricci curvature bound.

\begin{theorem}[Precompactness under integral curvature bound]\label{precompactness}
Let $(M_i^n,g_i,q_i)$ be a sequence of complete manifolds satisfying $(A1)$, then there exists a subsequence   $(M_j^n,g_j,q_j)$ converges to a limit metric space $(X,d,q)$ in pointed Gromov-Hausdorff topology.
\end{theorem}

 To study the structure of limit space, let us briefly review the notion of symmetry of a metric space.
\begin{definition} Given a metric space \(Y\), we define the following:
\begin{enumerate}
    \item We say \(Y\) is \(k\)-symmetric at \(y\in Y\) if there exists a pointed isometry 
    \(\iota : \mathbb{R}^k \times C(Z) \to Y\) with \(\iota(0^k, z_c) = y\), where \(Z\) is a compact metric space and \(z_c\) is a cone point.
    \item We say \(Y\) is \(k\)-symmetric with respect to \(L^k \subseteq Y\) if \(L^k = \iota(\mathbb{R}^k \times \{z_c\})\).
\end{enumerate}
\end{definition}
 
This notion of symmetry leads to a natural quantitative generalization, first introduced in \cite{CN13} for spaces with lower Ricci curvature bound and, generalize to spaces with integral Ricci curvature bound later by \cite{Q}.  It is the notion of quantitative symmetry which will play the most important role for us in this paper, in particular in the discussion of neck regions. Let us begin with a discussion of quantitative symmetry:

\begin{definition} Given a metric space \(Y\) with \(y\in Y\), \(r>0\) and \(\epsilon>0\), we say
\begin{enumerate}
    \item \(B_r(y)\) is \((k,\epsilon)\)-symmetric if there exists a pointed \(\epsilon r\)-GH map 
    \(\iota : B_r(0^k, z_c) \subseteq \mathbb{R}^k \times C(Z) \to B_r(y) \subseteq Y\) with \(\iota(0^k, z_c) = y\);
    \item \(B_r(y)\) is \((k,\epsilon)\)-symmetric with respect to \(\mathcal{L}_\epsilon^k \subseteq B_r(y)\) if 
    \(\mathcal{L}_\epsilon^k \equiv \iota(\mathbb{R}^k \times \{z_c\}) \cap B_r(y)\).
\end{enumerate}
\end{definition}
 
To state the definition in words, we say that $B_r(y) \subset Y$ is $(k, \epsilon)$-symmetric if the ball $B_r(y)$ looks very close to having $k$-symmetries. The quantitative stratification is then defined as follows: 
\begin{definition} For each \(\epsilon > 0\), \(0 < r < 1\) and \(k \in \mathbb{N}\), define the closed quantitative \(k\)-stratum  \(S^k_{\epsilon,r}(X)\), by  
\[
S^k_{\epsilon,r}(X) \equiv \bigl\{ x \in X : \text{for no } r \leq s \leq 1 \text{ is } B_s(x) \text{ a } (k+1,\epsilon)\text{-symmetric ball} \bigr\}.\] 
\end{definition}
  In order to study the almost rigidity structure under Ricci curvature (lower, two-sided, integral) bound,  let us recall the  definition of splitting map introduced by \cite{C01}.
\begin{definition}[$\vepsilon $-splitting map] A $(k,\vepsilon)$-splitting map on $B_r(x)$ is a harmonic map $u=(u_1,u_2,\cdots,u_k):B_{r}(x)\to \mathbb{R}^k$ such that 
\begin{enumerate}
\item $\displaystyle\sup_{B_r(x)}|\nabla u_j|\le C(n)$,\text{\quad for all $1\le j\le k$};
\item $\displaystyle\sum_{j,l=1}^k\fint_{B_{r}(x)} \left|\langle\nabla u_j,\nabla u_l\rangle -\delta_{jl}\right| <\vepsilon$;
\item $\displaystyle\sum_{j=1}^k r^2\fint_{B_{r}(x)} |\nabla^2 u_j|^2<\vepsilon$.
\end{enumerate}\end{definition}
The following almost rigidity theorems  for manifolds with integral bounded Ricci curvature are well-known. We refer to \cite{PW},\cite{TZ},\cite{C21} for a detailed proof, see also \cite{Q} for the study of singular set of spaces with integral Ricci curvature bound.

\begin{theorem}[Almost splitting theorem]\label{almost splitting} Let $(M^n,g)$   satisfy  $\displaystyle\int_{B_{2}(x)}|\Ric|^{p}<\delta$ for some $p>n/2$. For any  $\vepsilon>0$, if $\delta<\delta_0(n,\vepsilon)$ then 
\begin{enumerate}
\item If there exists a $(k,\delta)$-splitting map on $B_{2}(x)$, then $d_{GH}(B_1(x),B_1(0^k,z))<\vepsilon$ for some length space $(0^k,z)\in\mathbb{R}^k\times Z$.
\item If  $d_{GH}(B_2(x),B_2(0^k,z))<\delta$ for some length space $(0^k,z)\in\mathbb{R}^k\times Z$, then there exists a $(k,\vepsilon)$-splitting map on $B_{1}(x)$.
\end{enumerate}
 \end{theorem}

Recall the volume ratio $\displaystyle \cV_r(x):=\frac{\Vol(B_r(x))}{\Vol(B_r(0^n))}$,  where $\Vol(B_r(0^n))=\omega_nr^n$ is the volume of $r$-ball in $\mathbb{R}^n$.
\begin{theorem}[Almost volume cone implies almost metric cone]\label{volume cone} Let $(M^n,g)$  satisfy  $\displaystyle  \int_{B_{4}(x)}|\Ric|^{p}<\delta $ for some $p>n/2$ and $|\cV_{r}(x)-\cV_{2r}(x)|\le \delta$  for some $r\le 2$. For any $\vepsilon>0$, there exists $\delta(n,\vepsilon)>0$ such that if $\delta\le \delta(n,\vepsilon)$ then \[d_{GH}(B_{r}(x),B_{r}(z))<\vepsilon r,
\] for some metric cone $z\in C(Z)$.
\end{theorem}
 
 \begin{theorem}[Almost maximal volume implies almost Euclidean] Let $(M^n,g)$  satisfy  $\displaystyle  \int_{B_{4}(x)}|\Ric|^{p}<\delta$ for some $p>n/2$ and $\cV_{2r}(x)\ge 1-\delta$ for some $r\le 2$. For any $\vepsilon>0$, there exists $\delta(n,\vepsilon)>0$ such that if $\delta\le \delta(n,\vepsilon)$ then \[d_{GH}(B_{r}(x),B_{r}(0^n))<\vepsilon r.
\] 
\end{theorem}
Take a decreasing sequence of small scales $R_0\ge R_1\ge \cdots \ge R_k\to 0$. In the case of  $\Ric\ge 0$,   the standard Bishop-Gromov volume comparison theorem implies the summability of volume drop \[\sum_{k=0}^{\infty} |\mathcal{V}_{R_{k+1}}(x) - \mathcal{V}_{R_k}(x)|=\sum_{k=0}^{\infty} \left(\mathcal{V}_{R_{k+1}}(x) - \mathcal{V}_{R_k}(x)\right)=1-\mathcal{V}_{R_0}(x).
\] This summability playes a key role in the study of infinitesimal structure of manifold with lower Ricci curvature bound and its limit space. If we have $L^p$ Ricci curvature bound instead of lower Ricci curvature bound, we still have the summability of volume drop:
\begin{lemma}[Summability of volume drop]\label{summability of volume drop}
Let \((M^n, g)\) satisfy $(A1)$ and $(A2)$. Then there exists $R_0=R_0(n,p,\rv)\Lambda^{-1/(2p)}>0$ and $R_k=2^{-k}R_0$ such that 
\[\sum_{k=0}^{\infty} |\mathcal{V}_{R_{k+1}}(x) - \mathcal{V}_{R_k}(x)|\le C(n, p, \rv, \Lambda)<\infty.
\]
In particular, for any $\delta>0$, there are at most $N(\delta,n,p,\rv,\Lambda)$ many scales such that $|\mathcal{V}_{R_{k+1}}(x) - \mathcal{V}_{R_k}(x)|>\delta.$
\end{lemma}
\begin{proof}
By the volume comparison theorem (Theorem \ref{Volume comparison}), for $0 < r \le R \le 1$:
\begin{equation}\label{eq:volume drop}
    \mathcal{V}_r(x) \ge \mathcal{V}_R(x) \left(1 - C(n,p)\delta(R)^{1/2}\right)^{2p},
\end{equation}
where $\delta(R) = R^2 (\fint_{B_R(x)} |\Ric|^p)^{1/p}$. Using the uniform non-collapsing assumption $\Vol(B_R(x)) \ge \rv R^n$, we can bound $\delta(R)$:
\begin{equation} \label{eq:delta_decay}
    \delta(R)^p = R^{2p} \frac{1}{\Vol(B_R(x))} \int_{B_R(x)} |\Ric|^p \le \rv^{-1} R^{2p-n} \int_{B_1(x)} |\Ric|^p \le \rv^{-1} \Lambda R^{2p-n}.
\end{equation}
Because $p > n/2$, the exponent $2p-n > 0$. Consequently, $\delta(R) \le C(p,\rv)\Lambda^{1/p} R^{2 - n/p}$ decays polynomially as $R \to 0$.

By (\ref{eq:volume drop}) there exists $R_0=R_0(n,p,\rv)\Lambda^{-1/(2p)}>0$ such that  for $0 < r \le R\le R_0$: 
\[ \mathcal{V}_r(x) \ge \frac{1}{2}\mathcal{V}_R(x).
\]  Since $\mathcal{V}_r(x) \to 1$ as $r\to 0$, this in particular implies $c(n,\rv)\le\mathcal{V}_R(x)\le 2$.

By (\ref{eq:delta_decay}), the volume ratio volume drop satisfies:
\[\cV_R(x) - \cV_r(x) \le \cV_r(x) \big[ (1 - C(n,p)\delta(R)^{1/2})^{-2p} - 1 \big] \le \cE(R),\]
where $\cE(R) =C(n, p,\rv)\Lambda^{1/(2p)}R^{1 - n/(2p)}.$  
To bound the absolute value, note that $|y| = y + 2\max(0, -y)$. Applying this to the volume difference $y = \cV_r(x) - \cV_R(x)$, we get:
\begin{equation}
    |\cV_r(x) - \cV_R(x)| = (\cV_r(x) - \cV_R(x)) + 2\max(0, -(\cV_r(x) - \cV_R(x))\le (\cV_r(x) - \cV_R(x))+2\cE(R).
\end{equation}
Thus we conclude
\[\begin{aligned}
    \sum_{k=0}^{\infty} |\mathcal{V}_{R_{k+1}}(x) - \mathcal{V}_{R_k}(x)| &\le  1- \mathcal{V}_{R_0}(x) + 2  \sum_{k=0}^{\infty} \cE(R_k)\\
    & \le 1-c(n,\rv)+C(n,p,\rv)\Lambda^{1/(2p)} \sum_{k=0}^{\infty} (2^{-k})^{1 - n/(2p)}\\ 
    &\le 1-c(n,\rv)+C(n,p,\rv)\Lambda^{1/(2p)}.
\end{aligned}\]
\end{proof} 
Combing precompactness theorem \ref{precompactness} and classical cone splitting principle, we have the following quantitative cone splitting:
\begin{theorem}[quantitative cone splitting]\label{cone-splitting}
For every $\epsilon, \tau > 0$, there exists $\delta(n, \epsilon, \tau) > 0$ such that if
\begin{enumerate}
    \item[$(1)$] $\displaystyle  \int_{B_{4}(x)}|\Ric|^{p}<\delta$ for some $p>n/2$,
    \item[$(2)$] $B_2(x)$ is $(k, \delta)$-symmetric with respect to $L_\delta^k \subseteq B_4(x)$,
    \item[$(3)$] there exists $z \in B_1(x) \setminus B_\tau L_\delta^k$ such that $B_2(z)$ is $(0, \delta)$-symmetric,
\end{enumerate}
then $B_1(x)$ is $(k+1, \epsilon)$-symmetric.
\end{theorem}

Combing assumption (A3) and  a quantitative stratification Theorem 3.9 in \cite{Q} under $L^p$ Ricci curvature bound, we can prove the $L^\alpha$ estimate of curvature for any $\alpha<2$.
\begin{theorem}[\cite{CN15}, Theorem 1.8]\label{Lq estimate}
Let $(M^n,g)$  satisfy assumptions $(A1),(A2),(A3)$, then for any $\alpha<2$,  there exists $C=C(n,p,\rv,\Lambda,\alpha)>0$ such that
\[ \int_{B_1(q)}|\Rm|^\alpha \le C.
\]
\end{theorem} 
\begin{proof}The proof follows the same line of  \cite{CN15}. To begin with, let us define $r_y=\sup\{0<r\le 1:\sup_{B_r(y)}|\Rm|\le r^{-2}\}$. Let \( \alpha < 2 \) be fixed, set \( \epsilon = 4 - 2\alpha \). By assumption (A3), there exists $\vepsilon=\vepsilon(n)>0$ such that if $y$ is $(n-3,\vepsilon,2r)$-symmetric, then $r_y> r.$ The
contrapositive gives the inclusion
\begin{equation}\label{inclusion}\begin{aligned}\{y\in B_1(q):r_y\le r\}\subset S_{\varepsilon, 2r}^{n-4} \cap B_1(q).\end{aligned}
\end{equation}  And by Theorem 3.9 in \cite{Q},   there exists  
\( C(n,p,\rv,\Lambda,\alpha)>0\) such that  
\[
\Vol\bigl( T_r(\{y \in S_{\varepsilon, 2r}^{n-4} \cap B_1(q)\}) \bigr) < C r^{4- \epsilon}.
\]
 By inclusion $(\ref{inclusion})$
\[
\Vol\bigl( T_r(\{y \in B_1(q) : r_y \leq r\}) \bigr) < C r^{4-\epsilon}= C r^{2\alpha}.
\]
 Put \( \eta = 2 - \alpha \). Then \( \alpha + \frac{\eta}{2} < 2 \). Then we have already shown that
\begin{equation}\label{volume estimate}\begin{aligned}
\Vol(T_r(\{y \in B_1(q) : r_y\leq r\})) < C r^{2\alpha+\eta},
 \end{aligned}
\end{equation} 
for $C(n,p,\rv,\Lambda,\alpha)>0$. Consider the Vitali covering \(\{B_{\frac{1}{5}r_y}(y)\}\) of \(B_1(q)\) such that
\begin{enumerate}
\item \( B_1(q) \subseteq \bigcup B_{r_j}(y_j) \) with \( r_j =  r_{y_j} \);
\item \(\{B_{r_j/5}(y_j)\}\) are disjoint.
\end{enumerate}
By using (\ref{volume estimate}), we see that for each \( \beta \in \mathbb{N} \), we have
\[
\sum_{2^{- \beta -1} < r_j \leq 2^{- \beta }} r_j^{-2q}\Vol(B_{r_j}(y_j)) \leq C  r_j^{\eta} = C   2^{-\eta \beta }.
\]
By summing over \(  \beta  \), this gives
\[
\sum_{j=1}^\infty r_j^{-2q}  \Vol(B_{r_j}(y_j)) \leq C \sum_{ \beta =1}^\infty\sum_{2^{- \beta -1} < r_j \leq 2^{- \beta }} r_j^{-2q}\Vol(B_{r_j}(y_j)) \leq C.
\]
Finally, by definition of $r_j$ we get
\[
\int_{B_1(q)} |\mathrm{Rm}|^\alpha \leq   \sum_{j=1}^\infty \int_{B_{r_j}(y_j)} |\mathrm{Rm}|^\alpha
\leq C \sum_{j=1}^\infty r_j^{-2q}\Vol(B_{r_j}(y_j)) \leq C.
\]
\end{proof}
Arguing the same as Theorem 7.20 of \cite{CN15}, together with quantitative stratification Theorem 3.9 in \cite{Q}, we can prove one can obtain $L^\alpha$ bounds on the Hessians of such harmonic functions for all $\alpha < 4.$
\begin{theorem}[\cite{CN15}, Theorem 7.20]\label{Lq hessian} Let $(M^n,g)$  satisfy assumptions $(A1),(A2),(A3)$. Let \( u : B_2(q) \to \mathbb{R} \) satisfy  
\[
|u| \leq 1, \quad |\Delta u| \leq 1,
\]
then for every \( \alpha < 4 \),  there exists $C=C(n,p,\rv,\Lambda,\alpha)>0$  such that
\[
\int_{B_1(q)} |\nabla^2 u|^\alpha \leq C.
\]
 \end{theorem} 
 \subsection{Regularity of Ricci flows with scalar bounded curvature}
 In this section, we review some basic results of Ricci flows with scalar bounded curvature developed by Bamler and others. Let \( (M^n, (g_t)_{t \in [-2,0]}) \) be a Ricci flow on a compact  manifold that satisfies
\begin{enumerate}
\item[(i)] \( \nu[g_{-2}, 4] \geq -A \);
\item[(ii)] \( |\R| \leq A \) on \( M \times [-2,0] \).
\end{enumerate}
By Perelman's no local collapsing theorem, it is known that there exists $\rv=\rv(A)$ such that  for any $(x, t)\in M \times [-1, 0]$ and $r > 0$, we have $\Vol(B_{g(t)}(x,r))\ge \rv r^n.$ Recall the following pseudolocality theorem, see \cite{BZ15},\cite{CY}.

\begin{theorem}[Pseudolocality]
  Let \( (M^n, (g_t)_{t \in [-2,0]}) \) be a Ricci flow satisfying $(i)$ and $(ii)$ above and \( (x, t) \in M \times [-1,0] \), \( 0 < r < 1 \). Then for any $\vepsilon>0$ there exists $\delta(\vepsilon,n,A)>0$ such that if \[d_{GH}(B_{g(t)}(x,r),B_r(0^n))<\delta r,\] then 
  \[ |\Rm|_{g(t)}\le \left(\frac{\vepsilon} {r}\right)^2 \quad\text{on}\quad B_{g(t)}(x, r) \times [t-(r/\vepsilon )^2, t + ( r/\vepsilon )^2].
  \]

 \end{theorem}
 Let us now recall compactness and partial regularity result for sequences of Ricci flows with bounded scalar curvature due to Bamler.
 \begin{theorem}[Compactness of Ricci flows with bounded scalar curvature]
  Let \((M^n_i, (g_i^t)_{t \in [-2,0]} ,q_i)\) be a sequence of pointed Ricci flows that satisfying $(i)$ and $(ii)$ above.  Then after passing to a subsequence,
there is an \( n \)-dimensional  pointed singular space \( (\mathcal{X}, q_\infty) \),  such that the sequence of pointed Riemannian manifolds \( (M_i, g_0^i, q_i) \)
converges to \( (\mathcal{X}, q_\infty) \). Moreover, the singular space \( \mathcal{X} \) has singularities of codimension 4.
 \end{theorem}
Now we introduce Bamler's $L^{4-\epsilon}$ estimate for curvature radius. More precisely, define the curvature radius  \[r_{\Rm}(x,t):=\sup\{r>0:|\Rm|_{g(t)}\le r^{-2}\text{ on $B_{g(t)}(x,r)$}\}.\] 
\begin{theorem}
 Let \( (M^n, (g_t)_{t \in [-2,0]}) \) be a Ricci flow satisfying $(i)$ and $(ii)$ above and \( (x, t) \in M \times [-1,0] \). For any $\epsilon>0$, there exists $C(n,A,\epsilon)>0$ such that \[
\int_{B_{g(t)}(x,1)}\left( r_{\Rm}(\cdot,t)\right)^{-(4-\epsilon)}d\Vol_{g(t)}\le C(n,A,\epsilon).
\]
\end{theorem}
Thanks to Proposition \ref{Improved Ricci curvature estimate}, we have $|\Rm(x,t)|\le r^{-2}_{\Rm}(x,t)$ and $|\Ric(x,t)|^2\le C_0A\,r^{-2}_{\Rm}(x,t)$, it is then immediate that \[\int_{B_{g(t)}(x,1)} |\Rm(\cdot,t)| ^{2-\epsilon}d\Vol_{g(t)}\le C(n,A,\epsilon),
\]and 
\[\int_{B_{g(t)}(x,1)} |\Ric(\cdot,t)| ^{4-\epsilon}d\Vol_{g(t)}\le C(n,A,\epsilon).
\]
 Finally, we recall the following improved estimate for Ricci curvature on Ricci flows with bounded scalar curvature, which roughly gives the curvature inequality $|\Ric|\le \sqrt{|\R||\Rm|}$, see \cite{W11},\cite{BZ15},\cite{B18}. Combined with the above pseudolocality theorem and codimension four of singularity result, we observe that assumption (A3) on Ricci flows with bounded scalar curvature is  satisfied automatically. 
 
 \begin{proposition}[Improved Ricci curvature estimate]\label{Improved Ricci curvature estimate}
There are constants \( C_0, C_1, \ldots < \infty \) that only depend on the dimension \( n \) such that the following holds: Let \( (M^n, (g_t)_{t \in [-2,0]}) \) be a Ricci flow on compact manifold $M$ and \( (x, t) \in M \times [-1,0] \), \( 0 < r < 1 \). Assume that \( |\Rm| \le r^{-2} \) and \( |\R| \leq A\) on the parabolic neighborhood \( P(x, t, r, -r^2)=B_{g(t)}(x,r)\times[t-r^2,t] \). Then for all \( m \geq 0 \),
\[
   \sup_{B_{g(t)}(x,r/2)}|\nabla^m \Ric|\le C_m A^{1/2} r^{-m-1}.
\]   
\end{proposition}
 
\section{Neck regions}
The central geometric objects in the proof of $L^2$ estimate are \emph{neck regions}.  They are open sets that, on many scales, close to cone spaces $\mathbb{R}^{n-4}\times C(Z^3)$. Let us briefly recall the definition of neck region here, for a detialed explanation of neck region, see \cite{JN} and \cite{CJN}. 
\begin{definition}[$(\delta,\tau)$-neck region] \label{def:neck-region}
Let \((M^n,g)\) be a Riemannian manifold.
A subset \(\N = B_2(q)\setminus \bigcup_{x\in\C}\overline{B}_{r_x}(x)\) is called a $(\delta,\tau)$-neck region if there exist a closed set \(\C = \C_0\cup\C_+\) (with \(q\in\C\)), a radius function \(r:\C_+\to (0,\infty)\) (and \(r_x\equiv 0\) on \(\C_0\)), and a constant \(\tau>0\) such that:
\begin{enumerate}
    \item[$(n1)$] The balls \(\{B_{\tau r_x}(x)\}_{x\in\C_+}\) are pairwise disjoint.
    \item [$(n2)$] For every \(x\in\C\) and every \(r_x\le s\le 1\) there exists a pointed $\delta s$-Gromov--Hausdorff approximation
        \[
        \iota_{x,s}: B_{\delta^{-1}s}(0^{n-4},y_c)\subset \mathbb{R}^{n-4}\times C(S^3/\Gamma) \longrightarrow B_{\delta^{-1}s}(x),
        \]
        where \(\Gamma\subset O(4)\) is nontrivial and \(\iota_{x,s}(0^{n-4},y_c)=x\).
    \item [$(n3)$] For each \(x\in\C\) and \(r_x\le s\) with \(B_{2s}(x)\subset B_2(q)\), we have $
        \mathcal{L}_{x,s}:=\iota_{x,s}\bigl(B_s(0^{n-4})\times\{y_c\}\bigr) \subset B_{\tau s}(\mathcal{C})$ and $
        \C\cap B_s(x) \subset B_{\delta s}(\mathcal{L}_{x,s}).$
    \item [$(n4)$] The radius function is a Lipschitz function with $
        |\Lip r_x| := \sup_{y,z\in\C}\frac{|r_y-r_z|}{d(y,z)} \le \delta.$
\end{enumerate}
For \(\tau\le s\) we denote ``enlarge'' neck region \(\N_s = B_2(q)\setminus \bigcup_{x\in\C}\overline{B}_{s\cdot r_x}(x)\). 
We also say $\N \subset B_{2R}(q)$ is a neck region for some $R > 0$ if
on the rescaled ball $B_2(q)=R^{-1}B_{2R}(q)$ with $r_x$ becomes $r_x/R$, we have that
$\N$ satisfies the above. To each neck region we associate a \emph{packing measure}:
\[
\mu_\N := \sum_{x\in\C_+} r_x^{n-4}\,\delta_x \;+\; \Haus^{n-4}\llcorner\C_0 .
\]
\end{definition}
 \begin{remark}\label{remark_neck}
 Since we will only focus on smooth manifolds, the set $\C_0$ is empty. 
 
 As pointed by Lemma 3.11 in \cite{JN}, the harmonic radius at a point in the neck region is effectively equivalent to the distance of that point to the effective singular set $\cC$, i.e. $r_h(y)\approx d(y,\cC)$ for $y\in \cN_{10^{-6}}$ up to a constant.  
\end{remark}

Next we can build a good cutoff function associated with a neck region.
\begin{lemma}[Good cutoff functions]\label{Good cutoff fuctions}
Let \( (M^n, g) \) satisfy assumptions $(A1),(A2)$, and \( \cN \subseteq B_2(q) \) a \( (\delta, \tau) \)-neck region.  
Then for \( \tau < \tau(n) \) and \( \delta < \delta(n, \rv, \tau) \) there exists a cutoff function \( \phi_\cN  = \phi_1 \cdot \phi_2: B_2(q) \to [0, 1] \) such that  

\begin{enumerate}
    \item[$(1)$] \( \phi_1(y) \equiv  1 \) if \( y \in B_{18/10}  \) with \( \supp(\phi_1) \subseteq B_{19/10} \).  
    \item[$(2)$] \( \phi_2(y) \equiv 1 \) if \( y \in \cN_{10^{-3}} \) with \( \supp(\phi_2)  \subseteq \cN_{10^{-4}} \).  
    \item[$(3)$] \( |\nabla \phi_1|, |\Delta \phi_1| \leq C(n) \).  
    \item[$(4)$] \( \supp(|\nabla \phi_2|) \cap B_{19/10} \subseteq A_{10^{-4}r_x, 10^{-3}r_x}(\mathcal{C}) \) with \( r_x |\nabla \phi_2|, r_x^2 |\nabla^2 \phi_2| < C(n, \tau, \rv) \) in each \( B_{r_x}(x) \).
\end{enumerate}
\end{lemma} 
\begin{proof}
This follows from the same argument of Lemma 3.15 in \cite{JN} together with good cutoff functions under $L^p$ Ricci curvature bound constructed in Lemma 5.4 in \cite{DWZ}.\end{proof}
The following theorem constructs a Reifenberg type map for the effective singular set $\cC$ of neck regions, and plays a key rule to prove the lower Ahlfors regularity bound. As remarked in  \cite{JN}, one can check that the constructions of this theorem work for a general Reifenberg type set.
\begin{theorem}[\cite{JN}, Theorem 3.24]\label{Reifenberg map}
   let $\cN \equiv B_2(q) \setminus \bar{B}_{r_x}(\mathcal{C})$ be a $(\delta, \tau)$-neck region with $\max_{x \in \mathcal{C}} r_x \leq 10^{-3}$. For each $\epsilon > 0$, if $\tau < \tau(n)$ and $\delta \leq \delta(n, \epsilon)$, then there exists a map $\Phi : \mathcal{C} \cap B_{19/10}(q) \to \mathbb{R}^{n-4}$ with $\Phi(q) = 0^{n-4}$ such that

\begin{enumerate}
    \item[$(1)$] $(1 - \epsilon)d(x, y)^\epsilon \leq \dfrac{|\Phi(x) - \Phi(y)|}{d(x, y)} \leq d(x, y)^{-\epsilon};$
    
    \item[$(2)$] for each $x \in \mathcal{C} \cap B_{9/5}(q)$ and $r_x \leq r \leq 10^{-3}$, there exists $T_{x,r} = T \in \mathbb{R}^{(n-4) \times (n-4)}$ such that the map $T \circ \Phi : B_{30r}(x) \cap \mathcal{C} \to \mathbb{R}^{n-4}$ is an $\epsilon r$-isometry and $T_{x,10^{-3}} = \operatorname{Id};$
    
    \item[$(3)$] $|T_{x,r} \circ (T_{x,s})^{-1}| + |T_{x,s} \circ (T_{x,r})^{-1}| \leq C(n, \epsilon) \left(\dfrac{r}{s}\right)^\epsilon$ for all $r_x \leq s \leq r \leq 10^{-3}$ and $x \in \mathcal{C} \cap B_{9/5}(q);$
    \item[$(4)$] for any $x, y \in \mathcal{C} \cap B_{9/5}(q)$ and $10^{-3} \geq r \geq r_x$ with $d(x, y) \leq 10r$, there exists $O_{xy,r} \in O(n-4)$ such that $|T_{x,r} \circ (O_{xy,r}T_{y,r})^{-1} - \operatorname{Id}| \leq \epsilon;$
    
    \item[$(5)$] if $T_x := T_{x,r_x}$, then we have the covering
    \[
    B_{7/4}(0^{n-4}) \subseteq \bigcup_{x \in \mathcal{C} \cap B_{9/5}(q)} T_x^{-1} \overline{B _{r_x}}(\Phi(x)),
    \]
    where $T_x^{-1} \overline{B _{r_x}}(\Phi(x)) := \{T_x^{-1}v + \Phi(x) : v \in \overline{B _{r_x}}(0^{n-4})\} \subset \mathbb{R}^{n-4}$.
\end{enumerate}
\end{theorem}
\section{$L^2$ estimate on neck regions}
 
Let \(\cN \subseteq B_2(q)\) be a \((\delta, \tau)\)-neck region with packing measure \(\mu_\cN\) on $(M^n,g)$. The main goal of this section is to establish $L^2$ curvature bound on neck regions (Theorem \ref{strong L2 neck}) under an addtional Ahlfors regularity assumption:
 \begin{enumerate}
    \item[(A4)]\textbf{Ahlfors regularity:} 
There exists $B>0$, such that for each \(x \in \mathcal{C}\) and \(r_x < r\) with \(B_{2r}(x) \subseteq B_2(q)\), we have \begin{equation}
\begin{aligned}\label{Ahlfors regularity}
B^{-1}r^{n-4} < \mu(B_r(x)) < Br^{n-4}.
\end{aligned}
\end{equation}\end{enumerate}  In  next sections, we will show that (A4) is a consequence of  (A1), (A2), and (A3). 

More precisely, we want to prove:
\begin{theorem}\label{weak L2 on neck}
Let $(M^n,g)$  satisfy assumptions $(A1), (A2), (A3)$ and $(A4)$.   If $\delta\le\delta_0(n,p,\rv,\Lambda,B)$, then there exists $C(n,p,\rv,\Lambda,B)>0$ such that
\[
r^{4-n} \int_{\mathcal{N}_{10^{-5}} \cap B_r(x)} |\Rm|^2(z) \, dz \leq C.
\] 
\end{theorem} 
\begin{remark}\label{weak L4 Ricci on neck} 
Under the assumption of Theorem \ref{weak L2 on neck}, we can also prove \[
r^{4-n} \int_{\mathcal{N}_{10^{-5}} \cap B_r(x)} |\Ric|^4(z) \, dz \leq C.
\] Of course, this estimate is nontrivial only when 
$4>p>n/2$.
\end{remark}
The proof of Theorem \ref{weak L2 on neck} follows from a powerful superconvexity estimate originated from \cite{JN}. Let's begin with a discussion of Green’s functions on neck region with respect to the packing measure $\mu$ .

\begin{definition}\label{def_Green}
Given \((M^n, g)\) with \(\cN \subseteq B_2(q)\) a \((\delta, \tau)\)-neck region with packing measure \(\mu_\cN\), let us consider \(\bar{\mu} = \mu_\cN|_{B_{39/20}(q)}\). We define the following:

\begin{enumerate}
    \item We denote by \(G_x(y)\) a Green's function at \(x\). That is, \(-\Delta_y G_x(y) = \delta_x\), where \(\delta_x\) is the Dirac delta at \(x\).

    \item We denote by \(G_{\bar{\mu}}(y)\) the function \(G_{\bar{\mu}}(y) = \int_{B_{39/20}(q)} G_x(y) \, d\mu(x)\) a Green's function that solves \(-\Delta G_{\bar{\mu}} = \bar{\mu}\).

    \item We denote by \(b(y) \equiv G_{\bar{\mu}}^{-1/2}\) the \(\bar{\mu}\)-Green's distance function to the effective singular set \(\mathcal{C}\).
\end{enumerate}
\end{definition}
\begin{remark}
 In fact, thanks to the strong $\vepsilon$-regularity assumption (A3), if $\delta\le \delta(n,\vepsilon)$ is small enough,   we can show $|\Rm|\le \vepsilon b^{-2} \text{ on $\cN_{10^{-5}}$}.$ In particular 
 \[ \int_{\cN_{10^{-5}}\cap B_1(x)}|\Rm|^2\le  \vepsilon  \int_{\cN_{10^{-5}}\cap B_1(x)} b^{-4}.
 \]
However, it can be shown (see (\ref{integral})) that \[ \int_{\cN_{10^{-5}}\cap B_1(x)} b^{-4+\epsilon}\le C(n,\rv,B,\tau,\epsilon)<\infty,  
 \] for any $\epsilon>0$, while $\int_{\cN_{10^{-5}}\cap B_1(x)} b^{-4}=\infty.$ 
 Thus  strong $\vepsilon$-regularity only gives weak $L^2$ curvature bound. This is why we need   \textit{superconvexity estimate}   to prove the full $L^2$ curvature bound.  
\end{remark}
The next lemma plays a key role in the proof of superconvexity estimate.

\begin{lemma}[Green's distance function estimate]\label{Green's distance function estimate}
Let \((M^n, g)\) satisfy assumptions  $(A1),(A2)$ and $(A4)$. Then for any \(B > 0\), there exists \(\delta(n,p, \rv,\Lambda,B), \tau(n), C(n,p, \rv,\Lambda,B) > 0\) and a Green function \(G_{\bar{\mu}}\) such that  for \(x \in \cN_{10^{-6}} \cap B_{39/20}(q)\), we have
\begin{enumerate} 
   \item[$(1)$] \(C^{-1} d(x, \mathcal{C}) \leq b(x) \leq C d(x, \mathcal{C})\);
    \item[$(2)$]  \(C^{-1} \leq |\nabla b| \leq C\).
\end{enumerate} 
\end{lemma}
 
\begin{proof} 
Under the assumptions of (A1) and (A2), it is known by \cite{DWZ} that such manifold satisfies  a uniform Sobolev inequality. Consequently, by Proposition 3.16 of \cite{LSZ}, there exists $C(n, p,\rv,\Lambda)>0$ such that  Green's function $G_x(y)$ and its gradient satisfies the following standard estimates:
\begin{equation}\label{Green's estimate}
    C ^{-1} d(x,y)^{2-n} \le G_x(y) \le C  d(x,y)^{2-n}, \quad |\nabla_y G_x(y)| \le C  d(x,y)^{1-n}.
\end{equation}

Fix $y \in \mathcal{N}$ and let $r = d(y, \C)$. Let $x_0 \in \C$ be a point such that $d(y, x_0) = r$. We decompose  $B_{39/20}(q)$  into annuli: $A_0 =  B_{39/20}(q) \cap B_{2r}(x_0)$ and $A_k = B_{39/20}(q)\cap (B_{2^{k+1}r}(x_0) \setminus B_{2^k r}(x_0))$ for $k \ge 1$.  

For $x \in A_0$, we have $r \le d(x,y) \le 3r$. Thus, $C^{-1} r^{2-n}\le G_x(y)\le Cr^{2-n}$. Since $\mu(A_0) \le B(2r)^{n-4} \le  Cr^{n-4}$, we have 
\begin{equation}
    \int_{A_0} G_x(y) d\mu(x) \le C r^{2-n} \cdot r^{n-4} = C r^{-2}.
\end{equation}
For $x \in A_k$, we have $d(x,y) \ge d(x, x_0) - d(x_0, y) \ge 2^k r - r \ge 2^{k-1}r$ for $k \ge 1$. The measure is $\mu(A_k) \le B(2^{k+1}r)^{n-4}$. Thus:
\begin{equation}
    \int_{A_k} G_x(y) d\mu(x) \le C (2^k r)^{2-n} \cdot (2^k r)^{n-4} = C 2^{-2k} r^{-2}.
\end{equation}
Summing over $k \ge 0$, we obtain \[G_{\bar{\mu}}(y) = \sum_{k=0}^{\infty} \int_{A_k} G_x(y) d\mu(x) \le C r^{-2} \sum_{k=0}^{\infty} 2^{-2k} = C' r^{-2}.\] The lower bound follows similarly by considering the integral over $B_{r/2}(x_0) \cap \C$. Since $b(y) = G_{\bar{\mu}}^{-1/2}$, we conclude $C^{-1}r\le  b(y)\le  C r$ with $r = d(y, \C)$.

To derive an upper bound for $|\nabla b(y)|$, 
We differentiate $b(y)$ to get $\nabla b = -\frac{1}{2} G_{\bar{\mu}}^{-3/2} \nabla G_{\bar{\mu}}$. From above computation, $ C^{-1} r^3 \le G_{\bar{\mu}}^{-3/2} \le C r^3$.   By use the same dyadic decomposition:
\begin{equation}
    |\nabla G_{\bar{\mu}}(y)| \le \int_{B_{39/20}(p) } |\nabla_y G_x(y)| d\mu(x) \le \sum_{k=0}^{\infty} \int_{A_k} C d(x,y)^{1-n} d\mu(x).
\end{equation}
We estimate each term:
\begin{equation}
    \int_{A_k} d(x,y)^{1-n} d\mu(x) \le C (2^k r)^{1-n} (2^k r)^{n-4} = C 2^{-3k} r^{-3}.
\end{equation}
Summing gives $|\nabla G_{\bar{\mu}}(y)| \le C r^{-3}$, which implies $|\nabla b(y)| \le C r^3 \cdot r^{-3} = C$.

Now we establish the lower bound $|\nabla b(y)| \ge C^{-1}$ through a contradiction argument. Suppose the conclusion fails. Then there exists a sequence of  Riemannian manifolds $(M^n_j, g_j,q_j)$ satisfying uniform  (A1),(A2)  and  (A4) with $\delta_j \to 0$, and a sequence of points $y_j \in \cN^{(j)}_{10^{-6}}$ such that
\begin{equation}
    |\nabla b_j(y_j)|_{g_j} \le \frac{1}{j} \longrightarrow 0, \quad \text{as } j \to \infty.
\end{equation}
Let $r_j = d_{g_j}(y_j, \cC_j)$ be the distance to the center set $\cC_j$, and pick $x_j\in \cC _j$ so that $r_j = d_{g_j}(y_j, x_j)$. Rescale the metrics by setting $\tilde{g}_j = r_j^{-2} g_j$. Under this sequence of rescaled metrics:
\begin{enumerate}
    \item The distances scale as $\tilde{d}_j(x, y) = r_j^{-1} d_{g_j}(x, y)$. In particular, $\tilde{d}_j(y_j,  x_j) = 1$.
    \item The $(n-4)$-dimensional Ahlfors regular packing measure $\mu_j$ on $\C_j$ scales as $\tilde{\mu}_j = r_j^{4-n} \mu_j$.
    \item The fundamental solution to the Laplacian scales as $\tilde{G}_{x}(y) = r_j^{n-2} G_{x}(y)$.
\end{enumerate}
Consequently,  $G_{\bar{\mu}_j}(y) = \int_{B_{39/20}(q_j)} G_{x}(y) d\mu_j(x)$ transforms under rescaling as:
\begin{equation}
    \tilde{G}_{\tilde{\mu}_j}(y)= \int_{B_{39/20}(q_j)} \tilde{G}_{x}(y) d\tilde{\mu}_j(x) = \int_{B_{39/20}(q_j)} (r_j^{n-2} G_{x}(y)) (r_j^{4-n} d\mu_j(x)) = r_j^2G_{\bar{\mu}_j}(y).
\end{equation}
Therefore, the rescaled Green's distance function is $\tilde{b}_j(y)  = r_j^{-1} b_j(y)$. The pointwise norm of the gradient is exactly scale-invariant:
\begin{equation}
    |\tilde{\nabla} \tilde{b}_j|_{\tilde{g}_j} = r_j \cdot |\nabla (r_j^{-1} b_j)|_{g_j} = |\nabla b_j|_{g_j}.
\end{equation}
Thus, our contradiction hypothesis implies $|\tilde{\nabla} \tilde{b}_j(y_j)|_{\tilde{g}_j} \to 0$.
  Thanks to the uniform $L^p$ Ricci bound, the non-collapsing condition and Ahlfors regularity, the sequence of pointed metric measure spaces $(M_j, \tilde{g}_j, x_j, \tilde{\mu}_j)$ is precompact in the pointed measured Gromov-Hausdorff topology. Up to passing to a subsequence, they converge to a limit space $(M_\infty, d_\infty, x_\infty, \mu_\infty)$. 

Since the points $x_j\in \cC_j$ with $\delta_j \to 0$, by Cheeger-Colding's almost splitting theorem (Theorem \ref{almost splitting}), the limit space $M_\infty$ is isometric to the  $\mathbb{R}^{n-4} \times C(S^3/\Gamma)$ for some $\Gamma\le O(3)$. The singular sets $\C_j$ converge to the singular $\C_\infty = \mathbb{R}^{n-4} \times \{0\}$. The rescaled measures $\tilde{\mu}_j$ converge weakly to $\mu_\infty$, which is a constant multiple of the $(n-4)$-dimensional Lebesgue measure on $\mathbb{R}^{n-4} \times \{0\}$.
  
On the sequence of manifolds $M_j$,  $\tilde{G}_{\tilde{\mu}_j}$ solve the elliptic equation $-\Delta_{\tilde{g}_j} \tilde{G}_{\tilde{\mu}_j}= \tilde{\mu}_j$. By $(\ref{Green's estimate})$, we have 
uniform bound of $\tilde{G}_{\tilde{\mu}_j}$ and $\tilde{\nabla}\tilde{G}_{\tilde{\mu}_j}$ away from
singular $\C_\infty$. To obtain equicontinuity, we rely on higher-order interior regularity. Under the $L^p$ Ricci bound, we have harmonic coordinates with uniform $W^{2,p}$ (and hence $C^{\alpha}$ by Sobolev embedding) bounds on the metric tensor away from $\cC_j$. Consequently, the Schauder theory for elliptic equation $-\Delta_{\tilde{g}_j} \tilde{G}_{\tilde{\mu}_j}= \tilde{\mu}_j$ provides uniform local  $C^{1,\alpha}$ estimates for the solutions $\tilde{G}_{\tilde{\mu}_j}$. 

 By the Arzel\'a-Ascoli theorem, the functions $\tilde{G}_{\tilde{\mu}_j}$ and their gradients $\tilde{\nabla} \tilde{G}_{\tilde{\mu}_j}$ converge uniformly on compact subsets of $M_\infty \setminus \C_\infty$ to a limit harmonic function $G_\infty$ and its gradient $\nabla G_\infty$. This  convergence guarantees that:
\begin{equation}
    \lim_{j \to \infty} |\tilde{\nabla} \tilde{b}_j(y_j)|_{\tilde{g}_j} = |\nabla b_\infty(y_\infty)|_{g_\infty},
\end{equation}
where $y_j\to y_\infty$ and $b_\infty = G_\infty^{-1/2}$. 

Finally, we can explicitly compute $|\nabla b_\infty(y_\infty)|_{g_\infty}$ in the cone space $\mathbb{R}^{n-4} \times C(S^3/\Gamma)$. Since $\mu_\infty$ is the Lebesgue measure on  $\mathbb{R}^{n-4} \times \{0\}$, the solution $G_\infty$ to $-\Delta G_\infty = \mu_\infty$ depends only on the radial distance $\rho = d_\infty(\cdot, \C_\infty)$ within the cone factor $C(S^3/\Gamma)$. 

In a 4-dimensional cone, the fundamental solution depending only on $\rho$ satisfies $G_\infty(\rho) = c \cdot \rho^{-2}$ for some universal constant $c > 0$. Thus, the limit distance function is exactly:
\begin{equation}
    b_\infty(\rho) = G_\infty(\rho)^{-1/2} = c^{-1/2} \rho.
\end{equation}
Since $\rho$ is the distance function in a metric cone, its gradient norm is identically 1 almost everywhere away from the vertex. Hence:
\begin{equation}
    |\nabla b_\infty(y_\infty)|_{g_\infty} = |c^{-1/2} \nabla \rho| = c^{-1/2} > 0.
\end{equation}
This contradicts our assumption that $\lim_{j \to \infty} |\tilde{\nabla} \tilde{b}_j(y_j)|_{\tilde{g}_j} = 0$, and thus finishes the proof of the lower bound. 
\end{proof} 
\begin{lemma}[\cite{JN}, Lemma 5.13]\label{evolution}
Given \((M^n, g, q)\) with \(\cN \subseteq B_2(q)\) a \((\delta, \tau)\)-neck and \(b(y) = b_\cN(y)\) the associated Green's distance function, then the following hold:
\begin{enumerate}
\item[$(1)$] for every compactly supported \(f : B_4(q) \to \mathbb{R}\), we have
\[
\frac{d}{dr} \left( r^{-3} \int_{b=r} f |\nabla b|(z) \, dz \right) = r^{-3} \int_{b \leq r} \Delta f(z) \, dz;
\]
\item[$(2)$] we have
\[
r \frac{d^2}{dr^2} \left( r^{-3} \int_{b=r} f |\nabla b|(z) \, dz \right) + 3 \frac{d}{dr} \left( r^{-3} \int_{b=r} f |\nabla b|(z) \, dz \right)
= r^{-2} \int_{b=r} \Delta f |\nabla b|^{-1}(z) \, dz.\]
\end{enumerate}
\end{lemma}
 
 Now we can state our superconvexity estimate of Riemannian curvature.  To begin with,
we fix small parameters $\sigma>0,$ $0<\eta\le\eta(n),$ where $\eta(n)$ is the constant shown in Proposition \ref{Improved Kato inequality}, set
\[
f = \phi_\cN\,\left(|\Rm|+\sigma\right)^{1-\eta},
\]
where $\phi_\cN$ is the cutoff function constructed in Lemma \ref{Good cutoff fuctions}. Define $F(r)=r^{-3}\int_{b=r} f\,|\nabla b|,$ and 
$H(r)=r^{2(1-\eta)}F(r)$. 

\begin{lemma}[$H(r)$ has Compact support]\label{compact support}
Let $H(r)=r^{-1-2\eta}\int_{b=r} \phi_\cN\,\left(|\Rm| +\sigma\right)^{1-\eta}|\nabla b|$ be defined as above,  there exists  $R=R(n,p, \rv,\Lambda,B)>0$ such that $H(r)$ has compact support  in $(0,R).$ 
\end{lemma}
\begin{proof} It suffices to show  $\phi_\cN \equiv 0 $ on level set $\{b=r\}$ when $r$ is small or large enough. In fact, by construction of cutoff function (Lemma \ref{Good cutoff fuctions}), $\supp(\phi_{\cN})\subset \cN_{10^{-4}}\cap B_{19/10}$, where $\cN_{10^{-4}}=B_2(q)\setminus\bigcup_{x\in \cC}\bar{B}_{10^{-4}r_x}(x)$. Since $\cC=\cC_+$, we have $r_x>0$ on $\cC$. Note that $\cC$ is compact and $r_x$ is continous, $\cN_{10^{-4}}$ has a uniform positive distance to $\cC$.  The lemma then follows directly from Lemma \ref{Green's distance function estimate}.  
\end{proof} 
 
\begin{lemma}\label{superconvexity}
Let \((M^n, g)\) satisfy assumptions  $(A1),(A2),(A3)$ and $(A4)$.  
If $\delta\le\delta_0(n,p, \rv,\Lambda,B),$ then there exists $\sigma=\sigma(n)>0$ and $C=C(n,p, \rv,\Lambda,B)>0$ such that for $H(r)=r^{-1-2\eta}\int_{b=r} \phi_\cN\,\left(|\Rm| +\sigma\right)^{1-\eta}|\nabla b|$, we have the following evolution equation:
\[
r^2 H''+(4\eta-1)rH'-3\eta(1-\eta)H
= r^{1-2\eta}\int_{b=r}\Delta f\,|\nabla b|^{-1}\ge -e(r),
\]
where \begin{equation}\label{error}
\begin{aligned}
e(r)&=r^{1-2\eta}\int_{b=r} \left(\left(|\Rm| +\sigma\right)^{1-\eta}|\Delta\phi_N| +2|\nabla\phi_N||\nabla\left(|\Rm|+\sigma\right)^{1-\eta}|\right)|\nabla b|^{-1} \\
&\quad +  C r^{1-2\eta}\int_{b=r}\phi_N\left(\left(|\Rm|+\sigma\right)^{-1-\eta}|\nabla \Ric|^2+ \left(|\Rm|+\sigma\right)^{-\eta}|\nabla^2\Ric|\right) |\nabla b|^{-1}.
\end{aligned}
\end{equation}
\end{lemma}
\begin{proof}
Applying Lemma \ref{evolution} to $F(r)=r^{-3}\int_{b=r} f\,|\nabla b|$ with $f = \phi_\cN\,\left(|\Rm| +\sigma\right)^{1-\eta}$, and by definition of $H(r)$, we have
\[
r^2 H''+(4\eta-1)rH'-4\eta(1-\eta)H
= r^{1-2\eta}\int_{b=r}\Delta f\,|\nabla b|^{-1}.
\]
Write $G=\left(|\Rm| +\sigma\right)^{1-\eta}$, we now compute \(\Delta f\) as follows:
\begin{equation}\label{delta f}
\Delta f = G\Delta\phi_\cN + 2\langle\nabla\phi_\cN,\nabla G\rangle + \phi_\cN\,\Delta G.
\end{equation}
where
\[\nabla G=(1-\eta)\nabla\left(|\Rm| +\sigma\right)^{1-\eta} ,
\] and
\[
\Delta G = (1-\eta)\left(|\Rm| +\sigma\right)^{-\eta}\Delta|\Rm| - \eta(1-\eta) \left(|\Rm| +\sigma\right)^{-\eta-1}|\nabla|\Rm||^2.
\]

By  Proposition \ref{laplacian of Rm} we have
\[
\frac12\Delta|\Rm|^2 = |\nabla \Rm|^2 + \langle \Rm,\nabla^2\mathrm{\Ric}\rangle + \langle \Rm,\mathrm{\Ric}*\Rm\rangle + \langle \Rm,\Rm*\Rm\rangle.
\]
Since $\displaystyle\Delta|\Rm| = \frac{\Delta|\Rm|^2}{2|\Rm|} - \frac{|\nabla|\Rm||^2}{|\Rm|},$ we can estimate
\begin{align*}
\Delta|\Rm| &\ge \frac{|\nabla \Rm|^2-|\nabla |\Rm||^2}{|\Rm|} + \frac{\langle \Rm,\nabla^2\mathrm{\Ric}\rangle}{|\Rm|}
+ \frac{\langle \Rm,\mathrm{\Ric}*\Rm\rangle}{|\Rm|}
+ \frac{\langle \Rm,\Rm*\Rm\rangle}{|\Rm|}\\
 &\ge \frac{|\nabla \Rm|^2-|\nabla |\Rm||^2}{|\Rm|}-|\nabla^2 \Ric|-C(n)\left(|\Ric|+|\Rm|\right)|\Rm|\\
 &\ge \frac{|\nabla \Rm|^2-|\nabla |\Rm||^2}{|\Rm|}-|\nabla^2 \Ric|-C(n)|\Rm|^2,
 \end{align*} where we have used the fact that 
$|\Ric|\le C(n)|\Rm|$ and $|\Ric*\Rm|+|\Rm*\Rm|\le C(n)|\Rm|^2.$ It then follows that 
\begin{equation}
\begin{aligned}\label{Delta}
 \Delta G&\ge(1-\eta)\left(|\Rm|+\sigma\right)^{-\eta}\left(\left(|\Rm|+\sigma\right)^{-1}\left(|\nabla \Rm|^2-(1+\eta)|\nabla |\Rm||^2  \right) - |\nabla^2 \Ric|-C(n)\left(|\Rm|+\sigma\right)^{2} \right)\\
 &\ge -C(n)\left(\left(|\Rm|+\sigma\right)^{-1-\eta}|\nabla \Ric|^2+\left(|\Rm|+\sigma\right)^{-\eta} |\nabla^2 \Ric|+\left(|\Rm|+\sigma\right)^{2-\eta}\right),
 \end{aligned}
 \end{equation}
where we have used the improved Kato inequality (Proposition \ref{Improved Kato inequality}) in the last line.
  Plug (\ref{Delta}) into (\ref{delta f}), we have 
\begin{align*}
\Delta f &\ge G\Delta\phi_\cN  + 2\langle\nabla\phi_\cN,\nabla G\rangle-C(n)\phi_\cN\left(\left(|\Rm| +\sigma\right)^{-1-\eta}|\nabla \Ric|^2+\left(|\Rm| +\sigma\right)^{-\eta} |\nabla^2 \Ric|+\left(|\Rm| +\sigma\right)^{2-\eta} \right)\\
&\ge- \left(|\Rm| +\sigma\right)^{1-\eta}|\Delta\phi_\cN| -2|\nabla\left(|\Rm| +\sigma\right)^{1-\eta}||\nabla\phi_\cN| \\
&\quad-C(n)\phi_\cN\left(\left(|\Rm| +\sigma\right)^{-1-\eta}|\nabla \Ric|^2+\left(|\Rm| +\sigma\right)^{-\eta} |\nabla^2 \Ric|+ \left(|\Rm| +\sigma\right)^{2-\eta} \right) .
\end{align*} 

 Finally, we conclude
\begin{align*}
r^{1-2\eta}\int_{b=r}\Delta f\,|\nabla b|^{-1} 
&\ge -r^{1-2\eta}\int_{b=r} \left(\left(|\Rm| +\sigma\right)^{1-\eta}|\Delta\phi_\cN| +2|\nabla\phi_\cN||\nabla\left(|\Rm|+\sigma\right)^{1-\eta}|\right)|\nabla b|^{-1} \\
&\quad - Cr^{1-2\eta}\int_{b=r}\phi_\cN\left(\left(|\Rm|+\sigma\right)^{-1-\eta}|\nabla \Ric|^2+ \left(|\Rm|+\sigma\right)^{-\eta}|\nabla^2\Ric|\right) |\nabla b|^{-1}\\
&\quad- Cr^{1-2\eta}\int_{b=r}\phi_\cN\left(|\Rm|+\sigma\right)^{2-\eta}|\nabla b|,
\end{align*} where we used $C^{-1}\le|\nabla b|\le C$ in the last line by Lemma \ref{Green's distance function estimate}.

Since $\phi_\cN$ is supported in $(\delta,\tau)$-neck region, which is $\Psi(\delta)$ GH-close to $(n-3)$-symmetric cone, where $\Psi(\delta)\to 0$ as $\delta\to 0$. By Strong epsilon regularity assumption (A3), if $\delta<\delta_0$ is sufficently small relative to $\eta$, we have curvature estimate 
 \[   r^2 \phi_\cN|\Rm| \le C(n)^{-1}\eta(1-\eta) \phi_\cN \text{ on $\{b=r\}$}.
 \] Now we choose $\sigma=C(n)^{-1}\eta(1-\eta) R^{-2}$, where $R$ is given by Lemma \ref{compact support}, we conclude \[C(n)r^{1-2\eta}\phi_\cN\left(|\Rm|+\sigma\right)^{2-\eta} \le \eta(1-\eta) r^{-1-2\eta}  \phi_\cN \left(|\Rm|+\sigma\right)^{1-\eta} \text{ on $\{b=r\}$}.
 \] 
Thus we obtain
\[
r^2 H''+(4\eta-1)rH'-3\eta(1-\eta)H
= r^{1-2\eta}\int_{b=r}\Delta f\,|\nabla b|^{-1}\ge -e(r),
\]where 
\begin{align*}
e(r)&=r^{1-2\eta}\int_{b=r} \left(\left(|\Rm| +\sigma\right)^{1-\eta}|\Delta\phi_\cN| +2|\nabla\phi_\cN||\nabla\left(|\Rm|+\sigma\right)^{1-\eta}|\right)|\nabla b|^{-1} \\
&\quad + Cr^{1-2\eta}\int_{b=r}\phi_\cN\left(\left(|\Rm|+\sigma\right)^{-1-\eta}|\nabla \Ric|^2+ \left(|\Rm|+\sigma\right)^{-\eta}|\nabla^2\Ric|\right) |\nabla b|^{-1}.
\end{align*}
\end{proof}
 Now that we have obtained the evolution equation satisfied by  $H(r)$, we proceed to compute a special solution of this ODE, and then apply the comparison principle to derive an upper bound estimate for  
$H(r)$.
 
\begin{lemma}\label{special solution}
Let $R > 0$ and $\eta \in (0, 1)$. Let $g(r)$ be a bounded function on the interval $(0, R)$. Consider the differential equation: \[
h''(r)+\frac{4\eta-1}{r}h'(r)-\frac{3\eta(1-\eta)}{r^2}h(r)=-g(r).
\]
Then there exists a solution $h(r)$ with $h(0)=h(R)=0$ such that
\begin{equation} 
\int_0^R\frac{h(r)}{r}\,dr = \frac{1}{ 3\eta\,(1-\eta)}\left(\int_0^R s\,g(s)\,ds - R^{\alpha}\int_0^R s^{1-\alpha}g(s)\,ds\right),
\end{equation}where $\alpha = 1 - 2\eta - \sqrt{1-\eta(1-\eta) }$.
\end{lemma} 
 
\begin{proof}
By multiplying the given differential equation by $r$, we can rewrite it  as:
\begin{equation} \label{eq:rearranged}
r h''(r) + (4\eta-1)h'(r) - \frac{3\eta(1-\eta)}{r}h(r) = -r g(r).
\end{equation}
Rearranging to isolate the term $\frac{h(r)}{r}$, we have:
\begin{equation}
\frac{h(r)}{r} = \frac{1}{3\eta(1-\eta)} \Big( r h''(r) + (4\eta-1)h'(r) + r g(r) \Big).
\end{equation}
We integrate both sides from $0$ to $R$:
\begin{equation} \label{eq:integral_eq}
\int_0^R \frac{h(r)}{r} \, dr = \frac{1}{3\eta(1-\eta)} \left( \int_0^R r h''(r) \, dr + (4\eta-1)\int_0^R h'(r) \, dr + \int_0^R r g(r) \, dr \right).
\end{equation}
Let us evaluate the first two integrals on the right-hand side.
Using integration by parts on the first integral, we obtain:
\begin{equation*}
\int_0^R r h''(r) \, dr = \Big[ r h'(r) \Big]_0^R - \int_0^R h'(r) \, dr = R h'(R) - \lim_{r \to 0^+} r h'(r) - \Big( h(R) - h(0) \Big).
\end{equation*}
For the second integral, we simply have:
\begin{equation*}
(4\eta-1)\int_0^R h'(r) \, dr = (4\eta-1) \Big( h(R) - h(0) \Big).
\end{equation*}
Since $h(0) = h(R) = 0$, both $h(R) - h(0) = 0$ terms vanish.

To evaluate $\lim_{r \to 0^+} r h'(r)$ and $h'(R)$, we recall the explicit representation of the particular solution satisfying the boundary conditions. The homogeneous equation \[r^2 h'' + (4\eta-1)r h' - 3\eta(1-\eta)h = 0\] is a Cauchy-Euler equation with characteristic roots $\gamma_1, \gamma_2= 1 - 2\eta \pm \sqrt{1-\eta(1-\eta)}$, where $\gamma_1 > 0$ and $\gamma_2 < 0$. Denote $\gamma_2$ by $\alpha$, note that $\gamma_1 + \alpha = 2 - 4\eta$ and $\gamma_1 - \alpha = 2\sqrt{1-\eta(1-\eta)}$. Using the Green's function for this boundary value problem, the solution $h(r)$ is given by:
\begin{equation} \label{eq:explicit_h}
h(r) = \frac{R^{\gamma_1-\alpha} r^\alpha - r^{\gamma_1}}{C} \int_0^r s^{1-\alpha} g(s) \, ds + \frac{r^{\gamma_1}}{C} \int_r^R \left( R^{\gamma_1-\alpha} s^{1-\gamma_1} - s^{1-\alpha} \right) g(s) \, ds,
\end{equation}
where $C = (\gamma_1 - \alpha) R^{\gamma_1-\alpha}$. 

Differentiating $h(r)$ with respect to $r$ and applying the Fundamental Theorem of Calculus, the terms involving $g(r)$ perfectly cancel out, leaving:
\begin{equation} \label{eq:h_prime}
h'(r) = \frac{\alpha R^{\gamma_1-\alpha} r^{\alpha-1} - \gamma_1 r^{\gamma_1-1}}{C} \int_0^r s^{1-\alpha} g(s) \, ds + \frac{\gamma_1 r^{\gamma_1-1}}{C} \int_r^R \left( R^{\gamma_1-\alpha} s^{1-\gamma_1} - s^{1-\alpha} \right) g(s) \, ds.
\end{equation}
Multiplying by $r$, we have:
\begin{equation*}
r h'(r) = \frac{\alpha R^{\gamma_1-\alpha} r^\alpha - \gamma_1 r^{\gamma_1}}{C} \int_0^r s^{1-\alpha} g(s) \, ds + \frac{\gamma_1 r^{\gamma_1}}{C} \int_r^R \left( R^{\gamma_1-\alpha} s^{1-\gamma_1} - s^{1-\alpha} \right) g(s) \, ds.
\end{equation*}

Since $g$ is bounded by some $M > 0$, the first integral $\int_0^r s^{1-\alpha} g(s) \, ds$ is bounded by $\frac{M}{2-\alpha} r^{2-\alpha}$. Therefore, the first term in the expression is bounded by $\mathcal{O}(r^2)$, which approaches $0$ as $r \to 0^+$. Because $\gamma_1 > 0$, the second term also vanishes since $r^{\gamma_1} \to 0$.
Thus, $\lim_{r \to 0^+} r h'(r) = 0$.

Next, we evaluate $h'(R)$ using \eqref{eq:h_prime}. The second integral evaluates from $R$ to $R$, thereby vanishing, and yielding:
\begin{align*}
h'(R) &= \frac{\alpha R^{\gamma_1-\alpha} R^{\alpha-1} - \gamma_1 R^{\gamma_1-1}}{C} \int_0^R s^{1-\alpha} g(s) \, ds + 0 \\
&= \frac{(\alpha - \gamma_1) R^{\gamma_1-1}}{(\gamma_1 - \alpha) R^{\gamma_1-\alpha}} \int_0^R s^{1-\alpha} g(s) \, ds \\
&= -R^{\alpha-1} \int_0^R s^{1-\alpha} g(s) \, ds.
\end{align*}
Therefore, $R h'(R) = -R^\alpha \int_0^R s^{1-\alpha} g(s) \, ds$.

Substituting these evaluations back into \eqref{eq:integral_eq}, we get:
\begin{equation*}
\int_0^R \frac{h(r)}{r} \, dr = \frac{1}{3\eta(1-\eta)} \left( R h'(R) + \int_0^R r g(r) \, dr \right)
\end{equation*}
\begin{equation*}
\int_0^R \frac{h(r)}{r} \, dr = \frac{1}{3\eta(1-\eta)} \left( -R^\alpha \int_0^R s^{1-\alpha} g(s) \, ds + \int_0^R s g(s) \, ds \right).
\end{equation*}
Rearranging the terms inside the parentheses to match the exact statement of the Lemma yields:
\[
\int_0^R\frac{h(r)}{r}\,dr = \frac{1}{3\eta\,(1-\eta)}\left(\int_0^R s\,g(s)\,ds - R^{\alpha}\int_0^R s^{1-\alpha}g(s)\,ds\right),
\]
which concludes the proof.
\end{proof}
\begin{theorem}[Dini estimates]\label{Dini estimates}
 Let \((M^n, g)\) satisfy assumptions  $(A1),(A2),(A3)$ and $(A4)$.  If $\delta<\delta_0(n,p, \rv,\Lambda,B)$, there exists $C(n,p, \rv,\Lambda,B)>0$ such that for \[H(r)=r^{-1-2\eta}\int_{b=r} \phi_\cN\,\left(|\Rm|+\sigma\right)^{1-\eta}|\nabla b|,\] we have the Dini estimate:
\[
\int_0^\infty \frac{H(r)}{r} \, dr \le C.
\]
\end{theorem}
\begin{proof}
By Lemma \ref{compact support}, there exists $R>0$ such that $\phi_\cN \equiv 0 $ on $ \{b \geq R\} $. Therefore, we have
\[
\int_0^\infty \frac{1}{r} H(r) \, dr = \int_0^R \frac{1}{r} H(r) \, dr.
\]
Now let us begin by using Lemma~\ref{special solution} to choose a special solution \( h(r) \) of the following ODE:
\[  
h''(r)+\frac{4\eta-1}{r}h'(r)-\frac{3\eta(1-\eta)}{r^2}h(r)=-\frac{e(r)}{r^2}, \quad h(0)=h(R)=0,
\] where $e(r)$ is given by $(\ref{error})$.
By Lemma~\ref{superconvexity}, we have 
\[
r^2 H''+(4\eta-1)rH'-3\eta(1-\eta)H \ge -e(r),
\]
and by Lemma \ref{compact support},   $H(0)=H(R)=0.$ Then we we conclude 
\begin{equation}r^2 (H-h)''(r)+(4\eta-1)r (H-h)'(r)-3\eta(1-\eta)(H-h)(r) \ge0, \label{ODI}
\end{equation}with $(H-h)(0)=(H-h)(R)=0.$

 We now claim $H(r)\le h(r)$ in $(0,R)$. In fact, if there exists some  $r_0\in(0,R)$ such that $(H-h)(r)$ attains its maximum at $r_0$ with $(H-h)(r_0)>0$, then 
$ (H-h)'(r_0)=0$  and $(H-h)''(r_0)\le 0$.
  It follows that 
\[ r_0^2 (H-h)''(r_0)+(4\eta-1)r_0 (H-h)'(r_0)-3\eta(1-\eta)(H-h)(r_0)\le -3\eta(1-\eta)(H-h)(r_0)<0,
\]  which contradicts (\ref{ODI}).

Recall from Lemma \ref{special solution}, we have
\[\int_0^R\frac{h(r)}{r}\,dr = \frac{1}{3\eta\,(1-\eta)}\left(\int_0^R \frac{e(s)}{s}\,ds - R^{\alpha}\int_0^R \frac{e(s)}{s^{1-\alpha}}\,ds\right)\le \frac{1}{3\eta(1-\eta)}\int_0^R\frac{e(s)}{s}ds.\]
Now we can estimate 
\begin{align*}
 \int_0^R \frac{H(r)}{r} dr&\le \int_0^R \frac{h(r)}{r} dr\le  \frac{1}{3\eta(1-\eta)}\int_0^R\frac{e(r)}{r}dr.
\end{align*}

From  expression $(\ref{error})$ and coarea formula,  
we can write it as \[\int_0^R \frac{e(r)}{r}dr= e_1+e_2+e_3,\] where 
\[\begin{aligned}
e_1&=\int_0^R r^{-2\eta} \int_{b=r} \left(|\Rm| +\sigma\right)^{1-\eta}|\Delta \phi_\cN| |\nabla b|^{-1}
     = \int_{b \leq R} b^{-2\eta} \left(|\Rm| +\sigma\right)^{1-\eta}|\Delta \phi_\cN|,\\
e_2&=2\int_0^R r^{-2\eta} \int_{b=r} |\nabla\phi_\cN||\nabla\left(|\Rm| +\sigma\right)^{1-\eta}||\nabla b|^{-1}
     = 2\int_{b \leq R} b^{-2\eta} |\nabla\phi_\cN||\nabla\left(|\Rm| +\sigma\right)^{1-\eta}|,\\
     e_3 &= C(n)\int_0^R r^{-2\eta}\int_{b=r}\phi_\cN\left(\left(|\Rm| +\sigma\right)^{-1-\eta}|\nabla \Ric|^2+ \left(|\Rm| +\sigma\right)^{-\eta}|\nabla^2\Ric|\right) |\nabla b|^{-1}\\
     &= C(n)\int_{b \leq R} b^{-2\eta}\phi_\cN\left(\left(|\Rm| +\sigma\right)^{-1-\eta}|\nabla \Ric|^2+ \left(|\Rm| +\sigma\right)^{-\eta}|\nabla^2\Ric| \right). 
\end{aligned}\]

We now estimate each term as follows.
To begin with, since $\phi_\cN$ is supported in $\cN_{10^{-4}}=B_2(q)\backslash \bigcup_{x\in\C}\overline{B}_{10^{-4}\cdot r_x}(x)$, by Remark \ref{remark_neck} and Lemma \ref{Green's distance function estimate}, we have $r_h(y)\approx d(y,\C)\approx b(y)$. If follows from assumption  (A3)  that if $\delta<\delta_0$,  we have curvature derivative estimates: 
 \[\begin{aligned}|\nabla \left(|\Rm| +\sigma\right)^{1-\eta}|&\le(1-\eta)|\Rm|^{-\eta}|\nabla \Rm|\le Cb^{-1}|\Rm|^{1-\eta},\\
 |\nabla \Ric|^2&\le Cb^{-2}|\Ric|^2\le Cb^{-2}|\Ric||\Rm| ,\\
 |\nabla^2 \Ric|&\le Cb^{-2}|\Ric|\\
 |\Ric|&\le b^{-1}  .\end{aligned}\] Combing above estimates, we have
 \[\begin{aligned}
 e_1&\leq \int_{b \leq R} b^{-2\eta} (Cb^{-2})^{1-\eta}|\Delta \phi_\cN|\le C\int_{b \leq R} b^{-2} |\Delta \phi_\cN| ,\\
e_2&\leq C\int_{b \leq R} b^{-2\eta} |\nabla\phi_\cN|\left(b^{-1}|\Rm|^{1-\eta}\right)\leq  C\int_{b \leq R} b^{-2\eta} |\nabla\phi_N|\left(b^{-1}b^{2\eta-2}\right)=C\int_{b \leq R} b^{-3} |\nabla\phi_\cN| ,\\ 
e_3&\le C\int_{b \leq R} b^{-2\eta}\phi_\cN\left( \left(|\Rm| +\sigma\right)^{-\eta}  b^{-2}|\Ric|   \right)\le C\sigma^{-\eta}\int_{b \leq R} b^{-3-2\eta}\phi_\cN .
\end{aligned}\]  
Combining all these we arrive at
\begin{equation}\label{integral}
\int_0^\infty \frac{H(r)}{r} \, dr \leq C \left(  \int_{b \leq R} b^{-3-2\eta}\phi_\cN+ \int_{b \leq R}b^{-3}|\nabla\phi_\cN| + \int_{b \leq R} b^{-2}|\Delta\phi_\cN| \right).  
\end{equation}

We can now follow  the same approach of \cite{JN} to bound each term of $(\ref{integral})$. By  Lemma \ref{Good cutoff fuctions}, Lemma \ref{Green's distance function estimate} and assumption (A4)   we have estimate
\[
\begin{aligned}
&\int_{b \leq R} b^{-2}|\Delta \phi_\cN|(z) \, dz +\int_{b \leq R} b^{-3}|\nabla \phi_\cN|(z) \, dz +\int_{b \leq R} b^{-3-2\eta}(z)\phi_\cN(z) dz \\
&\leq C\, \left( \int_{d_\cC \leq 3} d_\cC^{-2}|\Delta \phi_\cN|(z) \, dz + \int_{d_\cC \leq 3} d_\cC^{-3}|\nabla \phi_\cN|(z) \, dz + \int_{d_\cC \leq 3} d_\cC^{-3-2\eta}\phi_\cN(z)dz \right) \\
&\leq C\,  \Biggl( \sum_{x \in \cC \cap B_{19/10}(p)} \int_{B_{r_x}(x)} \left( d_\cC^{-2}|\Delta \phi_\cN| + d_\cC^{-3}|\nabla \phi_\cN|\right)(z) \, dz + \int_{\{d_\cC \leq 3\} \cap \operatorname{supp} \phi_\cN}   d_\cC^{-3-2\eta}(z) \, dz \Biggr) \\
&\leq C\,   \Biggl( \sum_{x \in \cC \cap B_{19/10}(p)} r_x^{n-4} + \int_{\{d_\cC \leq 3\} \cap \operatorname{supp} \phi_\cN} d_\cC^{ -3-2\eta}(z) \, dz \Biggr) \\
&\leq C\,   \mu_\cN(B_{19/10}(q)) + C(n, \rv, B,\tau) \int_{\{d_\cC \leq 3\} \cap \operatorname{supp} \phi_\cN} d_\cC^{-3-2\eta}(z) \, dz \\
&\leq C  + C\,  \int_{\{d_\cC \leq 3\} \cap \operatorname{supp} \phi_\cN} d_\cC^{-3-2\eta}(z) \, dz,
\end{aligned}  
\]

We can compute that
\begin{equation}\label{dC}
\begin{aligned}
\int_{\{d_\cC \leq 3\} \cap \operatorname{supp} \phi_\cN} d_\cC^{ -3-2\eta}(z) \, dz
&\leq \sum_{i=0}^\infty \int_{\operatorname{supp} \phi_\cN \cap \{2^{-i} \leq d(x, \mathcal{C}) \leq 2^{-i+1}\}} d_\cC^{ -3-2\eta}(z) \, dz \\
&\leq \sum_{i=0}^\infty 2^{(3+2\eta)i} \operatorname{Vol}\bigl( \operatorname{supp} \phi_\cN \cap \{2^{-i} \leq d(x, \mathcal{C}) \leq 2^{-i+1}\} \bigr).
\end{aligned} 
\end{equation}

Choose a Vitali covering \(\{B_{2^{-i+3}}(x_\alpha),\; x_\alpha \in\cC,\; \alpha = 1, \dots, N\}\) of \(\operatorname{supp} \phi_\cN \cap \{2^{-i} \leq d(x, \cC) \leq 2^{-i+1}\}\) such that \(\{B_{2^{-i}}(x_\alpha),\; x_\alpha \in \cC,\; \alpha = 1, \dots, N\}\) are pairwise disjoint. We get
\[
C  \geq \mu(B_{19/10}(q)) \geq \mu(\operatorname{supp} \phi_\cN) \geq \sum_{\alpha=1}^N \mu(B_{2^{-i}}(x_\alpha)) \geq  C\,  N 2^{-(n-4)i} 6^{-(n-4)}. 
\]
Therefore, we have \(N \leq C\, 2^{(n-4)i}\). Thus
\begin{equation}\label{vol}\begin{aligned}
\operatorname{Vol}\bigl( \operatorname{supp} \phi_\cN \cap \{2^{-i} \leq d(x, \cC) \leq 2^{-i+1}\} \bigr)
\leq \sum_{\alpha=1}^N \operatorname{Vol}(B_{2^{-i+3}}(x_\alpha)) \leq C\, 2^{-4i}.  
\end{aligned}
\end{equation}

Combining $(\ref{dC})$ and $(\ref{vol})$ we get
\[
\int_{\operatorname{supp} \phi_{\cN}} d_\cC^{-3}(z) \, dz\leq C\sum_{i=0}^\infty 2^{(3+2\eta)i -4i}= C\sum_{i=0}^\infty \frac{1}{2^{ (1-2\eta)i}} \leq  C.  
\]

Hence we arrive at
\[
\int_{b \leq R} b^{-2}|\Delta\phi_\cN|(z) \, dz + \int_{b \leq R} b^{-3}|\nabla\phi_\cN|(z) \, dz  +\int_{b \leq R} b^{-3-2\eta}(z)\phi_\cN(z) dz \leq C.  
\]
Thus we have proved the Lemma.\end{proof}

\begin{proof}[Proof of Theorem \ref{weak L2 on neck}:] 
 Since the estimates are scaling invariant, without loss
of generality we can assume $r = 1$. 
By Lemma \ref{Dini estimates}, we have estimate
\begin{equation}\begin{aligned} \label{proof_dini} \int_{0}^\infty r^{-2-2\eta}\int_{b=r}\phi_\cN \left(|\Rm|+\sigma\right)^{1-\eta}|\nabla b|=\int_0^\infty\frac{H(r)}{r}dr\le C.\end{aligned}
\end{equation} 
 Observe that 
\[\int_{\cN_{10^{-5}}\cap B_1(x)}|\Rm|^2\le     C\int_M  b^{-2-2\eta}\phi_\cN \left(|\Rm|+\sigma\right)^{1-\eta}.
\]  
By coarea formula, Lemma (\ref{Green's distance function estimate}) and $(\ref{proof_dini})$ we have 
\begin{equation}
\begin{aligned}
 \int_M b^{-2-2\eta}\phi_\cN \left(|\Rm|+\sigma\right)^{1-\eta}&=
     \int_0^\infty r^{-2-2\eta}\int_{b=r}\phi_\cN  \left(|\Rm|+\sigma\right)^{1-\eta} |\nabla b|^{-1}\\
     &\le C\int_{0}^\infty r^{-2-2\eta}\int_{b=r}\phi_\cN \left(|\Rm|+\sigma\right)^{1-\eta}|\nabla b|\\ & \le C.
\end{aligned}
\end{equation}
This proves \[\int_{\cN_{10^{-5}}\cap B_1(x)}|\Rm|^2\le C.
\] 
\end{proof} 
\section{Ahlfors regularity}
\subsection{Splitting functions on neck regions}
This subsection is intended to study the behavior of splitting functions on neck regions, and will be the key technical ingredient in the proof of Ahlfors regularity (A4) used in previous sections. The proof of the main Theorem \ref{Splitting functions on neck regions} in this section follows the same   argument in \cite{JN}. In fact, we will provide a simplified proof of Dini estimate (Theorem \ref{hessian Dini}),   due to stronger curvature assumption (A3).
\begin{theorem}[\cite{JN}, Theorem 5.2]\label{Splitting functions on neck regions}
Let \((M^n, g)\) satisfy assumptions  $(A1), (A2), (A3)$ and $(A4)$. Let $u:B_3(q)\to \mathbb{R}^{n-4}$ be a $\delta'$-splitting function,  for every $\epsilon, B>$  0 and $\tau < \tau (n)$, if $\delta,\delta'\le\delta(n, p,\rv,\Lambda, \tau, B)$,  then there exists a subset \(\mathcal{C}_\epsilon \subseteq \mathcal{C} \cap B_1\) such that

\begin{enumerate}
\item[$(1)$] \(\mu((\mathcal{C} \cap B_1) \setminus \mathcal{C}_\epsilon) < \epsilon\);
\item[$(2)$] for each \(x, y \in \mathcal{C}_\epsilon\), we have that \(1 - \epsilon < \frac{|u(x) - u(y)|}{d(x,y)} < 1 + \epsilon\);
\item[$(3)$] for each \(x \in \mathcal{C}_\epsilon\) and \(r_x \leq r \leq 1\), we have that \(u : B_r(x) \to \mathbb{R}^{n-4}\) is an \(\epsilon\)-splitting.
\end{enumerate}
\end{theorem} 
   The following Theorem is a direct generalization from \cite{JN}, note that the condition $\int_{B_4}|\Ric|^p\le \Lambda$ with $p>n/2$  implies   $G_{\bar{\mu}}\in L^q$ for any $q<n/(n-2)$, which is sufficient for us to apply H\"older inequality to argue the same as \cite{JN}.
   
\begin{lemma}[\cite{JN}, Lemma 5.36]\label{hessian bound 1}
Under the condition of Theorem \ref{Splitting functions on neck regions}. If \( \delta, \delta' < \delta(n,p,\rv,\Lambda,B) \) and \( 0 < \tau \leq \tau(n) \), then
\[
\int_{B_{19/10}(q)} b^{-2} |\nabla^2 u|^2(z) \, dz \leq C(n, B, \rv,p).  
\]
\end{lemma}

\begin{theorem}[\cite{JN}, Theorem 5.47] Let \((M^n, g)\) satisfy conditions  $(A1), (A2), (A3)$. Assume for some \(R, C > 0\) that \(u_j : B_R(q_j) \to \mathbb{R}\) is  
a sequence of harmonic functions with \(|u_j| \leq C\). Then there exists harmonic  
\(u : B_R(q) \to \mathbb{R}\) such that after possibly passing to a subsequence,

\begin{enumerate}
\item \(u_j \to u\) uniformly on compact subsets of \(B_R(q)\);
\item for each \(x_j \to x \in X\) with \(B_{2r}(x_j) \subseteq B_R(q_j)\), we have  
\[
|\nabla u|(x) \leq \liminf_j \|\nabla u_j\|_{L^\infty(B_r(x_j))};
\]
\item for each \(0 < \alpha < 4\), we have that \(|\nabla^2 u_j|^\alpha \to |\nabla^2 u|^\alpha\) in measure ---  
in fact, for each \(x_j \to x \in X\) with \(B_{2r}(x_j) \subseteq B_R(q_j)\), we have that  
\[
\int_{B_r(x_j)} |\nabla^2 u_j|^\alpha(z) \, dz \to \int_{B_r(x)} |\nabla^2 u|^\alpha(z) \, dz.
\]
\end{enumerate}
\end{theorem}
\begin{proof} Since gradient estimate for harmonic functions are valid for manifolds with integral Ricci curvature bound (see, for instance, \cite{DWZ}), (1) and (2) follows from the same argument in \cite{JN}.   Finally, (3) follows from $\ref{Lq hessian}$ and  the same argument in \cite{JN}.  
\end{proof} Once we establish  tools
of the previous theorem, we can argue by contradiction and apply blow up analysis to prove the following theorem.
\begin{theorem} [\cite{JN}, Theorem 5.44]\label{hessian bound 2}
Under the condition of Theorem \ref{Splitting functions on neck regions}, for every \(\epsilon > 0\) and \(\tau < \tau(n)\), if \(\delta, \delta' \leq \delta(n,p,\rv,\Lambda,B)\) , then we have for every \(x \in \mathcal{C}\) and \(r_x < r < 1\) that
\[
r^2 \fint_{B_r(x)} |\nabla^2 u|^2(z) \, dz < \epsilon.  
\]
\end{theorem}
\begin{remark}
This theorem has a direct corollary:  If $y\in\cN_{10^{-5}}$ with $2r = d(y, C)\approx r_h\approx b$, then for $\delta, \delta_0 \le \delta(n, \epsilon, \tau, \rv)$, we can use elliptic estimates to obtain
$r|\nabla^2u|< \epsilon$ on $B_r(y)$. We can rephrase this as the pointwise Hessian estimate
$|\nabla^2u|(y)<\epsilon r_h^{-1}(y)$ on $\cN_{10^{-5}}$.
\end{remark}

\begin{theorem}[$L^1$ summable Hessian estimate]\label{L1 hessian}
Under the condition of Theorem \ref{Splitting functions on neck regions}, for every \(\epsilon, B > 0\) and \(\tau < \tau(n)\) if \(\delta, \delta' < \delta(n,p,\rv,\Lambda,B)\), then we have the estimate   
\[
\int_{N_{10^{-3}} \cap B_{18/10}(q)} b^{-3}|\nabla^2 u|(z) \, dz < \epsilon.
\]  
\end{theorem}
The proof of the above estimate also follows from the superconvexity estimate.  We will present a simplier proof due to stronger curvature assumption (A3).
\begin{lemma}[\cite{JN}, Proposition 5.65]\label{hessian superconvexity}
Under the condition of Theorem \ref{Splitting functions on neck regions}, if we define \[H(r) = r^{-2} \int_{b=r} \sqrt{|\nabla^2 u|^2 + \epsilon^2} \, \phi_\cN \, |\nabla b|(z) \, dz ,\]  then we have the following evolution equation:
\begin{equation}
\begin{aligned}
r^2 H''(r) + r H'(r) - H(r) &\geq e(r)\\&
= \int_{b=r} \left(-\sqrt{|\nabla^2 u|^2 + \epsilon^2} \, |\Delta \phi_\cN|-c(n) |\Rm| \cdot |\nabla^2 u| \phi_\cN + U_1 + U_2 \right) |\nabla b|^{-1}(z) \, dz, 
\end{aligned}
\end{equation}
where
\[
U_1 = 2 \langle \nabla \phi_\cN, \nabla \sqrt{|\nabla^2 u|^2 + \epsilon^2} \rangle,
\]
and
\[\begin{aligned}
U_2 &= \left( \nabla_j (\Ric_{ia} \nabla^a u) - \nabla_a (\Ric_{ij} \nabla^a u) + \nabla_i (\Ric_{ja} \nabla^a u) \right) * \nabla^i \nabla^j u \frac{\phi_\cN}{\sqrt{|\nabla^2 u|^2 + \epsilon^2}}\\
&= \left(\nabla \Ric*\nabla u +\Ric*\nabla^2 u\right)*\nabla^2 u \frac{\phi_\cN}{\sqrt{|\nabla^2 u|^2 + \epsilon^2}}.
\end{aligned}\]
\end{lemma}

\begin{lemma}[Special solution]\label{Special solution2}
 Let $g(r)$ be a   bounded function on the interval $(0, R)$. Consider the differential equation:
\[
h''(r) + \frac{1}{r}h'(r) - \frac{1}{r^2}h(r) =-g(r),
\]  Then there exists a solution $h(r)$ with $h(0)=h(R)=0$ such that
\[\int_0^R \frac{h(r)}{r} \, dr = \frac{1}{R} \int_0^R s(R - s) g(s)\, ds.\] 
\end{lemma}
\begin{proof}
We can directly check the specific solution satisfying the boundary value problem is:
\begin{equation}\label{sol_h}
h(r) = \frac{1}{2r} \int_0^r (s^2 - r^2) g(s) \, ds - \frac{r}{2R^2} \int_0^R (s^2 - R^2) g(s) \, ds.
\end{equation}
We now  compute: 
$I = \int_0^R \frac{h(r)}{r} \, dr$.
Using our solution \eqref{sol_h}, we can split the integral into two parts, $I = I_1 - I_2$:
\begin{align*}
I_1 &= \int_0^R \frac{1}{2r^2} \left( \int_0^r (s^2 - r^2) g(s) \, ds \right) dr, \\
I_2 &= \int_0^R \frac{1}{2R^2} \left( \int_0^R (s^2 - R^2) g(s) \, ds \right) dr.
\end{align*}

Let us evaluate $I_2$ first. The inner integral is independent of $r$, so we simply evaluate the outer integral:
\begin{equation*}
I_2 = \frac{1}{2R^2} \left( \int_0^R (s^2 - R^2) g(s) \, ds \right) \int_0^R 1 \, dr = \frac{1}{2R} \int_0^R (s^2 - R^2) g(s) \, ds.
\end{equation*}

For $I_1$, we apply Fubini's Theorem to reverse the order of integration over the triangular region $0 \le s \le r \le R$:
\begin{equation*}
I_1 = -\frac{1}{2} \int_0^R g(s) \left( \int_s^R \left( 1 - \frac{s^2}{r^2} \right) dr \right) ds.
\end{equation*}
Evaluating the inner integral with respect to $r$:
\begin{equation*}
\int_s^R \left( 1 - s^2 r^{-2} \right) dr = \left[ r + \frac{s^2}{r} \right]_{r=s}^{r=R} = \left( R + \frac{s^2}{R} \right) - (s + s) = \frac{R^2 - 2Rs + s^2}{R} = \frac{(R-s)^2}{R}.
\end{equation*}
So, $I_1$ becomes:
\begin{equation*}
I_1 = -\frac{1}{2R} \int_0^R (R - s)^2 g(s) \, ds.
\end{equation*}

Finally, we recombine $I = I_1 - I_2$:
\begin{equation*}
I = \frac{1}{2R} \int_0^R \Big( -(R-s)^2 - (s^2 - R^2) \Big) g(s) \, ds=\frac{1}{R} \int_0^R s(R - s) g(s) \, ds.
\end{equation*}  
\end{proof}
\begin{theorem}\label{hessian Dini}
Under the condition of Theorem \ref{Splitting functions on neck regions},  for every \(\epsilon, B > 0\) and \(\tau < \tau(n)\), if  \(\delta, \delta' \leq \delta(n,p,\rv,\Lambda,B)\), then we have Dini estimate:
\[\int_{0}^\infty \frac{H(r)}{r}dr\le C\epsilon.
\]\end{theorem}
\begin{proof} It is clear that there exists $R>0$ such that $\phi_\cN \equiv 0 $ on $ \{b \geq R\} $. Therefore, we have
\[
\int_0^\infty \frac{1}{r} H(r) \, dr = \int_0^R \frac{1}{r} H(r) \, dr.
\]
 By Lemma~\ref{Special solution2}, we can choose a special solution \( h(r) \) of the following ODE: \[
h''(r) + \frac{1}{r}h'(r) -\frac{1}{r^2}h(r) =\frac{e(r)}{r^2}, \quad h(0)=h(R)=0,
\]  where $e(r)$ is given by Lemma \ref{hessian superconvexity}. By the same argument of Theorem \ref{Dini estimates}, we have $H(r)\le h(r)$ on $(0,R)$. By Lemma~\ref{Special solution2} we have estimate
\[\int_0^R \frac{h(r)}{r} \, dr = \frac{1}{R} \int_0^R (R - s)\frac{e(s)}{s}\, ds\le \int_0^R \frac{|e(r)|}{r}dr.\]
By coarea we can estimate \[\int_0^ R \frac{|e(r)|}{r}dr\le e_1+e_2+e_3+e_4\] as follows, where 

\[\begin{aligned}
e_1&=\int_0^R r^{-1} \int_{b=r}  \sqrt{|\nabla^2 u|^2 + \epsilon^2} \, |\Delta \phi_\cN| |\nabla b|^{-1}  = \int_{b \leq R} b^{-1}  \sqrt{|\nabla^2 u|^2 + \epsilon^2} \, |\Delta \phi_\cN|,\\
e_2&=c(n)\int_0^R r^{-1} \int_{b=r} |\Rm||\nabla^2 u|\phi_\cN |\nabla b|^{-1}=c(n) \int_{b \leq R} b^{-1}   |\Rm||\nabla^2 u|\phi_\cN ,  \\
e_3&= 2\left|\int_0^R r^{-1} \int_{b=r}   \langle \nabla \phi_\cN, \nabla \sqrt{|\nabla^2 u|^2 + \epsilon^2} \rangle|\nabla b|^{-1}\right|
     = 2\left|\int_{b \leq R} b^{-1} \langle \nabla \phi_\cN, \nabla \sqrt{|\nabla^2 u|^2 + \epsilon^2} \rangle\right|\\
&\quad \le  C \int_{b \le R} \left(b^{-1} |\Delta \phi_\cN|+b^{-2}|\nabla \phi_\cN|\right)  \sqrt{|\nabla^2 u|^2 + \epsilon^2} \, dz,\\
     e_4 &= \left|\int_0^R r^{-1}\int_{b=r} \left(\nabla \Ric*\nabla u +\Ric*\nabla^2 u\right)*\nabla^2 u \frac{\phi_\cN}{\sqrt{|\nabla^2 u|^2 + \epsilon^2}} |\nabla b|^{-1}\right|\\
     &\quad \le C(n) \int_{b \leq R} b^{-1}\left(|\nabla \Ric||\nabla u| +|\Ric||\nabla^2 u|\right)  \phi_\cN.
\end{aligned}\]
We now estimate each term as follows.
To begin with, since $\phi_\cN$ is supported in $\cN_{10^{-4}}=B_2(q)\backslash \bigcup_{x\in\C}\overline{B}_{10^{-4}\cdot r_x}(x)$, by Remark \ref{remark_neck}, we have $r_h(y)\approx d(y,\C)\approx b(y)$. If follows from assumption $(A3)$ that if $\delta<\delta_0$,  we have \( b^{2}|\Rm|+b|\Ric|\leq \epsilon \) on \( \operatorname{supp} \phi_\cN \), and  curvature derivative estimates: 
\begin{equation} |\nabla \Ric|\le \epsilon b^{-2}.
\end{equation}  It also follows from Theorem \ref{hessian bound 1} that we have $|\nabla^2 u|\le \epsilon b^{-1}$ on \( \operatorname{supp} \phi_\cN \). Combing above estimate, we have 
\[\begin{aligned}
e_1&\le \epsilon\int_{b\le R}b^{-2}|\Delta \phi_\cN|\\
e_3&\le C\epsilon \int_{b \le R} \left(b^{-2} |\Delta \phi_\cN|+b^{-3}|\nabla \phi_\cN|\right)  \, dz,\\
     e_4 &\le C(n) \int_{b \leq R} b^{-1}\left(|\nabla \Ric||\nabla u| +|\Ric||\nabla^2 u|\right)  \phi_\cN\le C\epsilon \int_{b \le R}b^{-3}\phi_\cN.
\end{aligned} \]For $e_2$, by Cauchy inequality, Theorem \ref{weak L2 on neck} and Theorem \ref{hessian bound 1}, we have 
\[e_2\le c(n)\left(\int_{b\le R}|\Rm|^2\phi_\cN\right)^{1/2}\left(\int_{b\le R}b^{-2}|\nabla^2 u|\phi_\cN \right)^{1/2}\le C\epsilon.
\]

Collecting estimates we finally arrive at 
\begin{equation} 
\int_0^\infty \frac{H(r)}{r} \, dr \leq C\epsilon \left(  \int_{b \leq R} b^{-3}\phi_\cN+ \int_{b \leq R}b^{-3}|\nabla\phi_\cN| + \int_{b \leq R} b^{-2}|\Delta\phi_\cN| \right).  
\end{equation}
By the end of the proof of Theorem \ref{Dini estimates}, we conclude \[\int_0^\infty \frac{H(r)}{r}\le C\epsilon.\]
\end{proof}
 \begin{proof}[Proof of Theorem \ref{L1 hessian}]
Let us first observe that
\[
\int_{N_{10-3} \cap B_{18/10}(p)} b^{-3}|\nabla^2 u|(z) \, dz < \int_M b^{-3}|\nabla^2 u| \phi_\cN(z) \, dz,
\] By Theorem \ref{hessian Dini}
\[
\begin{aligned}
\int_0^\infty r^{-3} \int_{b=r}  |\nabla^2 u| \phi_\cN(z) \, dz dr 
&< \int_0^\infty r^{-3} \int_{b=r}  \sqrt{|\nabla^2 u|^2 + \epsilon'^2} \phi_\cN(z) \, dz dr \\
&< C(n, B, v, \tau) \epsilon'.
\end{aligned}
\]
By using the coarea formula we can rewrite this as
\[
\int_M b^{-3} |\nabla b| |\nabla^2 u| \phi_\cN(z) \, dz < \epsilon'.
\]
Finally, using the Green's function estimates \(|\nabla b| > C^{-1}\), this gives us the estimate
\[
\int_M b^{-3} |\nabla^2 u| \phi_\cN(z) \, dz < C\epsilon' < \epsilon.
\]\end{proof}

The proof of Theorem \ref{Splitting functions on neck regions} will be a consequence of the so called $\mu$-maximal function estimate (Theorem \ref{maximal function estimate}). To begin with, we need the following telescoping type estimate that allows us to compare the average of a function along
different nearby balls. The proof only relies on Poincar\'e inequality, which were established for space with integral curvature bound by \cite{DWZ}.
\begin{lemma} 
Let \( (M^n, g) \) satisfy  $(A1)$, and let \( f : B_4(q) \to \mathbb{R} \) be an \( $W$^{1,2} \)-function. Then if \( B_{r_1}(x_1), B_{r_2}(x_2), B_{r_3}(x_3) \subseteq B_2(q) \) satisfy \( r_1, r_2 > 10^{-n}r_3 \) and are such that \( B_{2r_1}(x_1), B_{2r_2}(x_2) \subseteq B_{r_3}(x_3) \), we have the estimate

\[
\left| \fint_{B_{r_1}(x_1)} f(z) \, dz - \fint_{B_{r_2}(x_2)} f(z) \, dz \right| < C(n,p,\Lambda) \, r_3 \fint_{B_{r_3}(x_3)} |\nabla f|(z) \, dz.
\]
\end{lemma}
Given above Lemma and hessian bound \ref{hessian bound 2}, by the same proof of Jiang-Naber \cite{JN}, we have
\begin{theorem}[\cite{JN}, Theorem 5.104]
Under the condition of Theorem \ref{Splitting functions on neck regions}, for every \(x \in \mathcal{C} \cap B_1(q)\) and \(r_x < r < 2/3\), we have that

\[
\fint_{B_r(x)} \bigl| \langle \nabla u_a, \nabla u_b \rangle - \delta_{ab} \bigr|(z) \, dz < \epsilon + C(n, \rv,p,\Lambda) \fint_{W_x} r_h^{1-n}|\nabla^2 u|(z) \, dz,
\]
where $W_x \equiv \{ y \in A_{10^{-2}r_x, 1}(x) : d(y, x) < 2 \cdot d(y, \cC) \}$ is the wedge region.
\end{theorem}
 
\begin{proposition}[\cite{JN}, Proposition 5.119]\label{maximal function estimate}
Under the condition of Theorem \ref{Splitting functions on neck regions}, for every \(\epsilon, B > 0\) and \(\tau < \tau(n)\), if \(\delta, \delta' \leq \delta(n,p,\rv,\Lambda,B)\) and if we define the \(\mu\)-supported maximal function \(m_u : \mathcal{C} \to \mathbb{R}^+\) by
\[
m(x) \equiv \sup_{r_x \leq r \leq 2/3} \left( \int_{B_r(x)} \bigl| \langle \nabla u_a, \nabla u_b \rangle - \delta_{ab} \bigr|(z) \, dz + r^2 \int_{B_r(x)} |\nabla^2 u|^2(z) \, dz \right),
\]
then we have the estimate
\[
\int_{B_1(p)} m(x) \, d\mu(x) < \epsilon.
\]
\end{proposition}
\begin{proof}[Proof of Theorem \ref{Splitting functions on neck regions}] This follows directly from the previous Proposition \ref{maximal function estimate} and the same proof of Jiang-Naber. 
\end{proof}
\subsection{Proof of Ahlfors regularity}

\begin{theorem}\label{strong L2 neck}
Let \((M^n, g)\) satisfy assumptions $(A1),  (A2), (A3)$. Then if $\delta\le\delta(n,p,\rv,\Lambda)$, then  
\begin{enumerate}
\item there exists $B(n,p,\rv,\Lambda)>0$ such that for each \(x \in \mathcal{C}\) and \(r_x < r\) with \(B_{2r}(x) \subseteq B_2\), we have\[B^{-1}r^{n-4} < \mu(B_r(x)) < Br^{n-4},\]
\item we have $L^4$ bound on neck region \[\int_{\cN\cap B_1(p)}|\Ric|^4< C(n,p,\rv,\Lambda).\]
\item we have $L^2$ bound on neck region \[\int_{\cN\cap B_1(p)}|\Rm|^2< C(n,p,\rv,\Lambda).\]
\end{enumerate}
 \end{theorem}
 \begin{proof} To begin with, note (2) and (3) of Theorem \ref{strong L2 neck} follows from Theorem \ref{weak L2 on neck} once we have proved Ahlfors regularity (1). 
 
  Ahlfors lower bound follows from the same proof in section 6.1 in \cite{JN} and Theorem \ref{Reifenberg map} while Ahlfors upper bound follows from the same proof in section 6.2 in \cite{JN} and Theorem \ref{Splitting functions on neck regions}.
 \end{proof} 
 
 \section{Neck decomposition}
 Our main goal in this section is to prove the following neck decomposition theorem. This is one of the most technical parts in this paper, and the proof here follows almost from the same line as in \cite{JN}. For the sake of completeness, we give the detailed proof of this  decomposition.
 \begin{theorem}[Neck decomposition] \label{neck decomposition}  
Let $(M^n,g)$  satisfy assumptions $(A1),(A2)$ and $(A3)$ above, then for  $\tau\le\tau_0(n)$, $\delta\le\delta(n,\mathrm{v},p,\Lambda)>0$ and $x\in M$, we can write
\[
B_1(x) \subset\bigcup_{a} \left(\N_a\cap B_{r_a}(x_a)\right) \cup \bigcup_{b} B_{r_b}(x_b),
\]
where
\begin{enumerate}
    \item each $\N_a$ is a $(\delta,\tau)$-neck region contained in some ball $B_{2r_a}(x_a)$;
    \item each $B_{r_b}(x_b)$  satisfies $r_h(x_b) > 2r_b$, where $r_h$ is harmonic radius defined in section $\ref{harmonic radius}$; 
    \item the following content bound holds:
        \[
        \sum_a r_a^{n-4} + \sum_b r_b^{n-4} \le C(n,\mathrm{v},p,\Lambda).
        \]
\end{enumerate}\end{theorem}
 
The above neck decomposition developed by Jiang-Naber  is fundamentally a combinatorial, metric-measure theoretic covering argument. Once we have established that the integral Ricci curvature bound (A1) and  uniform non-collapsing
(A2) implies quantitative splitting (Theorem \ref{almost splitting}), and local metric cone structure (Theorem \ref{volume cone}), together with  strong $\vepsilon$-regularity (A3)  which says “generic” singularities must have exactly codimension 4, the entirety of Section 9 in \cite{JN} acts as a formal ``black box” algorithm. It merely categorizes
metric balls and counts them.

Let us recall the volume ratio $\displaystyle \cV_r(x):=\frac{\Vol(B_r(x)}{\Vol(B_r(0^n))}$. Define $\bar{V} \equiv \inf_{y\in B_4(q)}\cV_{\eta^{-1}}$, and volume pinching set 
\[ \mathcal{E}_{\eta}(x,r)=\{y \in B_{4r}(x):\cV_{\eta r}(y)<\bar{V}+\eta\}.
\] Let us now consider the various ball types that will appear in the proof. The conventions below will allow us to keep oriented throughout the proof of the various conditions we will be forced to consider. After defining the sets carefully we will discuss their meaning in words:

\begin{itemize}
\item[(a)] a ball \( B_{r_a}(x_a) \) with subscript \( a \) will be such that there exists a \( (\delta, \tau) \)-neck region \( \cN_a \subseteq B_{2r_a}(x_a) \);

\item[(b)] a ball \( B_{r_b}(x_b) \) with subscript \( b \) will be such that \( r_h(x_b) \geq 2r_b \);

\item[(c)] a ball \( B_{r_c}(x_c) \) with subscript \( c \) will be such that \( \operatorname{Vol}(B_{\delta r_c} \mathcal{E}_\eta(x_c, r_c)) > \epsilon \delta^4 r_c^n \);

\item[(d)] a ball \( B_{r_d}(x_d) \) with subscript \( d \) will be such that \( \operatorname{Vol}(B_{\delta r_d} \mathcal{E}_\eta(x_d, r_d)) \leq \epsilon \delta^4 r_d^n \) and \( \mathcal{E}_\eta(x_d,r_d) \neq \varnothing\);

\item[$(\tilde{\mathrm{d}})$] a ball \( B_{r_{\tilde{d}}}(x_{\tilde{d}}) \) with subscript \( \tilde{d} \) will be such that \( \operatorname{Vol}(B_{\delta r_{\tilde{d}}} \mathcal{E}_\eta(x_{\tilde{d}}, r_{\tilde{d}})) \leq \epsilon \delta^4 r_{\tilde{d}}^n \);

\item[(e)] a ball \( B_{r_e}(x_e) \) with subscript \( e \) will be such that \(\operatorname{Vol}(B_{\delta r_{\tilde{d}}} \mathcal{E}_\eta(x_{\tilde{d}}, r_{\tilde{d}})) \leq \epsilon \delta^4 r_{\tilde{d}}^n \) and  \( \mathcal{E}_\eta(x_e, r_e) = \varnothing\), that is, we have the volume jump \( \inf_{y \in B_{4r_e}(x_e)} V_{\eta r_e}(y) \geq \bar{V} + \eta \);

\item[(f)] a ball \( B_{r_f}(x_f) \) with subscript \( f \) will be a typical ball for which we have no a priori knowledge of additional structure.
\end{itemize}

 Let us remark that of the various ball types above, only the first two take part in the final
result of Theorem \ref{neck decomposition}. A $c$-ball is one for which there are a lot of points that are volume pinched on many scales, from an $n-4$  content point of view.  A $\tilde{d}$-ball is a $d$-ball or $e$-ball. To explain the
e-ball, let us now note that the final proof of the decomposition theorem will
be by an induction process, where we will induct on the lower volume of the
ball. Therefore, an $e$-ball is a ball for which the volume has increased by some
definite amount after dropping some scales, and thus we will be able to apply
our inductive hypothesis to deal with it.

Let us first prove cone splitting principle in  content sense.
 \begin{theorem}[Content cone-splitting]\label{Content cone-splitting}
Let \((M^n,g)\) satisfy satisfy $(A1), (A2)$ with \(0 < \delta, \zeta \leq \delta(n, \rv)\) and \(\epsilon > 0\). Then there exists \(\eta(n, \rv, \delta, \epsilon, \zeta) > 0\) such that if
\begin{enumerate}
\item [$(1)$] $\int_{B_{4}(q)}|\Ric|^{p} < \eta,$ and
\item [$(2)$]  if \(\bar{V} \equiv \inf_{y\in B_4(q)} \cV_{\eta^{-1}}(y)\) and \(\mathcal{E}_\eta(q, 1) \equiv \{y \in B_4(q) : \cV_\eta(y) < \bar{V} + \eta\}\) satisfy
\[
\operatorname{Vol}(B_\delta \mathcal{E}_\eta) > \epsilon \delta^k, 
\]
then there exists \(q' \in B_4(q)\) such that \(B_{\zeta^{-1}}(q')\) is \((n-k, \zeta^2)\)-symmetric.
\end{enumerate}
\end{theorem}

\begin{proof}
We argue by contradiction. Suppose the theorem is false. Then there exist fixed parameters $\epsilon > 0, \delta > 0, \zeta > 0$, and a sequence of  manifolds $(M_i, g_i)$ satisfying the volume uniform non-collapsing  and  curvature bound:
\begin{equation}
    \int_{B_4(q_i)} |\Ric_{g_i}|^p \le \eta_i \longrightarrow 0,
\end{equation}
such that for the associated volume pinched sets $\mathcal{E}_{\eta_i} \subset B_4(q_i)$, we have:
\begin{equation} \label{eq:vol_contradiction}
    \Vol(B_\delta(\mathcal{E}_{\eta_i})) > \epsilon \delta^k,
\end{equation}
but there does \textit{not} exist any point $q' \in B_4(q_i)$ such that $B_{\zeta^{-1}}(q')$ is $(n-k, \zeta^2)$-symmetric.

Because $p > n/2$ and the $L^p$ norm of the Ricci curvature is uniformly bounded, we can apply Theorem \ref{precompactness}. Passing to a subsequence, the manifolds $(M_i, g_i,q_i,\Vol_{g_i})$ converge in the measured Gromov-Hausdorff sense to a non-collapsed limit metric measure space $(X, d,q_\infty, \mu)$.   By volume convergence  under $L^p$ bounds (see \cite{TZ}), the measures converge to the Hausdorff measure $\mathcal{H}^n$, and the condition \eqref{eq:vol_contradiction} passes to the limit space, yielding:
\begin{equation} \label{eq:limit_content}
    \mathcal{H}^n(B_\delta(\mathcal{E}_\infty)) \ge \epsilon \delta^k,
\end{equation} where we pass the sets $\mathcal{E}_{\eta_i}$ to a limit subset $\mathcal{E}_\infty \subset B_4(q_\infty) \subset X$.

Now, for any point $y \in \mathcal{E}_\infty$. By the definition of $\mathcal{E}_{\eta_i}$, the  volume ratio satisfies $\mathcal{V}_{\eta_i}(y_i) < \inf \mathcal{V}_{\eta_i^{-1}} + \eta_i$. Since $\eta_i \to 0$, taking the limit implies that the volume ratio $\mathcal{V}_r(y)$ on the space $X$ is  constant for all radii $0 < r < \infty$. 

By Theorem \ref{volume cone}, the constancy of the volume ratio at $y$ forces the limit space $X$ to be a metric cone with vertex $y$. Because this holds for every $y \in \mathcal{E}_\infty$, every point in $\mathcal{E}_\infty$ is a vertex of the cone $X$. 
Thanks to the bound \eqref{eq:limit_content}, the Hausdorff dimension of the vertex set $\mathcal{E}_\infty$ must be at least $n-k$. 

By the classical cone splitting principle, if the vertex set of a metric cone has dimension at least $n-k$, the cone must isometrically split off an $(n-k)$-dimensional Euclidean factor. Therefore:
\begin{equation}
    X \cong \mathbb{R}^{n-k} \times C(Z),
\end{equation}
for some metric space $Z$.

By the definition of pointed Gromov-Hausdorff convergence, this implies that for $i$ sufficiently large, there exists a point $q'_i \in B_4(q_i)$ corresponding to the origin of $\mathbb{R}^{n-k} \times C(Z)$ such that $B_{\zeta^{-1}}(q'_i)$ is strictly $(n-k, \zeta^2)$-symmetric.  This   contradicts our initial assumption that no such point $q_i'$ exists. Thus, the theorem holds.
\end{proof}

 \begin{theorem} 
Let \((M^n, g)\) satisfy $(A1),(A2)$ and $(A3)$ with \(\delta, \epsilon > 0\).  
Then for \(\eta \leq \eta(n, \rv, \delta, \epsilon)\) with \(\bar{V} \equiv \inf_{y\in B_4(q)} V_{\eta-1}(y)\) and \(\mathcal{E}_\eta(q, 1) \equiv \{y \in B_4(q) : V_\eta(y) < \bar{V} + \eta\}\), if we have $\int_{B_{\delta^{-1}}(q)}|\Ric|^{p} < \eta,$ for some $p>n/2$ and $\Vol(B_\delta \mathcal{E}_\eta) > \epsilon \delta^4,$ 
then either  
\begin{enumerate}
\item[$(1)$] we have that \(r_h(q) > \delta^{-1}\), or  
\item[$(2)$]  there exists \(q' \in B_4(q)\) such that \(B_{\delta^{-1}}(q')\) is pointed \(\delta\)-Gromov Hausdorff close to \(B_{\delta^{-1}}(0^{n-4}, y_c) \subseteq \mathbb{R}^{n-4} \times C(S^3/\Gamma)\), where \(\Gamma \subseteq O(4)\) is nontrivial and acts freely away from the origin with \(y_c\) the cone point.
\end{enumerate}
\end{theorem}
\begin{proof}
By Theorem \ref{Content cone-splitting} (with $k=4$),  if \(\eta\le \eta(n, \rv, \delta, \epsilon)\), the condition $\Vol(B_\delta \mathcal{E}_\eta) > \epsilon \delta^4$ guarantees that there exists $q' \in B_4(q)$ such that $B_{4\delta^{-1}}(q')$ is  $\delta$-Gromov Hausdorff close to $\mathbb{R}^{n-4}\times C(Z)$.
 Let us begin with the following:

\noindent\textbf{Claim:} For $\eta \leq \eta (n,\mathrm{v},\delta)$ sufficiently small, we have that $Z = S^{3}/\Gamma$ is a smooth three-manifold.

Indeed, assume this is not the case. Then by applying Theorem \ref{Content cone-splitting} we have a sequence $(M_{j}^{n},g_{j},q_{j})$ with $\int_{B_{\delta^{-1}}(q_j)}|\Ric|^{p} < \eta_j\to 0$ such that $B_{\delta^{-1}}(q_{j})$ is $\delta_{j}$-GH close to $\mathbb{R}^{n-4}\times C(Z_{j})$ but we do not have that $B_{4\delta^{-1}}(q_{j})$ is $\delta$-GH close to $\mathbb{R}^{n-4}\times C(S^{3}/\Gamma)$. By precompactness theorem under $L^p$ Ricci curvature bound (Theorem \ref{precompactness}), after passing to a subsequence we have that $M_{j}\to X\equiv \mathbb{R}^{n-4}\times C(Z_{\infty})$. Let us first see that $Z_{\infty}$ is a smooth manifold. Indeed, if $Z_{\infty}$ were not smooth, then we see that the singular set of $X$ has codimension at most three. However, by assumption (A3) the singular set of $X$ is at most of codimension four, which tells us that $Z_{\infty}$ is indeed smooth. Additionally, since $\eta_{j}\to 0$, we have that $X$ is Ricci flat on its regular part. It is standard for a smooth manifold $Z_{\infty}$ that $C(Z_{\infty})$ is Ricci flat if and only if $Z_{\infty}$ has constant sectional curvature $1$ which, in particular, implies that $Z_{\infty} = S^{3}/\Gamma$. This means we have that $B_{4\delta^{-1}}(q_j)$ is $\delta$-GH close to $\mathbb{R}^{n-4}\times C(S^{3}/\Gamma)$ for $j$ large enough, which is a contradiction and so proves the claim. $\square$

Finally, to finish the proof we now distinguish between the cases when $\Gamma$ is trivial or nontrivial. If $\Gamma$ is trivial, then $B_{4\delta^{-1}}(q)$ is $\delta$-GH close to $\mathbb{R}^{n}$, which by Theorem \ref{harmonic radius} implies $r_h(q) > 2\delta^{-1}$, and so $r_h(p) > \delta^{-1}$ as in case (1). On the other hand if $\Gamma$ is nontrivial, then this is exactly case (2), which finishes the proof.
\end{proof}
In order to construct neck regions with desired properties, we first construct  weak neck regions. For a detailed introduction to weak neck regions, we refer to \cite{JN}.
\begin{definition}[Weak neck region]\label{weak neck region}
We call \(\tilde{\cN} \subseteq B_2(q)\) a weak \((\delta, \tau)\)-neck region if there exists a closed subset \(\tilde{\cC} = \tilde{\cC}_0 \cup \tilde{\cC}_+ = \tilde{\cC}_0 \cup \{x_i\} \subset B_2(q)\) with \(q \in \tilde{\cC}\) and a radius function \(r : \tilde{\cC} \to \mathbb{R}^+\) with \(r_x > 0\) on \(\tilde{\cC}_+\) and \(r_x = 0\) on \(\tilde{\cC}_0\) such that \(\tilde{\cN} \equiv B_2 \setminus \overline{B}_{r_x}(\mathcal{C}) \) satisfies

\begin{itemize}
\item[\((\tilde{n}1)\)] \(\{B_{\tau^2 r_x}(x)\}\) are pairwise disjoint;
\item[\((\tilde{n}2)\)] for each \(r_x \leq r \leq 1\), there exists a \(\delta r\)-GH map \(\iota_{x,r} : B_{\delta^{-1}r}(0^{n-4}, y_c) \subseteq \mathbb{R}^{n-4} \times C(S^3/\Gamma) \to B_{\delta^{-1}r}(x)\), where \(\Gamma \subseteq O(4)\) is nontrivial;
\item[\((\tilde{n}3)\)] for each \(r_x \leq r\) with \(B_{2r}(x) \subseteq B_2(q)\), we have that \(\mathcal{L}_{x,r} \equiv \iota_{x,r}(B_r(0^{n-4}) \times \{y_c\}) \subseteq B_{\tau(r+r_y)}(\mathcal{C})\) and \(\mathcal{C} \cap B_r(x) \subseteq B_{\delta r}(\mathcal{L}_{x,r})\), where \(B_{\tau(r+r_y)}(\mathcal{C}) = \bigcup_{y \in \mathcal{C}} B_{\tau(r+r_y)}(y)\).
\end{itemize}
\end{definition}
\begin{remark}Since we only consider smooth manifolds, the set $\tilde{\cC}_0=\varnothing$ .  
\end{remark}
\begin{proposition}[Weak $c$-ball covering]\label{weak c-ball}
 Let \((M^n, g)\) satisfy $(A1),(A2),(A3)$ with \(\delta > 0\) and \(0 < \epsilon, \tau < \tau(n)\). Then for \(\eta \leq \eta(n, \rv, \delta, \epsilon, \tau)\), let us assume the following holds:
\begin{enumerate}
\item[$(1)$] $\int_{B_{\delta^{-1}}(q)}|\Ric|^{p} < \eta$;
\item[$(2)$] if \(\bar{V} \equiv \inf_{y\in B_4(q)} \mathcal{V}_{\eta^{-1}}(y)\) and \(\mathcal{E}_\eta \equiv \{x \in B_4(q) : \mathcal{V}_\eta(x) < \bar{V} + \eta\}\), then \(\operatorname{Vol}(B_\delta \mathcal{E}_\eta) > \epsilon   \delta^4\).
\item[$(3)$] \(r_h(q) < 2\).
\end{enumerate}
Then we can write \(B_1(q) \subseteq  \tilde{\cN} \cup \bigcup_{\tilde d} B_{r_{\tilde d}}(x_{\tilde d})\), where
\begin{itemize}
\item[(a)] \(\tilde{\mathcal{N}} = B_{11}(\tilde{q} ) \setminus   \bigcup_{\tilde{d}} B_{r_{\tilde{d}}}(x_{\tilde{d}})  \) is a weak \((\delta, \tau)\)-neck region, where \(\tilde q \in B_4(q)\);
\item[(d)] for each \(\tilde{d}\)-ball \(B_{r_{\tilde{d}}}(x_{\tilde{d}})\), if \(\mathcal{E}_{\tilde{d}}\equiv \{x \in B_{4r_{\tilde{d}}}(x_{\tilde{d}}) : \mathcal{V}_{\eta r_{\tilde{d}}}(x) < \bar{V} + \eta\}\), then we have \(\operatorname{Vol}(B_{\delta r_{\tilde{d}}} \mathcal{E}_{\tilde{d}}) \leq \epsilon  \delta^4 r_{\tilde{d}}^2\).
\end{itemize}
\end{proposition}

\begin{proof}
Let us consider \(\delta'\leq \delta'(n,\rv ,\tau ,\delta)\ll \delta\) to be fixed later. Denote \(s = 10^{-10n}\). To prove the result we will build on \(B_{2}(q)\) a sequence of weak \((\delta ,\tau)\)-necks \(\tilde{\cN}^{i}\), where each is a refinement of the last, with the eventual goal of arriving at a weak \((\delta ,\tau)\)-neck that cannot be extended any further. Before discussing the inductive procedure for this construction, we need to build the base neck region \(\tilde{\cN}^{1}\). Precisely, by conditions (2) and (3) we can, for \(\eta \leq \eta (n,\delta',\epsilon ,\tau)\), apply Theorem \ref{Content cone-splitting} to get that \(B_{{\delta'}^{- 1}}(q')\) is \(\delta^{\prime}\)-GH close to \(\mathbb{R}^{n-4}\times C(S^{3} / \Gamma)\) for some \(q'\in B_{4}(q)\). In particular, \(B_{12}(q')\) is \((n-4,\delta')\)-symmetric with respect to some \(\mathcal{L}^{1}\equiv \iota^{1}(\mathbb{R}^{n-4}\times \{y_{0}\})\cap B_{12}(q')\) where \(\iota^{1}:B_{{\delta'}^ { -1}}(0^{n-4},y_{0})\subseteq \mathbb{R}^{n-4}\times C(S^{3} / \Gamma)\to B_{{\delta'}^{- 1}}(q')\) is the Gromov‑Hausdorff map. Notice that \(\mathcal{L}^{1}\) acts as an approximate singular set on the ball.

Now consider a Vitali covering
\[
\mathcal{L}^{1}\subseteq \bigcup B_{\tau^{3/2}s}(x_{f}^{1})\subseteq \bigcup B_{s}(x_{f}^{1}),
\]
where \(x_{f}^{1}\in \mathcal{L}^{1}\) and \(\{B_{\tau^{5/3} s}(x_{f}^{1})\}\) is a maximal disjoint collection. We then define
\[
\tilde{\cN}^{1}\equiv B_{12}(q')\setminus \bigcup B_{s}(x_{f}^{1}),
\]
and it is easy to check that for \(\eta \leq \eta (n,\epsilon ,\tau ,\delta^{\prime})\) this is indeed a weak \((\delta ,\tau)\)-neck. Before moving on to the inductive construction, let us further separate the singular balls into a couple of better groups. To accomplish this, recall that for a ball \(B_{r}(x)\), we can define the \(\eta\)-pinched subset \(\mathcal{E}_{\eta}\) by
\[
\mathcal{E}_{\eta}(x,r)\equiv \{y\in B_{4r}(x):\mathcal{V}_{\eta r}(y)< \overline{V} +\eta \}.
\] 
Thus for each ball \(B_{s}(x_{f}^{1})\) in our covering, if we consider the pinched points \(\mathcal{E}_{\eta}(x_{f}^{1},s)\), then we can consider one of two cases:
\begin{enumerate}
\item[$(c)$] $\Vol\bigl(B_{\delta s}\mathcal{E}_{\eta}(x_{f}^{1},s)\bigr) > \epsilon \delta^{4}s^{n},$ 
\item[$(\tilde{d})$] $\Vol\bigl(B_{\delta s}\mathcal{E}_{\eta}(x_{f}^{1},s)\bigr)\leq \epsilon \delta^{4}s^{n}.$
\end{enumerate}
Of course, in case (c) above we have that \(B_{s}(x_{f}^{1})\) is a \(c\)-ball, while in case $(\tilde{d})$ above, we have that \(B_{s}(x_{f}^{1})\) is a \(\tilde{d}\)-ball. Therefore, let us break our balls into these two subsets and write
\[
\tilde{\cN}^{1}\equiv B_{12}(q')\setminus \Bigl(\bigcup_{c}B_{s}(x_{c}^{1})\cup \bigcup_{\tilde{d}}B_{s}(x_{\tilde{d}}^{1})\Bigr).
\]
Now we move on to the inductive procedure for the construction of the weak \((\delta ,\tau)\)-necks \(\tilde{\cN}^{i}\). Thus, assume we have constructed the weak \((\delta ,\tau)\)-neck \(\tilde{\cN}^{i}\) given by
\[
\tilde{\cN}^{i}\equiv B_{12}(q')\setminus \Bigl(\bigcup_{c}B_{s^{i}}(x_{c}^{i})\cup \bigcup_{j\leq i}\bigcup_{\tilde{d}}B_{s^{j}}(x_{\tilde{d}}^{j})\Bigr),
\]
such that the following additional conditions hold:
\begin{itemize}
\item[(i.1)] Each \(c\)-ball \(B_{s^{i}}(x_{c}^{i})\) satisfies \(\Vol\bigl(B_{\delta  s^{i}}\mathcal{E}_{\eta}(x_{c}^{i},s^{i})\bigr) > \epsilon \delta^{4}(s^{i})^{n}\) and there exists an approximate singular set \(\mathcal{L}_{c}^{i} = \iota_{i,c}(\mathbb{R}^{n-4}\times \{y_{i,c}\} \cap B_{2}(0^{n-4},y_{i,c}))\) where the map \(\iota_{i,c}:B_{\delta'^{- 1} s^{i}}(0^{n-4},y_{i,c})\subset \mathbb{R}^{n-4}\times C(S^{3} / \Gamma_{i,c})\to B_{\delta'^{- 1}s^{i}}(\hat{x}_{c}^{i})\) is a \(\delta^{\prime}s^{i}\)-GH map for some \(\hat{x}_{c}^{i}\in B_{\delta^2s^{i}}(x_{c}^{i})\).
\item[(i.2)] Each \(\tilde{d}\)-ball \(B_{s^{j}}(x_{\tilde{d}}^{j})\) satisfies \(\Vol\bigl(B_{\delta s^{j}}\mathcal{E}_{\eta}(x_{\tilde{d}}^{j},s^{j})\bigr)\leq \epsilon \delta^{4}(s^{j})^{n}\).
\item[(i.3)] Denote \(\tilde{\mathcal{C}}_{i}\coloneqq \{x_{c}^{i}\} \bigcup \{x_{\tilde{d}}^{j},j\leq i\}\) and the associated \(r_{x}=s^{i}\) for \(x=x_{c}^{i}\) and \(r_{x}=s^{j}\) for \(x=x_{\tilde{d}}^{j}\). Then the balls \(\{B_{\tau^{5/3} r_{x}}(x),x\in \tilde{\mathcal{C}}_{i}\}\) are pairwise disjoint.
\item[(i.4)] Every \(c\)-ball has radius \(s^{i}\).
\item[(i.5)] \(\tilde{\mathcal{C}}_{i-1}\subset \tilde{\mathcal{C}}_{i}\), where the center of a \(c\)-ball in \(\tilde{\mathcal{C}}_{i-1}\) is a center of a \(\tilde{d}\)-ball or \(c\)-ball in \(\tilde{\mathcal{C}}_{i}\) with radius \(s^{i}\), and the center of a \(\tilde{d}\)-ball in \(\tilde{\mathcal{C}}_{i-1}\) is still a \(\tilde{d}\)-ball in \(\tilde{\mathcal{C}}_{i}\) with the same radius.
\item[(i.6)] Denote \(\mathcal{L}^{i-1}\equiv \bigcup_{c}\bigl(\mathcal{L}_{c}^{i-1}\cap B_{2s^{i-1}}(x_{c}^{i-1})\bigr)\cap B_{12}(q')\), where \(\mathcal{L}_{c}^{i-1}\) is the approximate singular set corresponding to each \(c\)-ball \(B_{s^{i-1}}(x_{c}^{i-1})\) as in (i.1). Then the centers \(\tilde{\mathcal{C}}_{i}\setminus \tilde{\mathcal{C}}_{i-1}\) with radius \(s^{i}\) are chosen such that
\[
\tilde{\mathcal{C}}_{i}\setminus \tilde{\mathcal{C}}_{i-1}\subset \mathcal{L}^{i-1}\setminus \Bigl(\bigcup_{x_{c}^{i-1}}B_{2\tau^{5/3} s^{i}}(x_{c}^{i-1})\bigcup_{j\leq i-1}B_{2\tau^{5/3} s^{j}}(x_{\tilde{d}}^{j})\Bigr)
\]
and
\[
\mathcal{L}^{i-1}\setminus \Bigl(\bigcup_{x_{c}^{i-1}}B_{2\tau^{5/3} s^{i}}(x_{c}^{i-1})\bigcup_{j\leq i-1}B_{2\tau^{5/3} s^{j}}(x_{\tilde{d}}^{j})\Bigr)\subseteq B_{\tau^{3/2} s^{i}}(\tilde{\mathcal{C}}_{i}\setminus \tilde{\mathcal{C}}_{i-1}).
\]
\item[(i.7)] \(r_{h}(x)\leq 2r_{x}\) for any \(x\in \tilde{\mathcal{C}}_{i}\).
\end{itemize}

We therefore wish to build a weak \((\delta ,\tau)\)-neck \(\tilde{\cN}^{i+1}\) that satisfies \((i+1).1-(i+1).7\). In order to build this weak neck region we must therefore break apart the \(c\)-balls in our covering. First we consider the total approximate singular set on scale \(s^{i}\) given by
\[
\mathcal{L}^{i}\equiv \bigcup_{c}\mathcal{L}_{c}^{i}\cap B_{2s^{i}}(x_{c}^{i})\cap B_{12}(q').
\]

Now consider a collection of balls \(\{B_{s^{i+1}}(x_{f}^{i+1})\}\) such that
\[
\begin{aligned}
&\mathcal{L}^{i}\setminus \Bigl(\bigcup_{x_{c}^{i}}B_{2\tau^{5/3} s^{i+1}}(x_{c}^{i})\bigcup_{j\leq i}B_{2\tau^{5/3} s^{j}}(x_{\tilde{d}}^{j})\Bigr)\subseteq \bigcup_{f}B_{\tau^{3/2} s^{i+1}}(x_{f}^{i+1}),\\ 
&x_{f}^{i+1}\in \mathcal{L}^{i}\setminus \Bigl(\bigcup_{x_{c}^{i}}B_{2\tau^{5/3} s^{i+1}}(x_{c}^{i})\bigcup_{j\leq i}B_{2\tau^{5/3} s^{j}}(x_{\tilde{d}}^{j})\Bigr),\\
&\{B_{\tau^{5/3},s^{i+1}}(x_{f}^{i+1})\} \text{ are disjoint}.
\end{aligned}\]
Now each ball \(\{B_{s^{i+1}}(x_{f}^{i+1}),B_{s^{i+1}}(x_{c}^{i})\}\) is either a \(c\)-ball that satisfies
\[
\mathrm{Vol}\bigl(B_{\delta s^{i+1}}\mathcal{E}_{\eta}(x_{f}^{i+1},s^{i+1})\bigr) > \epsilon \delta^{4}(s^{i+1})^{n}
\]
with \(x_{f}^{i+1}\in B_{\delta^{2} s^{i+1}}\mathcal{L}_{f}^{i+1}\) thanks to quantitative cone-splitting (Theorem \ref{cone-splitting}) and strong \(\vepsilon\)-regularity assumption (A3), or is a \(\tilde{d}\)-ball that satisfies the converse inequality
\[
\Vol\bigl(B_{\delta s^{i+1}}\mathcal{E}_{\eta}(x_{f}^{i+1},s^{i+1})\bigr)\leq \epsilon \delta^{4}(s^{i+1})^{n}.
\]

Thus we can separate the balls into these two types by
\[
\{B_{s^{i+1}}(x_{f}^{i+1}),B_{s^{i+1}}(x_{c}^{i})\} = \{B_{s^{i+1}}(x_{c}^{i+1})\}_{c}\cup \{B_{s^{i+1}}(x_{\tilde d}^{i+1})\}_{\tilde d}.
\]

This allows us to define the region \(\tilde{\cN}^{i+1}\) by
\[
\tilde{\cN}^{i+1}\equiv B_{12}(q')\setminus \Bigl(\bigcup_{c}B_{s^{i+1}}(x_{c}^{i+1})\cup \bigcup_{j\leq i+1}\bigcup_{\tilde{d}}B_{s^{j}}(x_{\tilde{d}}^{j})\Bigr).
\]

It is easy to check from the construction that if \(\eta \leq \eta (\delta',n,\tau ,\rv)\), then the region \(\tilde{\cN}^{i+1}\) satisfies \((i+1).1-(i+1).7\), where \((i+1).7\) follows by the inductive fact that the ball \(B_{s^{i}}(x_{c}^{i})\) is close to a nontrivial cone \(\mathbb{R}^{n-4}\times C(S^{3}/\Gamma)\).

Now check that if \(\delta'\leq \delta'(\delta ,\tau ,n,\rv)\) and \(\eta \leq \eta (\delta',\delta ,n,\tau ,\rv)\), then \(\tilde{\cN}^{i+1}\) is indeed a weak \((\delta ,\tau)\)-neck region as defined in Definition \ref{weak neck region} except for the condition \(q'\in \tilde{\mathcal{C}}\). Condition \((\tilde{n}1)\) follows directly from $(i+1).3.$ To see conditions \((\tilde{n}2)\) and \((\tilde{n}3)\), for any \(x\in \tilde{\mathcal{C}}_{i+1}\) and \(1\geq r\geq r_{x}=s^{j}\), assume \(s^{j_{0}}\leq r< s^{j_{0}-1}\) with \(1\leq j_{0}\leq j\). By the inductive scheme, there exist \(\{x_{c}^{k}\in \tilde{\mathcal{C}}_{k},k=j-1,\ldots ,j_{0}-1\}\) such that
\[
x\in \mathcal{L}_{c}^{j-1}\cap B_{s^{j-1}}(x_{c}^{j-1}),\quad x_{c}^{j-1}\in \mathcal{L}_{c}^{j-2}\cap B_{s^{j-2}}(x_{c}^{j-2}),\ldots ,x_{c}^{j_{0}}\in \mathcal{L}_{c}^{j_{0}-1}\cap B_{s^{j_{0}-1}}(x_{c}^{j_{0}-1}).
\]

By the quantitative cone-splitting (Theorem \ref{cone-splitting}) and strong \(\vepsilon\)-regularity assumption (A3), and condition \(k.7\) for \(k=j-1,\ldots ,j_{0}\), if \(\eta \leq \eta (n,\delta^{\prime},\tau ,\rv)\), then \(\mathcal{L}_{c}^{k}\subset B_{\delta^{\prime}s^{k+2}}\mathcal{L}_{c}^{k-1}\). This implies
 \begin{equation}\label{distance to approximate singuar}
d(x,\mathcal{L}_{c}^{j_{0}-1})\leq \sum_{k=j_{0}}^{j-1}\delta^{\prime}s^{k+2}\leq \delta^{\prime}s^{j_{0}+1}.
\end{equation}

If \(\delta^{\prime}\leq \delta^{\prime}(n,\tau ,\delta ,\rv)\) and \(\eta \leq \eta (\delta^{\prime},\delta ,n,\tau ,\rv)\), the above gives \((\tilde{n}2)\). To check \((\tilde{n}3)\), by the inductive scheme \(k.5\) for \(k=i,\ldots ,j_{0}\), we get
\[
\tilde{\mathcal{C}}_{j_{0}}\subset \tilde{\mathcal{C}}_{j_{0}+1}\subset \dots \subset \tilde{\mathcal{C}}_{i+1}.
\]

By (\ref{distance to approximate singuar}),  the quantitative cone-splitting (Theorem \ref{cone-splitting}) and strong \(\vepsilon\)-regularity assumption (A3), and condition \((j_{0}-1).7\) again, if \(\eta \leq \eta (\delta',\delta ,n,\tau ,\rv)\), we have
\[
\mathcal{L}_{x,r}\subset B_{\delta^{\prime}s^{j_{0}}}\mathcal{L}_{c}^{j_{0}-1},
\]
where \(\mathcal{L}_{x,r}\) is defined in condition \((\tilde{n}3)\). On the other hand, by condition \(j_{0}.6\) we have
\[
\mathcal{L}_{c}^{j_{0}-1}\subset B_{\tau^{3/2}s^{j_{0}}}(\tilde{\mathcal{C}}_{j_{0}})\cup_{x_{\tilde{d}}^{j}\in \tilde{\mathcal{C}}_{j_{0}}}B_{\tau^{3/2}s^{j}}(x_{\tilde{d}}^{j}).
\]

By the definition of \(r_{x}\) for \(x\in \tilde{\mathcal{C}}_{i+1}\), this is equivalent to
\[
\mathcal{L}_{c}^{j_{0}-1}\subset B_{\tau^{3/2}s^{j_{0}}}(\tilde{\mathcal{C}}_{j_{0}})\cup_{y\in \tilde{\mathcal{C}}_{j_{0}}\cap \tilde{\mathcal{C}}_{i+1}}B_{\tau^{3/2}r_{y}}(y).
\]
Combining the above, we get
\[
\mathcal{L}_{x,r}\subset B_{\tau^{4/3}s^{j_{0}}}\tilde{\mathcal{C}}_{i+1}\cup_{y\in \tilde{\mathcal{C}}_{i+1}}B_{\tau^{4/3}r_{y}}(y)\subset B_{\tau^{4/3}(r+r_{y})}(\tilde{\mathcal{C}}_{i+1}),
\]
which proves one of the two covering conditions in \((\tilde{n}3)\). For the other covering condition, we argue by contradiction. Assume there exists \(y\in \tilde{\mathcal{C}}_{i+1}\cap B_{r}(x)\setminus B_{\delta r}(\mathcal{L}_{x,r})\). If \(r_{y}\leq \tau^{-100}r\) then by condition \((\tilde{n}2)\) which we have checked, if \(\eta \leq \eta (n,\rv,\tau,\delta,\delta^{\prime})\) the ball \(B_{r}(y)\) must be close to a metric cone. This contradicts \((i+1).7\) the quantitative cone-splitting (Theorem \ref{cone-splitting}) and strong \(\vepsilon\)-regularity assumption (A3). On the other hand, if \(r_{y}\geq \tau^{-100}r\) then since \(y\in B_{r}(x)\), we have \(B_{r_{x}}(x)\subset B_{r}(x)\subset B_{\tau^{3}r_{y}}(y)\), which contradicts the fact that \(B_{\tau^{2}r_{x}}(x)\cap B_{\tau^{2}r_{y}}(y)=\varnothing\). Thus we have shown \(\tilde{\cN}^{i+1}\) is a weak \((\delta,\tau)\)-neck region.

Now to finish the proof of the proposition, consider the closed discrete sets \(\mathcal{C}_{c}^{i}\equiv \bigcup_{c}\{x_{c}^{i}\}\) and note by construction that \(\mathcal{C}_{c}^{i+1}\subseteq B_{s^{i}}(\mathcal{C}_{c}^{i})\). Since $M$ is smooth, it is clear that the Hausdorff limit
\[
\tilde {\mathcal{C}}_{0} = \lim_{i\to\infty}\mathcal{C}_{c}^{i+1}=\varnothing.
\] 
Taking the weak neck regions \(\tilde{\cN}^{i}\) and letting \(i\to\infty\), we arrive at a region
\[
\tilde{\cN}\equiv B_{12}(q')\setminus   \bigcup_{j}\bigcup_{\tilde d}B_{s^{j}}(x_{\tilde d}^{j}) .
\] 
One can check that \(\tilde{\cN}\) is a weak \((\delta,\tau)\)-neck region except in the condition \(q'\in \tilde{\mathcal{C}}\) by using the same argument as above. Now choose a point \(\tilde q\in B_{4}(q')\cap \tilde{\mathcal{C}}\) and we have completed the proof.
\end{proof}

\begin{proposition}[Srong $c$-ball covering]
 Let \((M^n, g)\) satisfy $(A1),(A2),(A3)$   with \(\delta >0\) and \(0< \epsilon,\tau < \tau(n)\). Then for \(\eta \leq \eta(n,\rv,\delta,\epsilon,\tau)\), assume the following hold:
\begin{enumerate}
    \item[$(1)$]   $\int_{B_{\delta^{-1}}(q)}|\Ric|^{p} < \eta$;
    \item[$(2)$] if \(\overline{V}\equiv \inf_{B_4(p)}\mathcal{V}_{\eta^{-1}}(y)\) and \(\mathcal{E}_\eta\equiv \{x\in B_4(q):\mathcal{V}_\eta(x)< \overline{V}+\eta\}\), then \(\Vol(B_\delta \mathcal{E}_\eta) > \epsilon\cdot\delta^4\);
    \item[$(3)$] \(r_h(q)<2\).
\end{enumerate}
Then we can write
\[
B_1(q)\subseteq \cN\cup \bigcup_{b}B_{r_b}(x_b)\cup \bigcup_{c}B_{r_c}(x_c)\cup \bigcup_{d}B_{r_d}(x_d)\cup \bigcup_{e}B_{r_e}(x_e),
\]
where
\begin{itemize}
    \item[(a)] \(\cN = B_{11}(\tilde q )\setminus\Bigl( \displaystyle\bigcup_{b}B_{r_b}(x_b)\cup\displaystyle\bigcup_{c}B_{r_c}(x_c)\cup\displaystyle\bigcup_{d}B_{r_d}(x_d)\cup\displaystyle\bigcup_{e}B_{r_e}(x_e)\Bigr)\) 
          is a \((\delta,\tau)\)-neck region  with \(\tilde q \in B_4(q)\);
    \item[(b)] for each \(b\)-ball \(B_{r_b}(x_b)\), we have \(r_h(x_b)>2r_b\);
    \item[(c)] for each \(c\)-ball \(B_{r_c}(x_c)\), if \(\mathcal{E}_c\equiv \{x\in B_{4r_c}(x_c):\mathcal{V}_{r_c}(x)<\overline{V}+\eta\}\), then 
          \(\Vol(B_{\delta r_c}\mathcal{E}_c) > \epsilon \delta^4 r_c^n\);
    \item[(d)] for each \(d\)-ball \(B_{r_d}(x_d)\), if \(\mathcal{E}_d\equiv \{x\in B_{4r_d}(x_d):\mathcal{V}_{r_d}(x)<\overline{V}+\eta\}\), then 
          \(0<\Vol(B_{\delta r_d}\mathcal{E}_d)\leq \epsilon\delta^4 r_d^n\);
    \item[(e)] for each \(e\)-ball \(B_{r_e}(x_e)\), if \(\mathcal{E}_e\equiv \{x\in B_{4r_e}(x_e):\mathcal{V}_{\eta r_e}(x)<\overline{V}+\eta\}\), then \(\mathcal{E}_e=\varnothing\).
\end{itemize}
Moreover, we have the content bounds
\[
\sum_{x_b\in B_5(q)} r_b^{n-4}+\sum_{x_d\in B_5(q)} r_d^{n-4}+\sum_{x_e\in B_5(q)} r_e^{n-4} < C(n,\tau),
\qquad
\sum_{x_e\in B_5(q)} r_e^{n-4} < C(\eta,\tau)\,\epsilon.
\]
\end{proposition}
\begin{proof}
Let us begin the proof by applying Proposition \ref{weak c-ball} with parameters $ \delta^{\prime}$ and $\tau $, where \(\delta^{\prime}\leq \delta^{\prime}(n,\rv ,\delta ,\tau)\), in order to get a covering
\[
B_{1}(q)\subseteq   \tilde{\mathcal{N}}\cup \bigcup_{\tilde{d}\in \tilde{\mathcal{C}}_{+}}B_{\tilde{r}_{\tilde{d}}}(x_{\tilde{d}}),
\]
where \(\tilde{\cN}\equiv B_{11}(\tilde q)\setminus   \bigcup_{\tilde{d}\in \tilde{\mathcal{C}}_{+}}B_{\tilde{r}_{\tilde{d}}}(x_{\tilde{d}})\) is a weak \((\delta^{\prime},\tau)\)-neck region with \(\tilde q\in B_{4}(q)\), and \(B_{\tilde{r}_{\tilde{d}}}(x_{\tilde{d}})\) are all \(\tilde{d}\)-balls with respect to \(\eta^{\prime}\). We will eventually pick our \(\eta \ll \eta^{\prime}\). In the proof below we will use \(\tilde{r}\) to represent the radius of a ball with respect to \(\tilde{\cN}\).

Since \(\tilde{\cN}\) is a weak neck region, we have, in particular, that the collection of balls \(\{B_{\tau^{2}\tilde{r}_{\tilde{d}}}(x_{\tilde{d}})\}\) are all disjoint. Now define the following radius function:
\[
r_y \equiv 
\begin{cases} 
\delta^2 \tau^3 \tilde{r}_{\tilde d} & \text{if } y \in \overline{B}_{\tau^3 \tilde{r}_{\tilde d}}(x_{\tilde d}), \\ 
\delta^2 d(y, \tilde{\cC}) & \text{if } y \notin \bigcup B_{\tau^3 \tilde{r}_d}(x_{\tilde d}). 
\end{cases}
\]
Note that we have the gradient bound \(|\nabla r_y|\leq \delta^{2}\), and \(r_y\geq \delta^{2} d(y,\tilde{\mathcal{C}})\).

In order to build our new covering let us make several observations. To begin with, for each \(y\in B_{2}(q)\), let \(\tilde{y}\in \tilde{\mathcal{C}}\) denote the center point that minimizes \(d(y,\tilde{\mathcal{C}})\). Recalling from the definition of neck region that \(\mathcal{L}_{x,r} = \iota_{x,r}\bigl(B_{r}(0^{n-4})\times\{y_c\}\bigr)\) is the effective singular set on \(B_{r}(x)\), we define the set for \(\tilde{\delta} = \tilde{\delta}(\delta ,\rv ,n,\tau)< \delta\) to be fixed later so that
\[
\tilde{\mathcal{S}}\equiv \{y\in B_{11}(q):y\in B_{\tilde{\delta}^{2} r_y}\mathcal{L}_{\tilde{y},\max\{2d(y,\tilde{\mathcal{C}}),2\tilde{r}_{\tilde{y}}\}}\},
\]
which is roughly to say that \(\tilde{\mathcal{S}}\) are the set of points that belong to the effective singular set on their own scale \(r_y\). Our first claim, which is immediate once you untangle the constants in the definition of \(\tilde{\mathcal{S}}\), is that points of \(\tilde{\mathcal{S}}\) look very close to \(\mathbb{R}^{n-4}\times C(S^{3}/\Gamma)\) at scales bigger than \(r_y\). Precisely,

\begin{claim}
For any given \(\tilde{\delta} >0\), let \(y\in \tilde{\mathcal{S}}\) with \(\delta^{\prime}\leq \delta^{\prime}(n,\rv ,\tau ,\tilde{\delta})\). Then for all \(r_y< r\leq 2\), we have that \(B_{\tilde{\delta}^{-1}r}(y)\) is pointed \(\tilde{\delta} r\)-GH close to \(B_{\tilde{\delta}^{-1}r}(0^{n-4},y_c)\subseteq \mathbb{R}^{n-4}\times C(S^{3}/\Gamma)\).
\end{claim}

Now consider the covering of \(\tilde{\mathcal{S}}\) given by
\[
\tilde{\mathcal{S}}\subseteq \bigcup_{y\in \tilde{\mathcal{S}}}B_{\tau^{3/2}r_{y}}(y),
\]
and let \(\{B_{\tau^{3/2}r_i}(y_i)\}\) be a maximal Vitali subcovering such that \(\{B_{\tau^{2}r_i}(y_i)\}\) are disjoint and \(\tilde{\mathcal{C}}\subset \mathcal{C}\coloneqq   \{y_i\}\). If we define
\[
\cN\equiv B_{11}(\tilde q)\setminus   \bigcup_{i}\overline{B_{r_i}}(y_i),
\]
then it follows from almost the same argument as Proposition \ref{weak c-ball} that \(\cN\) is a \((\delta,\tau)\)-neck region if \(\tilde{\delta}\leq \tilde{\delta}(n,\rv ,\tau ,\delta)\) and \(\delta^{\prime}\leq \delta^{\prime}(n,\rv ,\tau ,\delta)\). Actually, noting the above claim, let us validate the only nontrivial condition (n3). For any \(x\in \mathcal{C}\) and \(1\geq r\geq r_x\) with \(B_{2r}(x)\subset B_{11}(\tilde q)\), we will show \(\mathcal{L}_{x,r}\cap B_r(x)\subset B_{\tau r}(\mathcal{C})\) and \(\mathcal{C}\cap B_r(x)\subset B_{\delta r}(\mathcal{L}_{x,r})\), where \(\mathcal{L}_{x,r}\) is the approximate singular set corresponding to the \((n-4,\tilde{\delta})\)-symmetric ball \(B_r(x)\) from the above claim.

The proof of \(\mathcal{L}_{x,r}\cap B_r(x)\subset B_{\tau r}(\mathcal{C})\) uses a contradiction argument and the fact that \(B_{\tau^{3}r}(y)\cap\tilde{\mathcal{S}}\neq\varnothing\) for any \(y\in\mathcal{L}_{x,r}\cap B_r(x)\) when \(\delta'\) and \(\tilde{\delta}\) are small enough. The other covering \(\mathcal{C}\cap B_r(x)\subset B_{\delta r}(\mathcal{L}_{x,r})\) follows directly from the quantitative cone-splitting (Theorem \ref{cone-splitting}) and strong \(\vepsilon\)-regularity assumption (A3) if \(\tilde{\delta}\leq \tilde{\delta}(n,\rv,\tau,\delta)\) and \(\delta'\leq\delta'(n,\tau,\delta,\rv,\tilde{\delta})\).

By Ahlfors Regularity (Theorem \ref{strong L2 neck}) we now get the estimate
\[
\cH^{n-4}( B_5(\tilde q))+\sum_{x_i\in B_5(\tilde q)\cap\mathcal{C}}r_i^{n-4}\leq C(n,\tau).
\]

Now certainly each ball \(B_{r_i}(x_i)\) is classified into a $(b) \sim (e)$ ball, and therefore we can write
\[
\{B_{r_i}(x_i)\}= \{B_{r_b}(x_b)\}_b\cup \{B_{r_c}(x_c)\}_c\cup \{B_{r_d}(x_d)\}_d\cup \{B_{r_e}(x_e)\}_e.
\]

The crucial aspect of the proof that remains is to show that we have the estimate
\begin{equation}\label{volume estimate 2}
\sum_{x_c\in B_5(\tilde q)\cap\mathcal{C}}r_c^{n-4}\leq C(n,\tau)\epsilon,
\end{equation}
which is to say that the \(c\)-balls have small content. In order to see this, begin by observing that by the definition of \(r_y\) and since \(\tilde{\cN}\) is a weak \((\delta',\tau)\)-neck, we have the covering
\begin{equation}\label{covering}
\tilde{\mathcal{S}}\subseteq  \bigcup_{x_{\tilde d}\in\tilde{\mathcal{C}}}B_{10\tau\tilde{r}_{\tilde d}}(x_{\tilde d}).
\end{equation}
To see $(\ref{covering})$, for any \(y\in\tilde{\mathcal{S}}\) but not in \( \bigcup_{\tilde d} B_{9\tau\tilde{r}_{\tilde d}}(x_{\tilde d})\), assume \(\tilde{y}\in\tilde{\mathcal{C}}\) with \(d(y,\tilde{\mathcal{C}})=d(y,\tilde{y})\). Then \(r_y=\delta^2 d(y,\tilde{\mathcal{C}})\geq 9\delta^2\tau\tilde{r}_{\tilde d}\). By the property of a weak \((\delta',\tau)\)-neck region we get
\[
\mathcal{L}_{\tilde{y},\max\{2d(y,\tilde{\mathcal{C}}),2\tilde{r}_{\tilde{y}}\}}\subset \bigcup_{x_d\in\tilde{\mathcal{C}}}B_{4\tau(d(y,\tilde{\mathcal{C}})+\tilde{r}_{\tilde d})}(x_{\tilde d}).
\]
Thus \(y\in \bigcup_{x_{\tilde d}\in\tilde{\mathcal{C}}}B_{4\tau(d(y,\tilde{\mathcal{C}})+\tilde{r}_{\tilde d})+\delta^2 r_y}(x_{\tilde d})\). A short computation shows that necessarily \(d(y,x_{\tilde d})<10\tau\tilde{r}_{\tilde d}\) for some \(x_{\tilde d}\in\tilde{\mathcal{C}}\), proving  $(\ref{covering})$.

Using Theorem \ref{strong L2 neck} we have that   the packing measure $\mu$ associated to the neck region \(\cN\)  is a doubling measure. In particular, since \(\{B_{\tau^2\tilde{r}_{\tilde d}}(x_{\tilde d})\}\) are disjoint and \(\tilde{\mathcal{C}}\subset\mathcal{C}\) with \(r_{x_{\tilde d}}\leq\delta^2\tau^3\tilde{r}_{\tilde d}\), we obtain
\[
\sum_{x_{\tilde d}\in B_9(\tilde q)}\tilde{r}_{\tilde d}^{\,n-4}\leq C(n,\tau)\mu(B_{10}(\tilde q))\leq C(n,\tau). 
\]

Now for each ball \(B_{\tilde{r}_{\tilde d}}(x_{\tilde d})\), consider the set
\[
C_{\tilde d}\equiv \{x_i\in\mathcal{C}\cap B_{10\tau\tilde{r}_{\tilde d}}(x_{\tilde d}): B_{r_i}(x_i)\text{ is a }c\text{-ball}\}.
\]

\begin{claim}
We have the estimate \(\mu(C_{\tilde d})\leq C(n,\tau)\,\epsilon\,\tilde{r}_{\tilde d}^{\,n-4}\).
\end{claim}

To prove the claim, note that if \(x_i\in C_{\tilde d}\) then \(r_i<\delta^2\tilde{r}_{\tilde d}\) and thus \(B_{4r_i}(x_i)\subseteq B_{\tilde{r}_{\tilde d}}(x_{\tilde d})\). Moreover, there exists \(y_i\in B_{4r_i}(x_i)\) such that \(\mathcal{V}_{\eta r_i}(y_i)<\bar{V}+\eta<\bar{V}+\eta'\). On the other hand, since $\eta r_i\ll\eta'\tilde{r}_{\tilde{d}}\ll 1$, by smallness of volume pinch (Theorem \ref{summability of volume drop}), we conclude \(y_i\in\mathcal{E}_{{\tilde r}_{\tilde d},\eta' }:= \left\{ y \in B_{4\tilde r_{\tilde d}}(x_{\tilde d}) : \mathcal{V}_{\eta' \tilde r_{\tilde d}}(y) < \bar{V} + \eta' \right\}. \)  Therefore \(B_{10^{-1}\delta\tilde{r}_{\tilde d}}(x_i)\subset B_{\delta\tilde{r}_{\tilde d}}\mathcal{E}_{\tilde{r}_{\tilde d},\eta'}\). Choose a maximal collection of balls \(\{B_{\delta\tilde{r}_{\tilde d}}(x_i')\}_{1}^{N'}\) with \(x_i'\in C_{\tilde d}\) such that \(\{B_{10^{-1}\delta\tilde{r}_{\tilde d}}(x_i')\}\) are disjoint. Because \(B_{\tilde{r}_d}(x_{\tilde d})\) is a \(\tilde{d}\)-ball, we have
\[
N'\delta^{n}\tilde{r}_{\tilde d}^{\,n}\leq C(n)\Vol(B_{\delta\tilde{r}_{\tilde d}}\mathcal{E}_{\tilde{r}_{\tilde d},\eta'})\leq C(n)\epsilon\,\delta^{4}\tilde{r}_{\tilde d }^{\,n},
\]
so \(N'\leq C(n)\epsilon\,\delta^{4-n}\). Using the Ahlfors regularity of \(\mu\) (Theorem \ref{strong L2 neck}) we get
\[
\mu(C_{\tilde d})\leq \sum_{i=1}^{N'}\mu(B_{\delta\tilde{r}_{\tilde d}}(x_i'))\leq C(n,\tau)\delta^{n-4}\tilde{r}_{\tilde d}^{\,n-4}N'\leq C(n,\tau)\epsilon\,\tilde{r}_{\tilde d}^{\,n-4},
\]
which proves the claim.

Finally, let
\[
C\equiv \{x_i\in B_5(q):x_i\text{ is a }c\text{-ball}\}.
\]
Then by (\ref{volume estimate 2}) and the claim,
\[
\sum_{x_c\in B_5(\tilde q)}r_c^{n-4}\leq C(n,\tau)\mu(C)\leq C(n)\sum_{\tilde{x}_{\tilde d}\in B_9(\tilde q)}\mu(C_{\tilde d})\leq C(n,\tau)\epsilon\sum_{x_{\tilde d}\in B_9(\tilde q)}\tilde{r}_{\tilde d}^{\,n-4}\leq C(n,\tau)\epsilon,
\]
which finishes the proof of the proposition. \qedhere
\end{proof}

\begin{proposition}[$d$-ball covering]\label{d-ball} 
 Let \((M^n, g)\) satisfy $(A1),(A2),(A3)$ 
with \(\delta, \tau > 0\) and \( \bar{V}= \inf_{y \in B_4(q)} \mathcal{V}_{\eta^{-1}}(y)>0\). Then for  
\(\epsilon \leq \epsilon(n, \rv)\), \(\eta \leq \eta(n, \bar{V}, \delta, \epsilon, \tau)\), let us assume the following holds:  

\begin{enumerate}
\item[$(1)$] $\int_{B_{\delta^{-1}}(q)}|\Ric|^{p} < \eta$;
\item[$(2)$] if \(\mathcal{E}_\eta \equiv \{y \in B_4(q) : \mathcal{V}_\eta(y) < \bar{V} + \eta\}\), then \(0 < \operatorname{Vol}(B_\delta \mathcal{E}_\eta) < \epsilon  \delta^4\).  
\end{enumerate}

Then we can write \(B_1(q) \subset  \tilde{S}\cup\bigcup_b B_{r_b}(x_b) \cup \bigcup_c B_{r_c}(x_c) \cup \bigcup_e B_{r_e}(x_e)\), where  

\begin{itemize} 
\item[(s)] $\cH^{n-4}(\tilde{S})=0;$ 
\item[(b)] for each b-ball \(B_{r_b}(x_b)\), we have \(r_h(x_b) > 2r_b;\)   
\item[(c)] for each c-ball \(B_{r_c}(x_c)\), if \(\mathcal{E}_c \equiv \{y \in B_{4r_c}(x_c) : \mathcal{V}_{\eta r_c}(y) < \bar{V} + \eta\}\), then  
$\Vol(B_{\delta r_c} \mathcal{E}_c) > \epsilon   \delta^4 r_c^n;$
\item[(e)] for each e-ball \(B_{r_e}(x_e)\), if \(\mathcal{E}_e \equiv \{y \in B_{4r_e}(x_e) : \mathcal{V}_{\eta r_e}(y) < \bar{V} + \eta\}\), then  
$\mathcal{E}_e = \varnothing.$
\end{itemize}

Further, we have the content estimates  
\[
\sum_b r_b^{n-4} + \sum_e r_e^{n-4} < C(n, \delta)
\]  
and  
\[
\sum_c r_c^{n-4} \leq C(n, \rv)\epsilon.
\]
\end{proposition}

\begin{proof}
The procedure of the proof will be to iterate a certain covering construction and keep track of the estimates. To illustrate this let us begin by considering a Vitali covering of $B_1(q)$ given by
\[
B_1(q)\subseteq \bigcup_{f}B_{\delta}(x_f^{1}),
\]
with $\{B_{\delta/10}(x_f^{1})\}$ disjoint. Recall that for any ball $B_r(x)$, we can define the $\eta$-pinched subset $\mathcal{E}_{\eta}$ given by
\[
\mathcal{E}_{\eta}(x,r)\equiv \{y\in B_{4r}(x):\cV_{\eta r}(y)< \overline{V} +\eta \}.
\]
In this way we wish to separate the balls $\{B_{\delta}(x_f)\}$ into three types, depending on whether they are $c$-balls, $d$-balls or $e$-balls based on the following conditions:
\[
\mathrm{Vol}\big(B_{\delta^2}\mathcal{E}_{\eta}(x_f^{1},\delta)\big)>\epsilon\delta^{4}\delta^{n};
\]
\[
0< \mathrm{Vol}\big(B_{\delta^2}\mathcal{E}_{\eta}(x_f^{1},\delta)\big)\leq\epsilon\delta^{4}\delta^{n};
\]
\[
\mathcal{E}_{\eta}(x_f^{1},\delta)=\varnothing.
\]
Thus, by breaking up the balls $\{B_{\delta}(x_f)\}$ into these categories we can write our covering as
\[
B_1(q)\subseteq \bigcup_{c=1}^{N_c^{1}}B_{\delta}(x_c^{1})\cup \bigcup_{d=1}^{N_d^{1}}B_{\delta}(x_d^{1})\cup \bigcup_{e=1}^{N_e^{1}}B_{\delta}(x_e^{1}).
\]
Since $B_1(q)$ is itself a $d$-ball, we can use the Vitali condition of the covering and assumption (2) to conclude by a standard covering argument that
\[
\sum_{e=1}^{N_e^{1}}\delta^{n-4}\leq C(n,\delta),\qquad
\sum_{c=1}^{N_c^{1}}\delta^{n-4} + \sum_{d=1}^{N_d^{1}}\delta^{n-4}\leq C(n,\rv)\epsilon.
\]
This does not quite finish the proposition because we still have a collection of $d$-balls in our covering, and therefore we must recover them. Let us remark that the only aspect about $B_1(q)$ used in the above covering was that $B_1(q)$ was a $d$-ball that satisfied condition (2). Thus, for each $d$-ball $B_{\delta}(x_d^{1})$, let us repeat this covering process just introduced for $B_1(q)$. If we do this for every $d$-ball, then we arrive at the covering
\[
\bigcup_{d=1}^{N_d^{1}}B_{\delta}(x_d^{1})\subseteq \bigcup_{c=1}^{N_c^{2}}B_{\delta^{2}}(x_c^{2})\cup \bigcup_{d=1}^{N_d^{2}}B_{\delta^{2}}(x_d^{2})\cup \bigcup_{e=1}^{N_e^{2}}B_{\delta^{2}}(x_e^{2}),
\]
such that we have the estimates
\[
\sum_{d=1}^{N_d^{2}}(\delta^{2})^{n-4}\leq C(n,\rv)\epsilon\sum_{d=1}^{N_d^{1}}\delta^{n-4}\leq \big(C(n,\rv)\epsilon\big)^{2},
\]
\[
\sum_{c=1}^{N_c^{2}}(\delta^{2})^{n-4}\leq C(n,\rv)\epsilon\sum_{d=1}^{N_d^{1}}\delta^{n-4}\leq \big(C(n,\rv)\epsilon\big)^{2},
\]
\[
\sum_{e=1}^{N_e^{2}}(\delta^{2})^{n-4}\leq C(n,\delta)\sum_{d=1}^{N_d^{1}}\delta^{n-4}\leq C(n,\delta)\cdot C(n,\rv)\epsilon.
\]
Combining this with the original covering we obtain a covering of $B_1(q)$ given by
\[
B_1(q)\subseteq \bigcup_{c=1}^{N_c^{1}}B_{\delta}(x_c^{1})\cup \bigcup_{e=1}^{N_e^{1}}B_{\delta}(x_e^{1})\cup \bigcup_{c=1}^{N_c^{2}}B_{\delta^{2}}(x_c^{2})\cup \bigcup_{d=1}^{N_d^{2}}B_{\delta^{2}}(x_d^{2})\cup \bigcup_{e=1}^{N_e^{2}}B_{\delta^{2}}(x_e^{2})
\]
\[
= \bigcup_{d=1}^{N_d^{2}}B_{\delta^{2}}(x_d^{2})\cup \bigcup_{j\leq 2}\bigcup_{c=1}^{N_c^{j}}B_{\delta^{j}}(x_c^{j})\cup \bigcup_{j\leq 2}\bigcup_{e=1}^{N_e^{j}}B_{\delta^{j}}(x_e^{j}),
\]
with
\[
\sum_{d=1}^{N_d^{2}}(\delta^{2})^{n-4}\leq \big(C(n,\rv)\epsilon\big)^{2},\qquad
\sum_{j\leq 2}\sum_{c=1}^{N_c^{j}}(\delta^{j})^{n-4}\leq C(n,\rv)\epsilon + (C(n,\rv)\epsilon)^{2},
\]
\[
\sum_{j\leq 2}\sum_{e=1}^{N_e^{j}}(\delta^{j})^{n-4}\leq C(n,\delta)\big(1 + C(n,\rv)\epsilon\big).
\]
Of course, we may now proceed to cover the $d$-balls $\{B_{\delta^{2}}(x_d^{2})\}$ in the same manner once again. In fact, after repeating this inductive covering $i$ times we see that we arrive at a covering
\[
B_1(q)\subseteq \bigcup_{d=1}^{N_d^{i}}B_{\delta^{i}}(x_d^{i})\cup \bigcup_{j\leq i}\bigcup_{c=1}^{N_c^{j}}B_{\delta^{j}}(x_c^{j})\cup \bigcup_{j\leq i}\bigcup_{e=1}^{N_e^{j}}B_{\delta^{j}}(x_e^{j}),
\]
with the estimates
\[
\sum_{d=1}^{N_d^{i}}(\delta^{i})^{n-4}\leq \big(C(n,\rv)\epsilon\big)^{i},
\]
\[
\sum_{j\leq i}\sum_{c=1}^{N_c^{j}}(\delta^{j})^{n-4}\leq \sum_{1\leq j\leq i}\big(C(n,\rv)\epsilon\big)^{j},
\]
\[
\sum_{j\leq i}\sum_{e=1}^{N_e^{j}}(\delta^{j})^{n-4}\leq C(n,\delta)\sum_{0\leq j\leq i}\big(C(n,\rv)\epsilon\big)^{j}.
\]
To finish the proof consider the closed discrete sets $\tilde{S}^i = \bigcup_{d=1}^{i}\{x_d^i\}$. Note that by construction we have $\tilde{S}^{i+1}\subseteq B_{\delta^{i}}(\tilde{S}^i)$, and we can use the above to conclude the estimate
\[
\mathrm{Vol}(B_{\delta^{i}}(\tilde{S}^i))\leq \big(C(n,\rv)\epsilon\big)^{i} (\delta^{i})^{4}.
\]
Taking the Hausdorff limit we can construct
\[
\tilde{S} = \lim_{i\to\infty}\tilde{S}^{i+1},
\]
and we arrive at the covering
\[
B_1(q)\subseteq \tilde{S}\cup \bigcup_{j}\bigcup_{c}B_{\delta^{j}}(x_c^{j})\cup \bigcup_{j}\bigcup_{c}B_{\delta^{j}}(x_c^{j}).
\]
Using the estimates above and the inclusion $\tilde{S}\subseteq B_{\delta^{i}}(\tilde{S}^i)$, we see that for $\epsilon \leq \epsilon(n,\rv)$ and $0< r\leq 1$, we get
\[
r^{-4}\mathrm{Vol}(B_r(\tilde{S}))\to 0 \quad\text{as } r\to 0,
\]
\[
\sum_j\sum_c(\delta^{j})^{n-4}\leq C(n,\rv)\epsilon,\qquad
\sum_{j\leq i}\sum_{c=1}^{N_c^{j}}(\delta^{j})^{n-4}\leq C(n,\delta),
\]
which finishes the proof of the Proposition.
\end{proof}

Before proceeding to the inductive covering, we must address a fundamental difference between pointwise and integral Ricci curvature bounds. In the case of lower Ricci curvature bound, volume ratio $\mathcal{V}_r(x)$ is monotonic. However, under $L^p$ Ricci bounds, the volume ratio is just almost monotonic, which means it can slightly increase as the scale decreases. However, by Theorem \ref{summability of volume drop}, the summability of volume  pinch at two consecutive scales to ensure the covering algorithm terminates after finite many steps.

\begin{lemma}\label{induction}  
Let \((M^n, g)\) satisfy $(A1),(A2),(A3)$ with \(\bar{V} \equiv \inf_{y \in B_4(q)} \cV_{\eta^{-1}}\). Let us fix \( \delta>0\) and \(0 < \tau \leq \tau(n)\). Then there exists \(\eta(n, \rv, \delta, \tau) > 0\) such that we can write
\[
B_1(q) \subseteq \tilde{S} \cup \bigcup_a (  \cN_a  \cap B_{r_a}(x_a)) \cup \bigcup_b B_{r_b}(x_b) \cup \bigcup_e B_{r_e}(x_e),
\]
where
\begin{itemize}
\item[(a)] \(\cN_a \subseteq B_{2r_a}(x_a)\) is a \((\delta, \tau)\)-neck region;
\item[(b)] for each \(b\)-ball \(B_{r_b}(x_b)\), we have \(r_h(x_b) > 2r_b\);
\item[(c)] for each \(e\)-ball \(B_{r_e}(x_e)\), if \(\mathcal{E}_e \equiv \{y \in B_{4r_e}(x_e) : V_{\eta r_e} < \bar{V} + \eta\}\), then \(\mathcal{E}_e = \varnothing\); 
\item[(s)]  \(\cH^{n-4}(\tilde{S})=0.\)
\end{itemize}
Further, we have the content bound \(\sum_a r_a^{n-4} + \sum_b r_b^{n-4} + \sum_e r_e^{n-4} < C(n, \tau, \delta)\) and radius bound \(r_a, r_b, r_e \leq \eta^2\).
\end{lemma}

\begin{proof}
Let us pick $\eta$ as in Propositions \ref{weak c-ball} and \ref{d-ball}. Note first that we can cover $B_{1}(q)$ by the Vitali covering
\[
B_{1}(q)\subseteq \bigcup_{f}B_{\eta^{2}}(x_{f}),
\]
where $\{B_{\eta^{2} / 10}(x_{f})\}$ are disjoint. Then by a standard covering argument there are at most $C(n,\tau ,\delta)$ balls in this set. Therefore if we focus our covering on each of these balls individually, we may then take the union and not effect the content bound by more than another factor.

Thus, by rescaling one of these balls to scale one and focusing on it, we can assume without loss of generality that $\int_{B_{{\eta}^{-1}}(q)}|\mathrm{Ric}|^p\leq \eta$. Now let us recall that for every ball $B_{r}(x)$, we can define the $\eta$-pinched points $\mathcal{E}_{\eta}(x,r)$ by
\[
\mathcal{E}_{\eta}(x,r)\equiv \{y\in B_{4r}(x):\cV_{\eta r}(y)< \overline{V} +\eta \}.
\]
In our decomposition we want to first distinguish between whether $B_{1}(q)$ is a $b$-ball, $c$-ball, $d$-ball, or $e$-ball. In the cases where it is either a $b$-ball or $e$-ball we are of course done; therefore we can assume $B_{1}(q)$ is either a $c$-ball or $d$-ball. The handling of the two cases is almost verbatim; therefore we will assume $B_{1}(q)$ is a $c$-ball. In this case, we can apply Proposition \ref{weak c-ball} in order to build the covering
\[
B_{1}\subseteq   \mathcal{N}\cup \bigcup_{b}B_{r_{b}}(x_{b})\cup \bigcup_{c}B_{r_{c}}(x_{c})\cup \bigcup_{d}B_{r_{d}}(x_{d})\cup \bigcup_{e}B_{r_{e}}(x_{e}),
\]
where $\mathcal{N}\subseteq B_{11}(\tilde q)$ is a $(\delta ,\tau)$-neck for some $\tilde q\in B_{4}(q)$ and we have the content estimates
\[
\sum_{b}r_{b}^{n-4}+\sum_{d}r_{d}^{n-4}+\sum_{e}r_{e}^{n-4}\leq C(n,\tau),\qquad
\sum_{c}r_{c}^{n-4}\leq C(n,\tau)\epsilon.
\]
What we have gained from the above is that the remaining $c$-balls have small $n-4$-content. We will need to remove both the $c$-balls and $d$-balls in this decomposition before the proof is complete. Let us first deal with the $d$-balls. Indeed, if we apply Proposition \ref{d-ball} to each $d$-ball $B_{r_d}(x_d)$, then we obtain the covering
\[
B_{1}(p)\subseteq \tilde{S}\cup   \mathcal{N}\cup \bigcup_{b}B_{r_{b}}(x_{b})\cup \bigcup_{c}B_{r_{c}}(x_{c})\cup \bigcup_{e}B_{r_{e}}(x_{e}),
\]
where $\tilde{S} = \bigcup_{d}\tilde{S}_{d}$ is a countable union of $n-4$-measure zero sets, and thus itself has $n-4$ measure zero. We also now have the estimates
\[
\sum_{c}r_{c}^{n - 4}\leq C(n,\tau)\cdot \epsilon,\qquad
\sum_{b}r_{b}^{n - 4} + \sum_{e}r_{e}^{n - 4}\leq C(n,\tau).
\]
Let us remark that our only original assumption to construct this covering was that $B_{1}(q)$ was a $c$-ball. Therefore, we can repeat this construction on each $c$-ball $B_{r_c}(x_c)$ in order to build a covering
\[
B_{1}(q)\subseteq \tilde{S}\cup \bigcup_{a}\big( \mathcal{N}_{a}\cap B_{r_{a}}\big)\cup \bigcup_{b}B_{r_{b}}(x_{b})\cup \bigcup_{c}B_{r_{c}}(x_{c})\cup \bigcup_{e}B_{r_{e}}(x_{e}),
\]
with the estimates
\[
\sum_{a}r_{a}^{n - 4}\leq 1 + C(n,\tau)\epsilon,\qquad
\sum_{c}r_{c}^{n - 4}\leq \big(C(n,\tau)\cdot \epsilon \big)^{2},
\]
\[
\sum_{b}r_{b}^{n - 4} + \sum_{e}r_{e}^{n - 4}\leq C(n,\tau)\big(1 + C(n,\tau)\epsilon \big).
\]
If we continue to recover the $c$-balls, then after $i$ iterations we have the covering
\[
B_{1}(q)\subseteq \tilde{S}\cup \bigcup_{a}\big( \mathcal{N}_{a}\cap B_{r_{a}}\big)\cup \bigcup_{b}B_{r_{b}}(x_{b})\cup \bigcup_{c}B_{r_{c}}(x_{c})\cup \bigcup_{e}B_{r_{e}}(x_{e}),
\]
with the estimates
\[
\sum_{a}r_{a}^{n - 4}\leq \sum_{j = 0}^{i}\big(C(n,\tau)\epsilon\big)^{j},\qquad
\sum_{c}r_{c}^{n - 4}\leq \big(C(n,\tau)\cdot \epsilon\big)^{i},
\]
\[
\sum_{b}r_{b}^{n - 4} + \sum_{e}r_{e}^{n - 4}\leq C(n,\tau)\sum_{j = 0}^{i}\big(C(n,\tau)\epsilon\big)^{j}.
\]
Now we may consider the discrete sets $\tilde{S}_{c}^{i}\bigcup_{c}\{x_{c}^{i}\}$ and their Hausdorff limit $\tilde{S}_{c} = \lim \tilde{S}_{c}^{i}$. Arguing as in Proposition \ref{d-ball} we see for $\epsilon \leq \epsilon (n,\tau)$ that $\cH^{n - 4}(\tilde{S}_{c}) = 0$. Combining this with our previous $\tilde{S}$ set we then arrive at the covering
\[
B_{1}(q)\subseteq \tilde{S}\cup \bigcup_{a}(  \mathcal{N}_{a}\cap B_{r_{a}})\cup \bigcup_{b}B_{r_{b}}(x_{b})\cup \bigcup_{e}B_{r_{e}}(x_{e}),
\]
which for $\epsilon \leq \epsilon(n,\tau)$ satisfies
\[
\cH^{n - 4}(\tilde{S}) = 0,\qquad
\sum_{a}r_{a}^{n - 4}\leq 2,\qquad
\sum_{b}r_{b}^{n - 4} + \sum_{e}r_{e}^{n - 4}\leq C(n,\tau),
\]
which completes the proof of the lemma.
\end{proof}

\begin{proposition}[Inductive covering]
Let \((M^n, g)\) satisfy $(A1),(A2),(A3)$. Let $ V =\inf_{B_4(q)}\cV_{1}(y)\ge \rv$ and fix $0< \delta$ and $0< \tau \leq \tau (n)$. Then there exists $v_{0}(n,\rv ,\delta ,\tau) > 0$ such that we can write
\[
B_{1}(q)\subseteq \tilde{S}\cup \bigcup_{a}\left(  \mathcal{N}_{a}\cap B_{r_{a}}\right)\cup \bigcup_{b}B_{r_{b}}(x_{b})\cup \bigcup_{v}B_{r_{v}}(x_{v}),
\]
where
\begin{itemize}
	\item[(a)] $\mathcal{N}_{a}\subseteq B_{2r_{a}}(x_{a})$ is a $(\delta ,\tau)$-neck region;
	\item[(b)] for each $b$-ball $B_{r_{b}}(x_{b})$, we have $r_{h}(x_{b}) > 2r_{b}$;
	\item[(v)] for each $v$-ball $B_{r_{v}}(x_{v})$, if $V_{v}\equiv \inf_{y\in B_{4r_{v}}(x_{v})}\mathcal{V}_{r_{v}}$, then $V_{v}\geq V + v_{0}$;
	\item[(s)]$\cH^{n - 4}(\tilde{S}) = 0$.
\end{itemize}
Further, we have the content bound $\sum_{a}r_{a}^{n - 4} + \sum_{b}r_{b}^{n - 4} + \sum_{v}r_{v}^{n - 4}< C(n,\tau ,\delta)$.
\end{proposition}

\begin{proof}
This proposition is very similar to the previous lemma; we essentially need only one additional covering in order to deal with the fact that $V$ here is coming from the volume pinching at scale one, instead of scale $\eta^{-1}$. Let $\eta$ be chosen as in Lemma \ref{d-ball}, and let us cover $B_{1}(q)$ by the Vitali covering
\[
B_{1}(q)\subseteq \bigcup_{f}B_{\eta^{2}}(x_{f}),
\]
where $\{B_{\eta^{2} / 10}(x_{f})\}$ are disjoint. By a standard covering argument we have that there are at most $C(n,\tau ,\delta)$ balls in this collection. For each ball $B_{\eta^{2}}(x_{f})$ let us apply Lemma \ref{induction}  to find the covering given by
\[
B_{1}(q)\subseteq \bigcup_{f}B_{\eta^{2}}(x_{f})\subseteq \tilde{S}\cup \bigcup_{a}\big( \mathcal{N}_{a} \cap B_{r_{a}}\big)\cup \bigcup_{b}B_{r_{b}}(x_{b})\cup \bigcup_{e}B_{r_{e}}(x_{e}),
\]
such that we have the content bounds $\sum_{a}r_{a}^{n - 4} + \sum_{b}r_{b}^{n - 4} + \sum_{e}r_{e}^{n - 4}< C(n,\tau ,\delta)$ and such that $r_{e}\leq \eta^{4}$. We will see that these $e$-balls will be the $v$-balls of the proposition with $v_{0}$ defined appropriately. To see this, let us be careful and note that the $e$-balls in this covering are with respect to one of the balls $B_{\eta^{2}}(x_{f})$. Therefore for each $e$-ball, by smallness of volume pinching at small scales (Theorem \ref{summability of volume drop}), we have the estimate
\[
\inf_{B_{4r_{e}}(x_{e})}\mathcal{V}_{\eta r_{e}}(y)\geq \inf_{B_{4\eta^{2}}(x_{f})}\mathcal{V}_{\eta}(y) + \eta \geq \inf_{B_{(1 + 4\eta^{2})}(q)}\mathcal{V}_{1}(y) + \frac{3}{4}\eta .
\]
In order to finish the proof, let us notice by volume comparison theorem \ref{Volume comparison}, 
\[
\inf_{B_{(1 + 4\eta^{2})}(q)}\mathcal{V}_{1}(y)\geq  (1-\Psi(\eta|n))\inf_{B_{1}(q)}\mathcal{V}_{1}(y) >  V - \frac{1}{4}\eta ,
\]
 for some $\eta < \eta (n)$, where $\Psi(\eta|n)\to 0$ as $\eta\to 0$. Combining this with the previous inequality we arrive at the estimate
\[
\inf_{B_{4r_{e}}(x_{e})}\mathcal{V}_{\eta r_{e}}(y)\geq \inf_{B_{(1 + 4\eta^{2})}(p)}\mathcal{V}_{1}(y) + \frac{3}{4}\eta \geq V + \frac{1}{2}\eta \equiv V + v_{0}.
\]
Choosing a Vitali covering $\{B_{r_{v}}(x_{v}),x_{v}\in B_{r_{e}}(x_{e})\}$ with $r_{v} = \eta r_{e}$ of $B_{r_{e}}(x_{e})$ we thus finish the proof of the proposition.
\end{proof}
 
\begin{proof}[Proof of Theorem \ref{neck decomposition}]
Let $M$ satisfy the assumptions of the theorem, and let $V\equiv \inf_{B_{4}(q)}\mathcal{V}_{1}(y) > 0$. Let us then fix $v_{0}(n,V,\delta ,\tau)$ as in Proposition \ref{induction}. Then if we apply Proposition \ref{induction}, we arrive at the covering
\[
B_{1}\subseteq \tilde{S}^{1}\cup \bigcup_{a}\big( \mathcal{N}_{a}^{1} \cap B_{r_{a}^{1}}\big)\cup \bigcup_{b}B_{r_{b}^{1}}(x_{b}^{1})\cup \bigcup_{v}B_{r_{v}^{1}}(x_{v}^{1}),
\]
where
\begin{itemize}
	\item $\mathcal{N}_{a}^{1}\subseteq B_{2r_{a}^{1}}(x_{a}^{1})$ is a $(\delta,\tau)$-neck;
	\item $r_{h}(x_{b}^{1}) > 2r_{b}^{1}$;
	\item for each $v$-ball $B_{r_{v}^{1}}(x_{v}^{1})$, if $V_{v}^{1}\equiv \inf_{y\in B_{4r_{v}^{1}}(x_{v}^{1})}\mathcal{V}_{r_{v}^{1}}(y)$, then $V_{v}^{1}\geq V + v_{0}$;
	\item $\cH^{n-4}(\tilde{S}^{1}) = 0$;
	\item we have the content bound $\sum_{a}(r_{a}^{1})^{n-4} + \sum_{b}(r_{b}^{1})^{n-4} + \sum_{v}(r_{v}^{1})^{n-4} < C(n,\tau,\delta)$.
\end{itemize}

Now we may simply apply Proposition \ref{induction} to each $v$-ball $B_{r_{v}^{1}}(x_{v}^{1})$. Since the volume lower bound on these balls has increased by $v_{0}$, after at most $K = K(n,V,v_{0})$ such iterations this volume bound must exceed the maximum possible volume lower bound, and therefore the process must stop (see also Theorem \ref{summability of volume drop}). More precisely, let us assume we have applied Proposition \ref{induction} $i$ times to obtain the covering
\[
B_{1}\subseteq \tilde{S}^{i}\cup \bigcup_{a}\big(   \mathcal{N}_{a}^{i} \cap B_{r_{a}^{i}}\big)\cup \bigcup_{b}B_{r_{b}^{i}}(x_{b}^{i})\cup \bigcup_{v}B_{r_{v}^{i}}(x_{v}^{i}),
\]
with the property that for every $v$-ball $B_{r_{v}^{i}}(x_{v}^{i})$ we have
\[
\inf_{B_{4r_{v}^{i}}(x_{v}^{i})}\mathcal{V}_{r_{v}^{i}}(y) \geq V + i v_{0}.
\]
Since the volume is uniformly bounded above by Theorem \ref{Volume comparison}, there exists $I = I(n,V,v_{0})$ such that $I v_{0} > \max_{y}\mathcal{V}_{1}(y)$, and therefore we must have that there are no $v$-balls left after $I$ iterations and $\tilde{S}^{I}=\varnothing$. Thus we can write
\[
B_{1}\subseteq  \bigcup_{a}\big(  \mathcal{N}_{a}^{I} \cap B_{r_{a}^{I}}\big)\cup \bigcup_{b}B_{r_{b}^{I}}(x_{b}^{I}).
\]

Let us now define   $\mathcal{N}_{a} \equiv \mathcal{N}_{a}^{I}$, $r_{a}\equiv r_{a}^{I}$, and $B_{r_{b}}(x_{b})\equiv B_{r_{b}^{I}}(x_{b}^{I})$,
this completes the proof of Theorem \ref{neck decomposition}.
\end{proof} 
 
\section{Proof of $L^2$-estimates} 
\begin{proof}[Proof of Theorem~\ref{thm:main}]
Fix $\varepsilon>0$ and $\tau=\tau_0(n)$.  
Choose $\delta$ according to Theorem~\ref{neck decomposition} and also small enough so that Theorem~\ref{strong L2 neck} applies.  
Apply the neck decomposition to  ball $B_1(x)$,   
we have
\[
\int_{B_1(x)} |\Rm|^2 \le \sum_a \int_{\N_a\cap B_{r_a}(x_a)} |\Rm|^2 + \sum_b \int_{B_{r_b}(x_b)} |\Rm|^2 .
\]
 By Theorem~\ref{strong L2 neck}, we have scale invariant $L^2$ estimate on each $\N_a\cap B_{r_a}(x_a)$:
\[
\int_{\N_a\cap B_{r_a}(x_a)} |\Rm|^2 \le C \,r_a^{n-4}.
\]
 And on each $B_{r_b}(x_b)$ we have $r_h(x_b)>2r_b$, since $p>\frac{n}{2}\ge 2$, by $W^{2,p}$ estimate,
\[
\int_{B_{r_b}(x_b)} |\Rm|^2 \le C\, r_b^{n-4}.
\]
 Using the content bound from Theorem~\ref{neck decomposition},
\[
\int_{B_1(x)} |\Rm|^2 \le C\, \left(\sum_a r_a^{n-4}+ \sum_b r_b^{n-4}\right) \le C(n,\mathrm{v},p,\Lambda).
\]
This completes the proof.
\end{proof}
\appendix
\section{$L^4$ Ricci curvature estimate on neck regions} In  this section we prove $L^4$ estimate for Ricci curvature.
\begin{theorem}\label{L4 ricci}
Let $(M^n,g)$  satisfy assumptions $(A1), (A2), (A3)$ and $(A4)$.   If $\delta\le\delta_0(n,p,\rv,\Lambda,B)$, then there exists $C(n,p,\rv,\Lambda,B)>0$ such that
\[
r^{4-n} \int_{\mathcal{N}_{10^{-5}} \cap B_r(x)} |\Ric|^4(z) \, dz \leq C.
\]
\end{theorem}  

 In fact, in this case we choose $
f = \phi_\cN\,\left(|\Ric|+\sigma\right)^{2-2\eta}$,
 $F(r)=r^{-3}\int_{b=r} f\,|\nabla b|,$ and 
$H(r)=r^{2(1-\eta)}F(r)$ in Lemma \ref{superconvexity}. 

\begin{lemma}\label{ricci superconvexity}
Let \((M^n, g)\) satisfy assumptions  $(A1),(A2),(A3)$ and $(A4)$.  
If $\delta\le\delta_0(n,p, \rv,\Lambda,B),$ then there exists $\sigma=\sigma(n)>0$ and $C=C(n,p, \rv,\Lambda,B)>0$ such that for $H(r)=r^{-1-2\eta}\int_{b=r} \phi_\cN\,\left(|\Ric| +\sigma\right)^{2-2\eta}|\nabla b|$, we have the following evolution equation:
\[
r^2 H''+(4\eta-1)rH'-3\eta(1-\eta)H
= r^{1-2\eta}\int_{b=r}\Delta f\,|\nabla b|^{-1}\ge -e(r),
\]
where \begin{equation}\label{ricci error}
\begin{aligned}e(r) =r^{1-2\eta}\int_{b=r} \left((|\mathrm{Ric}| + \sigma)^{2-2\eta}|\Delta\phi_\cN| +2|\nabla\phi_\cN||\nabla(|\mathrm{Ric}| + \sigma)^{2-2\eta}|+(2+2\eta)(|\Ric|+\sigma) ^{1-2\eta}|\nabla^2 \R|\right)|\nabla b|^{-1}
\end{aligned}
\end{equation}
\end{lemma}
\begin{proof}
By applying Lemma \ref{evolution} to $F(r)$, the fundamental differential equation for $H(r)$ is given by:
\begin{equation} \label{eq: H_basic}
    r^2 H''(r) + (4\eta - 1)r H'(r) - 4\eta(1-\eta)H(r) = r^{1-2\eta} \int_{b=r} \Delta f |\nabla b|^{-1}.
\end{equation}Let $G = (|\mathrm{Ric}| + \sigma)^{2-2\eta}$. The Laplacian of $f$ expands to:
\begin{equation}
    \Delta f = G\Delta\phi_{\mathcal{N}} + 2\langle\nabla\phi_{\mathcal{N}}, \nabla G\rangle + \phi_{\mathcal{N}} \Delta G.
\end{equation}
To compute $\Delta G$, we first find the Laplacian of $|\mathrm{Ric}|$. By the Bochner-Weitzenb\"ock formula for the Ricci tensor, we have:
\begin{equation}
    \frac{1}{2} \Delta |\mathrm{Ric}|^2 = |\nabla \mathrm{Ric}|^2 + \langle \nabla^2 \R, \mathrm{Ric} \rangle + \mathrm{Rm} * \mathrm{Ric} * \mathrm{Ric}.
\end{equation}
Using the identity $\Delta |\mathrm{Ric}| = \frac{\Delta |\mathrm{Ric}|^2}{2|\mathrm{Ric}|} - \frac{|\nabla |\mathrm{Ric}||^2}{|\mathrm{Ric}|}$, we can estimate:
\begin{equation}
    \Delta |\mathrm{Ric}| \ge \frac{|\nabla \mathrm{Ric}|^2 - |\nabla |\mathrm{Ric}||^2}{|\mathrm{Ric}|} - |\nabla^2 \R| - C(n)|\mathrm{Rm}||\mathrm{Ric}|\ge - |\nabla^2 \R| - C(n)|\mathrm{Rm}||\mathrm{Ric}|.
\end{equation}
Applying this to $G = (|\mathrm{Ric}| + \sigma)^{2-2\eta}$, a direct calculation yields:
\begin{align}
    \Delta G &= (2-2\eta)(|\mathrm{Ric}|+\sigma)^{1-2\eta}\Delta|\mathrm{Ric}| +(2-2\eta)(1-2\eta)(|\mathrm{Ric}|+\sigma)^{-2\eta}|\nabla|\mathrm{Ric}||^2 \nonumber \\
    &\ge -(2-2\eta)(|\mathrm{Ric}|+\sigma)^{1-2\eta}|\nabla^2 \R|-C(n)(|\mathrm{Ric}|+\sigma)^{2-2\eta}|\Rm| 
\end{align}
Plugging this back into $\Delta f$ gives:
\begin{equation} \label{eq: delta_f_bound}
    \Delta f \ge G\Delta\phi_{\mathcal{N}} + 2\langle\nabla\phi_{\mathcal{N}}, \nabla G\rangle - C(n)\phi_{\mathcal{N}} G |\mathrm{Rm}| -  (2-2\eta)(|\mathrm{Ric}|+\sigma)^{1-2\eta}|\nabla^2 \R|.
\end{equation}
Finally, we conclude
\begin{align}
r^{1-2\eta}\int_{b=r}\Delta f\,|\nabla b|^{-1} 
&\ge -r^{1-2\eta}\int_{b=r} \left(G|\Delta\phi_\cN| +2|\nabla\phi_\cN||\nabla G|+(2+2\eta)(|\mathrm{Ric}|+\sigma)^{1-2\eta}|\nabla^2 \R|\right) |\nabla b|^{-1}\\
&\quad- C(n)r^{1-2\eta}\int_{b=r}\phi_\cN G |\Rm||\nabla b|,
\end{align} where we used $C^{-1}\le|\nabla b|\le C$ in the last line by Lemma \ref{Green's distance function estimate}

Since $\phi_\cN$ is supported in $(\delta,\tau)$-neck region, which is $\Psi(\delta)$ GH-close to $(n-3)$-symmetric cone, where $\Psi(\delta)\to 0$ as $\delta\to 0$. By Strong epsilon regularity assumption (A3), if $\delta<\delta_0$ is sufficently small relative to $\eta$, we have curvature estimate 
 \[   r^2 \phi_\cN|\Rm| \le C(n)^{-1}\eta(1-\eta) \phi_\cN \text{ on $\{b=r\}$}.
 \] Then we conclude \[C(n)r^{1-2\eta}\phi_\cN G |\Rm|\le \eta(1-\eta) r^{-1-2\eta}  \phi_\cN G \text{ on $\{b=r\}$}.
 \] 
Thus,
\begin{equation} 
r^2 H''+(4\eta-1)rH'-3\eta(1-\eta)H
= r^{1-2\eta}\int_{b=r}\Delta f\,|\nabla b|^{-1}\ge -e(r),
\end{equation} where \begin{equation} 
e(r) =r^{1-2\eta}\int_{b=r} \left((|\mathrm{Ric}| + \sigma)^{2-2\eta}|\Delta\phi_\cN| +2|\nabla\phi_\cN||\nabla(|\mathrm{Ric}| + \sigma)^{2-2\eta}|+(2+2\eta)(|\Ric|+\sigma) ^{1-2\eta}|\nabla^2 \R|\right)|\nabla b|^{-1}.\end{equation} 
\end{proof} 
\begin{theorem}\label{ricci Dini estimates}
 Let \((M^n, g)\) satisfy assumptions  $(A1),(A2),(A3)$ and $(A4)$.  If $\delta<\delta_0(n,p, \rv,\Lambda,B)$, there exists $C(n,p, \rv,\Lambda,B)>0$ such that for \[H(r)=r^{-1-2\eta}\int_{b=r} \phi_\cN\,\left(|\Rm|+\sigma\right)^{2-2\eta}|\nabla b|,\] we have the Dini estimate:
\[
\int_0^\infty \frac{H(r)}{r} \, dr \le C.
\]
\end{theorem}\begin{proof}
By the proof of Theorem \ref{Dini estimates},
\begin{align*}
 \int_0^\infty \frac{H(r)}{r} dr& \le  \frac{1}{3\eta(1-\eta)}\int_0^R\frac{e(r)}{r}dr,
\end{align*}
it suffices to estimate the right hand side.
From  expression $(\ref{ricci error})$ and coarea formula,  
we can write it as \[\int_0^R \frac{e(r)}{r}dr= e_1+e_2+e_3,\] where 
\[\begin{aligned}
e_1&=\int_0^R r^{-2\eta} \int_{b=r} \left(|\Ric| +\sigma\right)^{2-2\eta}|\Delta \phi_\cN| |\nabla b|^{-1}
     = \int_{b \leq R} b^{-2\eta} \left(|\Ric| +\sigma\right)^{2-2\eta}|\Delta \phi_\cN|,\\
e_2&=2\int_0^R r^{-2\eta} \int_{b=r} |\nabla\phi_\cN||\nabla\left(|\Ric| +\sigma\right)^{2-2\eta}||\nabla b|^{-1}
     = 2\int_{b \leq R} b^{-2\eta} |\nabla\phi_\cN||\nabla\left(|\Ric| +\sigma\right)^{2-2\eta}|,\\
     e_3 &= (2+2\eta)\int_0^R r^{-2\eta}\int_{b=r}\phi_\cN  (|\Ric|+\sigma) ^{1-2\eta}|\nabla^2 \R|  |\nabla b|^{-1} 
      = (2+2\eta)\int_{b \leq R} b^{-2\eta}\phi_\cN  (|\Ric|+\sigma) ^{1-2\eta}|\nabla^2 \R| . 
\end{aligned}\]

We now estimate each term as follows.
To begin with, since $\phi_\cN$ is supported in $\cN_{10^{-4}}=B_2(q)\backslash \bigcup_{x\in\C}\overline{B}_{10^{-4}\cdot r_x}(x)$, by Remark \ref{remark_neck} and Lemma \ref{Green's distance function estimate}, we have $r_h(y)\approx d(y,\C)\approx b(y)$. If follows from assumption  (A3)  that if $\delta<\delta_0$,  we have curvature derivative estimates: 
 \[\begin{aligned}|\nabla \left(|\Ric| +\sigma\right)^{2-2\eta}| \le(2-2\eta)(|\Ric|&+\sigma)^{1-2\eta}|\nabla \Ric| \le 2b^{-1}(|\Ric|+\sigma)^{2-2\eta},\\ 
 |\nabla^2 \R|\le \sqrt{n} |\nabla^2 \Ric|&\le C(n)b^{-2}|\Ric|  .\end{aligned}\]
Now we specify $\sigma\le  CR^{-1}$, we have \begin{equation}\label{ricci} |\Ric| +\sigma  \le Cb^{-1}  \text{ on $\{b\le R\}$ }. \end{equation}Combing above estimates, we have
 \[\begin{aligned}
 e_1&\leq \int_{b \leq R} b^{-2\eta} (Cb^{-1})^{2-2\eta}|\Delta \phi_\cN|\le C\int_{b \leq R} b^{-2} |\Delta \phi_\cN| ,\\
e_2&\leq C\int_{b \leq R} b^{-2\eta} |\nabla\phi_\cN|\left(b^{-1}(|\Ric|+\sigma)^{2-2\eta}\right)\leq  C\int_{b \leq R} b^{-2\eta} |\nabla\phi_N|\left(b^{-1}b^{2\eta-2}\right)=C\int_{b \leq R} b^{-3} |\nabla\phi_\cN| ,\\ 
e_3&\le C\int_{b \leq R} b^{-2\eta}\phi_\cN b^{-2}(|\Ric|+\sigma)^{-2\eta}  \le C\sigma^{-2\eta}\int_{\{b\le R\}}  b^{-2-2\eta} \phi_\cN.
\end{aligned}\]
Combining all these we arrive at
\begin{equation}\label{ricci integral}
\int_0^\infty \frac{H(r)}{r} \, dr \leq C \left(  \int_{b \leq R} b^{-2-2\eta}\phi_\cN+ \int_{b \leq R}b^{-3}|\nabla\phi_\cN| + \int_{b \leq R} b^{-2}|\Delta\phi_\cN| \right).  
\end{equation}Following the same argument in Theorem \ref{Dini estimates}, we finish the proof.
\end{proof}
\begin{proof}[Proof of Theorem \ref{L4 ricci}]
By Theorem \ref{ricci Dini estimates} and coarea formula,
\[   \int_{0}^\infty r^{-2-2\eta}\int_{b=r}\phi_\cN \left(|\Ric|+\sigma\right)^{2-2\eta}|\nabla b|=\int_0^\infty\frac{H(r)}{r}dr\le C.\]
Then  
\[\int_{\cN_{10^{-5}}\cap B_1(x)}|\Ric|^4 \le  C\int_M  b^{-2-2\eta}\phi_\cN \left(|\Ric|+\sigma\right)^{2-2\eta}.
\]    
By coarea formula, Lemma (\ref{Green's distance function estimate}) and $(\ref{proof_dini})$ we have 
\begin{equation}
\begin{aligned}
 \int_M b^{-2-2\eta}\phi_\cN \left(|\Ric|+\sigma\right)^{2-2\eta}&=
     \int_0^\infty r^{-2-2\eta}\int_{b=r}\phi_\cN  \left(|\Ric|+\sigma\right)^{2-2\eta} |\nabla b|^{-1}\\
     &\le C\int_{0}^\infty r^{-2-2\eta}\int_{b=r}\phi_\cN \left(|\Ric|+\sigma\right)^{2-2\eta}|\nabla b|\\ & \le C.
\end{aligned}
\end{equation}
This proves \[\int_{\cN_{10^{-5}}\cap B_1(x)}|\Ric|^4\le C.
\] 
\end{proof}
 
 \bibliographystyle{plain}

\end{document}